\documentclass[a4paper,11pt,parskip=half*,numbers=noendperiod]{article}
\usepackage[margin=3.5cm]{geometry}

\usepackage[utf8]{inputenc}
\usepackage[T1]{fontenc}
\usepackage[english]{babel}
\usepackage{lmodern}
\usepackage{xparse}

\usepackage{lipsum}
\usepackage[shortlabels]{enumitem}
\usepackage{verbatim}
\usepackage[normalem]{ulem}
\usepackage[dvipsnames]{xcolor}
\usepackage{csquotes}
\usepackage{pdfpages}
\usepackage{easy-todo}
\usepackage{setspace}

\usepackage{graphicx}
\usepackage{caption}
\usepackage{subcaption}
\usepackage{tabularx}
\usepackage{multicol}
\usepackage{amsmath}
\usepackage{mathtools}
\usepackage{stmaryrd}
\usepackage{amsthm, thmtools}
\usepackage{framed}
\usepackage{nameref}
\usepackage[colorlinks,menucolor=blue,linkcolor=blue, citecolor=blue, urlcolor=blue]{hyperref} 
\usepackage[capitalise]{cleveref}
\crefname{subsection}{Subsection}{Subsections}
\Crefname{subsection}{Subsection}{Subsections}

\usepackage{amssymb}
\usepackage[mathscr]{euscript}
\usepackage{esint}
\usepackage{esvect}
\usepackage{relsize}

\usepackage{array,tabularx,booktabs}

\usepackage{tikz}
\usepackage{tikz-qtree}
\usepackage{tikz-cd}
\usetikzlibrary{calc}

\usetikzlibrary{fit}
\usepgfmodule{nonlineartransformations}
\usetikzlibrary{curvilinear}
\usepackage[framemethod=tikz]{mdframed}
\usetikzlibrary{cd}
\usetikzlibrary{matrix}
\usetikzlibrary{backgrounds}

\newtheorem{theorem*}{Theorem}
\newtheorem{theorem}{Theorem}[section]
	
\newtheorem{corollary}[theorem]{Corollary}	
\newtheorem{lemma}[theorem]{Lemma}
	
\newtheorem{definition}[theorem]{Definition}
\theoremstyle{definition}

\newtheorem{remark}[theorem]{Remark}

\newcommand{\I}{\text{$\mathrm{I}$}}
\newcommand{\II}{\text{$\mathrm{II}$}}

\newcommand{\Ps}{\mathcal{P}}
\newcommand{\Ord}{\mathsf{Ord}}
\newcommand{\M}{\text{$\mathcal{M}$}}
\newcommand{\N}{\text{$\mathcal{N}$}}

\newcommand{\A}{\text{$\mathcal{A}$}}
\newcommand{\KP}{\text{$\mathsf{KP}$}}
\newcommand{\CA}{\text{--}\mathsf{CA}_0}
\newcommand{\DC}{\mathsf{DC}}
\newcommand{\wo}{\mathsf{wo}}
\newcommand{\RCA}{\mathsf{RCA}_0}
\newcommand{\ACA}{\mathsf{ACA}_0}
\newcommand{\AC}{\mathsf{AC}}
\newcommand{\Det}{\mathsf{Det}}
\newcommand{\Rfc}{\text{--}\mathsf{Ref}}
\newcommand{\ZF}{\text{$\mathsf{ZF}$}}

\newcommand{\Z}{\text{$\mathsf{Z}$}}
\newcommand{\QSI}{\mathrm{S}^-_{\I}}
\newcommand{\QSII}{\mathrm{S}^-_{\II}}
\newcommand{\SI}{\mathrm{S}_{\I}}

\newcommand{\C}{\mathcal{C}}

\usepackage[abbrev,backrefs]{amsrefs}

\newcommand\blfootnote[1]{%
  \begingroup
  \renewcommand\thefootnote{}\footnote{#1}%
  \addtocounter{footnote}{-1}%
  \endgroup
}
\usepackage{soul}
\usepackage{authblk}

\begin{document}
\title{The Limits of Determinacy in Higher-Order Arithmetic}
\author[$^1$]{J. P. Aguilera}
\author[$^{2}$]{T. Kouptchinsky}
\affil[$^{1,2}$]{Institute of Discrete Mathematics and Geometry, Vienna University of Technology. Wiedner Hauptstrasse 8-10, 1040 Vienna, Austria}	
\affil[$^{2}$]{Corresponding author: \texttt{thibaut.kouptchinsky@tuwien.ac.at}}
\maketitle
\begin{abstract}
We prove level-by-level upper and lower bounds on the strength of determinacy for finite differences of sets in the hyperarithmetical hierarchy in terms of subsystems of finite- and transfinite-order arithmetic, extending the Montalb\'an-Shore theorem to each of the levels of the Borel hierarchy beyond $\Pi^0_3$. 
We also prove equivalences between reflection principles for higher-order arithmetic and quantified determinacy axioms, answering two questions of Pacheco and Yokoyama.
\blfootnote{MSC (2020): 03B30 (Primary), 03D60, 03E60, 03F35.}  \blfootnote{Keywords: Reverse mathematics, determinacy, admissible set, Kripke-Platek set theory, comprehension axiom.}
\end{abstract}
\setcounter{tocdepth}{1}
\tableofcontents

\section{Introduction}

Every Gale-Stewart game with Borel payoff is determined according to Martin's celebrated Borel Determinacy theorem \cite{BorelDetMartin}. We play these games as follows: two players, $\I$ and $\II$, alternate infinitely many turns playing digits $x_0$, $x_1$, $x_2$, and so on. After infinitely many turns, an infinite sequence $\langle x_0, x_1, x_2, \hdots \rangle \in \mathbb{N}^{\mathbb{N}}$ is produced. A set $A\subset\mathbb{N}^\mathbb{N}$ fixed in advance serves as the set of rules for the game. Given a play $x \in \mathbb{N}^{\mathbb{N}}$, we say that Player $\I$ wins if and only if $x \in A$, and otherwise $\II$ wins the game. According to Martin's theorem, the game is determined whenever $A$ is Borel.

An interesting aspect of Martin's theorem is that its proof requires the necessary use of all the axioms of $\mathsf{ZFC}$. This includes \textsc{replacement} and \textsc{power set}. Indeed, in what is perhaps the earliest result in Reverse Mathematics, Friedman \cite{Sigma5Det} proved that $\Sigma^0_5$--determinacy is not provable in Second-Order Arithmetic ($\Z_2$). Martin (unpublished) later improved this result to show that  $\Sigma^0_4$--determinacy is not provable in $\Z_2$. This is best as possible in the sense that $\Z_2$ does prove $\Sigma^0_3$--determinacy (see Welch \cite{Pi03CA>>Pi03Det}). The work of Montalb\'an and Shore \cite{MS} later revealed the limits of determinacy provable in $\Z_2$. Below, we denote by $(\Sigma^0_3)_n$ the class of sets obtained via $n$ nested differences of $\Sigma^0_3$ sets (see \S \ref{SectPreliminaries} for a precise definition).
\begin{theorem}[Montalb\'an and Shore]
Let $n\in\mathbb{N}$. Then, 
\begin{enumerate}
\item $\Pi^1_{n+2}\text{--}\mathsf{CA}_0 \vdash (\Sigma^0_3)_n\text{--}\Det$; but 
\item $\Delta^1_{n+2}\text{--}\mathsf{CA}_0\not\vdash (\Sigma^0_3)_n\text{--}\Det$.
\end{enumerate}
\end{theorem}

In addition to the theorem of Montalb\'an and Shore, there has been a great deal of work carried out on the reverse mathematical properties of determinacy principles in second-order arithmetic; we refer the reader to \cites{AgMSD,AgWe} for recent work, and to the introductions of these articles for an overview of the subject. Much work has also been carried out on determinacy principles that transcend $\mathsf{ZFC}$ in strength; see e.g., Larson \cite{History} for an overview. However, not much work has been done on the reverse mathematics of determinacy axioms in higher-order arithmetic, or in subsystems of $\mathsf{ZFC}$ beyond $\Z_2$. Hachtman \cite{Z3Det} and Schweber \cite{Schw15} have considered the strength of determinacy principles for games on reals, and
Hachtman \cite{Hachtman} has obtained a characterization of $\Sigma^0_\alpha$--determinacy in terms of a reflection principle for admissible sets. Hachtman \cite{Hachtman} also mentions the bounds on determinacy principles provable in $\mathsf{ZFC}^-$, attributing them to Martin, Friedman, Montalb\'an, and Shore. These coarse bounds can indeed be proved by modifications of their arguments. However, we shall see in this article that some unexpected subtleties arise when considering the level-by-level analysis. To explain these, let us restate the Montalb\'an and Shore theorem in the language of set theory.

\begin{theorem}[Montalb\'an and Shore]
Let $n\in\mathbb{N}$. Then,  \label{TheoremMSIntro}
\begin{enumerate}
\item $\mathsf{KP} + \Sigma_{n+1}${--}\textsc{Separation} $\vdash (\Sigma^0_3)_n \text{--}\Det$; but 
\item $\mathsf{KP} + \Sigma_{n+1}$--\textsc{Collection} $ + \Delta_{n+1}{}$--\textsc{Separation} $\not\vdash(\Sigma^0_3)_n\text{--}\Det$.
\end{enumerate}
\end{theorem}

Thus, when stating the result in the language of set theory, the degree of \textsc{Separation} required is one index smaller than the level of \textsc{Comprehension} needed when stating the result in the language of $\Z_2$. This is very standard and is merely a consequence of the translation between arithmetic and set theory; see Simpson \cite{Simpson}*{VII.5}. The key point is the fact that the membership relation on sets is wellfounded, and coding hereditarily countable sets by sets of integers requires restricting to wellfounded sets; phrasing the result in the language of set theory eliminates this consideration and yields a setting that can more uniformly be extended to higher-order arithmetic. Henceforth, we have stuck to this formalism in order to state and prove our results.

Based on what we know about determinacy in subsystems of $\mathsf{ZFC}$ from work of Martin and Friedman, one would expect Theorem \ref{TheoremMSIntro} to hold \textit{mutatis mutandi} for the difference hierarchy over $\Sigma^0_{1 + \gamma + 2}$ by adding the axiom ``$\mathcal{P}^{\gamma}(\mathbb{N})$ exists,'' as well as the axiom ``$\gamma$ is wellordered'' to avoid proof-theoretic pathologies. However, what we shall see is that this is not quite true: the indices in the extension of Theorem \ref{TheoremMSIntro} must be lowered yet one more stage:

\begin{theorem}
Suppose $n\in\mathbb{N}$ and $1 \leq \gamma < \omega_1^{\mathsf{CK}}$. Let $\mathsf{KP}^\gamma$ be the theory $\mathsf{KP}$ + $``\mathcal{P}^{\gamma}(\mathbb{N})$ exists'' + $``\gamma$ is wellordered.'' Then,  \label{TheoremMainIntro}
\begin{enumerate}
\item $\mathsf{KP}^\gamma + \Sigma_{n}\text{--}$\textsc{Separation} $\vdash(\Sigma^0_{1 + \gamma + 2})_n\text{--}\Det$; but 
\item $\mathsf{KP}^\gamma +$  {$\Sigma_{n}$}{--}\textsc{Collection} $ + \Delta_{n}\text{--}$\textsc{Separation} $\not\vdash (\Sigma^0_{1 + \gamma + 2})_n\text{--}\Det$.
\end{enumerate}
\end{theorem}
According to Theorem \ref{TheoremMainIntro}, the strength of $(\Sigma^0_{1 + \gamma + 2})_n$--determinacy  is strictly lower than the picture from second-order arithmetic would suggest. We emphasize that this is not merely a consequence of syntax, coding, or language: there really is a difference in how much separation can be extracted from determinacy in higher-order arithmetic. 
Our unprovability result (the second half of Theorem \ref{TheoremMainIntro}) is presented in \S \ref{SectionLowerBound}. Our proof involves considering a Friedman-style game similar to the one in \cite{MS} which cannot have any simple winning strategy. There are some differences in our game necessary to incorporate the \textsc{power set} axioms in the statement of Theorem \ref{TheoremMainIntro}. Some of these differences are crucial, as
one runs into serious difficulties when attempting to modify existing proofs in order to exploit determinacy principles to produce models satisfying both \textsc{power set} and \textsc{separation} axioms. Thus, the proof of Theorem \ref{TheoremMainIntro} yields a weaker conclusion than Theorem \ref{TheoremMSIntro}.

In \S \ref{SectionUpperBound}, we see that these difficulties are unavoidable. One can take advantage of the interplay between games on integers and games on objects of higher cardinality (via Martin's \textit{unravelling} technique) to construct winning strategies in a way which has no parallel in the context of subsystems of $\Z_2$. This yields the first part of Theorem \ref{TheoremMainIntro}, whose conclusion is \textit{stronger} than that of Theorem \ref{TheoremMSIntro}.
A hint towards this possibility was present in the work of Hachtman \cite{Hachtman}, who proved that $\Sigma^0_4$--$\Det$ follows from $\Sigma_1$-separation over $\KP$ + ``$\mathbb{R}$ exists'', but not from merely $\KP$ + ``$\mathbb{R}$ exists''. This again is in contrast to the picture within $\Z_2$, where $\Sigma^0_3\text{--}\Det$ is not provable in $\KP + \Sigma_1\textsc{-Separation}$ by Welch \cite{Pi03CA>>Pi03Det}.
Theorem \ref{TheoremMainIntro} yield the picture depicted in Figure \ref{FigureProvIntro} below, which is the analogue of the picture from \cite{MSCons} in the context of higher-order arithmetic. Indeed, we prove in section~\ref{SubsectionBetaModels} the following.

\begin{theorem}
    Assume $\RCA$ and let $2 \leq n < \omega$ and $1\leq \gamma < \omega_1^{\mathsf{ CK }}$. Then, 
    \begin{align*}
        \KP^{\gamma} + \Sigma_{n}\text{--\textsc{separation}} \to \beta((\Sigma^0_{1+\gamma+2})_{n}\text{--}\Det) \to
        (\Sigma^0_{1 + \gamma+ 2})_{n}\text{--}\Det \to
        \beta(\KP^{\gamma}_{n}).
    \end{align*}
    Moreover, no implication is reversible.~\label{generalised_inequalities}
    \end{theorem}

    For comparison, we state the theorem for $\Z_2$ in set-theoretic terms. Note the differences between Theorem \ref{generalised_inequalities} and Theorem \ref{TheoremMSBeta} concerning the index $n$.
    
    \begin{theorem}[Montalb\'an and Shore \cite{MSCons}] \label{TheoremMSBeta}
    Assume $\RCA$. Then for $1 \leq n < \omega$,
    \begin{align*}
        \KP + \Sigma_{n+1}\text{--\textsc{separation}} \to \beta((\Sigma^0_{3})_{n}\text{--}\Det) \to
        (\Sigma^0_{3})_{n}\text{--}\Det \to
        \beta(\KP_{n+1}).
    \end{align*}
    Moreover, no implication is reversible.
    \end{theorem}
    Theorem~\ref{generalised_inequalities} leads also to the following.
    \begin{corollary}\label{CorollaryKPGammanimpliesDet}
        In $\RCA$, taking $2 \leq n < \omega$ and $1 \leq \gamma < \omega_{1}^{\mathsf{CK}}$, \begin{align*}
            \KP^{\gamma} + \Sigma_{n}\text{--\textsc{Separation}} \rightarrow \beta(\KP^{\gamma}_{n} + (\Sigma^0_{1+\gamma+2})_{n}\text{--}\Det)
        \end{align*}
    \end{corollary}
    \begin{proof}
        We have $\KP^{\gamma} + \Sigma_{n}$-\textsc{Separation}. Let $L_{\alpha} \prec_{\Sigma_{n + 1}} L$ with $\omega_{\gamma} < \alpha$.  Then $L_{\alpha} \models \KP^{\gamma}_{n+1}$ and by $\Sigma_1$--elementarity, we have  $L_{\alpha} \models (\Sigma^0_{1+\gamma+2})_{n}\text{--}\Det$.  
    \end{proof}
The argument of Corollary \ref{CorollaryKPGammanimpliesDet} also works with $\beta(\KP^{\gamma} + \Sigma_{n}\text{--\textsc{Separation}})$ instead of $\KP^{\gamma} + \Sigma_{n}\text{--\textsc{Separation}}$, so we obtain the following picture:

\begin{figure}[h]
\begin{center}
\begin{tikzcd}
\KP^{\gamma} + \Sigma_{n+1}\text{--\textsc{Separation}}      \arrow[ddr]   &  \\
&\\
 &  \beta(\KP^{\gamma}_{n+1} + (\Sigma^0_{1+\gamma+2})_{n+1}\text{--}\Det)   \arrow[d]   \\
 &  \beta((\Sigma^0_{1+\gamma+2})_{n+1}\text{--}\Det)    \arrow[d]   \\
 \beta(\KP^{\gamma} + \Sigma_{n+1}\text{--\textsc{Separation}} )     \arrow[uur]    &    (\Sigma^0_{1+\gamma+2})_{n+1}\text{--}\Det   \arrow[d] \\
 &  \beta(\KP^{\gamma}_{n+1})   \arrow[d] \\                                                
 &  \beta(\KP^{\gamma}+ \Sigma_n\text{--\textsc{Separation}})                                          
\end{tikzcd}

\end{center}
\caption{Provability relations established in this article, for $1\leq \gamma \leq \omega_1^{\mathsf{CK}}$ and $1 \leq n < \omega$; none of the arrows can be reversed. Here, $\KP^\gamma_m$ denotes the theory $\KP$ augmented with the axioms ``$\mathcal{P}^\gamma(\mathbb{N})$ exists'' and ``$\gamma$ is wellordered,'' as well as with the schemata of $\Delta_m$--\textsc{separation} and $\Sigma_m$--\textsc{collection}.}
\label{FigureProvIntro}
\end{figure}
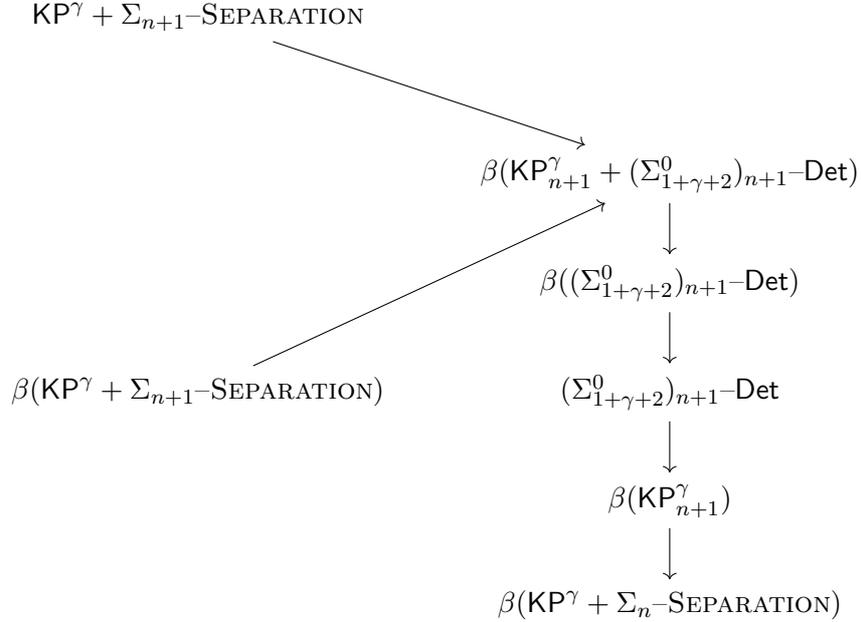

\begin{remark}
    It will be clear that our proof, while conceived for the lightface hierarchy, relativizes to any parameter. Thus, with the implications towards $\beta$--principles, we intend that the general pictures can be relativized to any parameter, so that the $\beta$-model can be constructed over it. In \S \ref{SectionReflection} and its results, we will explicitly use the boldface hierarchy because we will need the existence of $\beta$--models over any parameter.
\end{remark}

As a consequence of Theorem \ref{TheoremMainIntro}, we shall prove the equivalence between determinacy for arbitrary Boolean combinations of sets and reflection principles in higher-order arithmetic, answering two questions of Pacheco and Yokoyama \cite{pacheco}:
 
\begin{theorem}
Let $1 \leq \gamma < \omega_1^{\mathsf{CK}}$. Then the following are equivalent:
\begin{enumerate}
\item $\forall \delta < \gamma\ \forall m \, (\mathbf{{\Sigma}}^0_{1 + \delta + 2})_m$--determinacy, and
\item $\Pi^1_3\text{--}\mathrm{Ref}(\{\Z_{2 + \delta} : \delta < \gamma\})$.
\end{enumerate}

\end{theorem}
This is done in \S \ref{SectionReflection}. The results of \S \ref{SectionLowerBound} are based on the contents of the second-listed author's Master's thesis \cite{KouptchinskyMSc}, written under the supervision of the first author. 

\paragraph{Acknowledgements}
The authors would like to thank Keita Yokoyama and Leonardo Pacheco for fruitful discussions and for bringing the questions in \cite{pacheco} to their attention. This work was partially supported by FWF grant P36837.

\section{Preliminaries}
\label{SectPreliminaries}
In this section, we collect various definitions and conventions which will be used throughout the article.  
We will assume some familiarity with subsystems of second-order arithmetic, e.g., as in Simpson \cite{Simpson} and with Kripke-Platek set theory, e.g., as presented by Barwise \cite{Barwise}.
\subsection{The difference hierarchy}

While it will be clear that our results can be relativized to any real parameter, our use of transfinite levels of the Borel hierarchy requires us to introduce the \textit{effective} Borel hierarchy. 

\begin{definition}[Effective Borel Hierarchy ($\mathsf{ATR}_0$)]
    A code for an effective $\Sigma^0_{\alpha}$ set ($\alpha < \omega_1^{\mathsf{CK}}$) in a topological space $(X, \tau)$, where $\tau$ is a computable set of index (elements of $\omega$) for basic open sets, is a well-founded computable tree $T \subseteq \omega^{<\omega}$ satisfying the following:  \begin{enumerate}
        \item The leaf nodes represent basic open sets;
        \item A node $\sigma$ represents the (possibly infinite) union of the complements of every set represented by $\sigma^{\smallfrown}n$ for $n < \omega$;
        \item The ordinal rank of the Kleene-Brouwer ordering of $T$ is at most $\alpha$.
    \end{enumerate}
    The assumption of $\mathsf{ATR}_0$ guarantees the existence of a unique function interpreting this tree.
\end{definition}

In the following, we will generally not explain these coding techniques, assuming the reader is familiar with them. We will however emphasize when we use crucially this effective version of the Borel Hierarchy.

\begin{definition}[Hierarchy of differences (for $\Pi^0_{\alpha}$ sets)]
    In any topological space $X$, given ordinal numbers $\gamma, \alpha < \omega_1^{\mathsf{CK}}$, we say that a set $A$ is 
    ${(\Pi_{\alpha}^0)}_{\gamma}$ if there are $\Pi^0_{\alpha}$ sets $A_0, A_1, \dots, A_{\gamma} = \emptyset$ such that
    \begin{align*}
        x \in A \leftrightarrow \text{the smallest $\delta \leq \gamma$ such that $x \not\in A_{\delta}$ is odd}.
    \end{align*}
    We extend the odd/even partition of the natural numbers naturally by considering each limit ordinal as even, their successor as odd, and so on.
    We say that the sequence $\{ A_{\delta} : \delta \leq \gamma \}$ represents $A$ (as a ${(\Pi^0_{\alpha})}_{\gamma}$ set).\label{Pidif}
\end{definition}

We can take the sequence $\{ A_{\delta} : \delta \leq \gamma \}$ to be decreasing without loss of generality, in the sense that $A_{\delta_2} \subseteq A_{\delta_1}$ for all $\delta_1 \leq \delta_2 \leq \gamma $ (see figure~\ref{examplediff} for an example.)

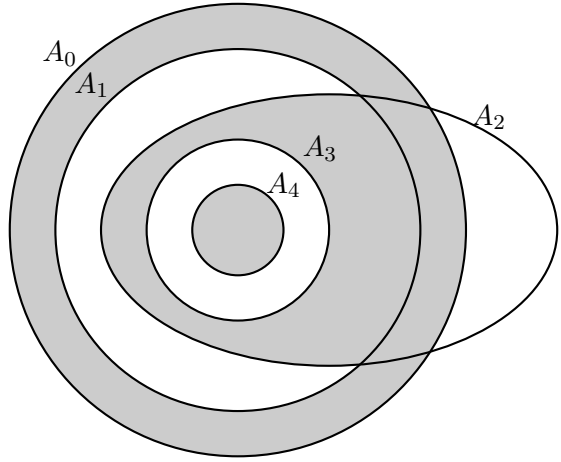
\begin{figure}[h]
    \centering
    \begin{tikzpicture}[scale = 0.6]
        \fill[black!20, even odd rule] (0,0) circle (5cm) (0,0) circle (4cm) (0,0) circle (3cm) (0,0) circle(2cm) (0,0) circle(1cm);
        
        \begin{scope}
            \clip (0,0) circle(4cm);
            \fill[black!20, even odd rule] (0,0) circle (2cm) (2,0) ellipse (5cm and 3cm);
        \end{scope}
        \begin{scope}
            \clip (0,0) circle(4cm);
            \fill[white, even odd rule] (0,0) circle (4cm) (2,0) ellipse (5cm and 3cm);
        \end{scope}
        
        \draw[thick] circle(1cm);
        \node at (1,1) {$A_4$};
        \draw[thick] circle(2cm);
        \node at (1.8,1.8) {$A_3$};
        \draw[thick] (2,0) ellipse (5cm and 3cm);
        \node at (5.5,2.5) {$A_2$};
        \draw[thick] (0,0) circle(4cm);
        \node at (-3.2,3.2) {$A_1$};
        \draw[thick] (0,0) circle(5cm);
        \node at (-3.9,3.9) {$A_0$};
        
    \end{tikzpicture}
    \caption{A ${(\Pi^0_{\alpha})}_5$ set, where $A_2$ plays the same role as $A_2 \cap A_1$.}\label{examplediff} 
\end{figure}

It is easy to see that $\bigcup_{n<\omega} {(\Pi_{\alpha}^0)}_n$ is exactly the set of the boolean combinations 
of $\Pi_{\alpha}^0$ sets. By a result of Kuratowski (\cite{Kuratowski}), we have \begin{align*}
    \Delta_{\alpha + 1}^0 = \bigcup_{\beta < \omega_1} {(\Pi_{\alpha}^0)}_\beta,
\end{align*}
where the right-hand side involves the extension of the difference hierarchy to the transfinite. 
Here, $\beta$ differences of $\Pi^0_{\alpha}$ sets involve infinite intersections of $\Pi^0_{\alpha}$ sets. Kuratowski's equality requires the fact that $\Pi$ classes are closed under infinite intersections, for if we had defined instead a hierarchy of differences for $\Sigma$ sets, we would have obtained $(\Sigma^0_{\alpha})_{\omega} = \Pi^0_{\alpha + 1}$. Thus, we need to change the definition for the hierarchy of differences of $\Sigma$ sets. 

\begin{definition}[Hierarchy of differences (for $\Sigma^0_{\alpha}$ sets)]
    In any topological space $X$, given ordinal numbers $\gamma, \alpha < \omega_1^{\mathsf{CK}}$, we say that a set $A$ is 
    ${(\Sigma_{\alpha}^0)}_{\gamma}$ if there are $\Sigma^0_{\alpha}$ sets $A_0, A_1, \dots, A_{\gamma} = X$ such that
    \begin{align*}
        x \in A \leftrightarrow \text{the smallest $\delta \leq \gamma$ such that $x \in A_{\delta}$ is odd}.
    \end{align*}
    We extend the odd/even partition of the natural numbers naturally by considering each limit ordinal as even, their successor as odd, and so on.
    We say that the sequence $\{ A_{\delta} : \delta \leq \gamma \}$ represents $A$ (as a ${(\Sigma^0_{\alpha})}_{\gamma}$ set).\label{Sigmadif}
\end{definition}
Here we can assume that the sequence is nested upwards. 
For determinacy, we can prove that the existence of winning strategies for $(\Pi^0_{\alpha})_{\beta}$ payoff sets is equivalent to the existence of winning strategies with $(\Sigma^0_{\alpha})_{\beta}$ payoff sets, for any ordinal numbers $\alpha$ and $\beta$. For the same reason, as far as it concerns determinacy, it does not matter if we change ``odd'' to ``even'' in the definition. What is more, in the finite case, both definition are interchangeable.

An example we want to highlight is that at our level of complexity, determinacy in the Baire space, $\omega^{\omega}$, is equivalent to determinacy in the Cantor space, $2^{\omega}$. This is illustrated by the following proposition (see e.g., \cites{MS, CantorDet}). We omit the proof. 

\begin{lemma}
Let $\Gamma$ be a pointclass closed under recursive substitutions, unions with $\Sigma^0_2$ sets, and intersections with $\Pi^0_2$ sets. Then, we have $\Gamma$--$\Det^* \leftrightarrow \Gamma$--$\Det$. 
In other words, determinacy for the payoff sets in class $\Gamma$ is independent of whether games take place in Baire space or in Cantor space.  
\end{lemma}

\subsection{Admissible sets}

In the scope of identifying subsystems of arithmetic with theories in the language of set theory $L_{\mathrm{Set}}$, we will use admissible sets, whose theory~\cite{Barwise} extensively develops. 

\begin{definition}[$\mathsf{KP}$]
\emph{Kripke-Platek} $L_{\mathrm{Set}}$-theory, $\mathsf{KP}$, is axiomatized by: 
\begin{enumerate}
\item The following basic axioms:
    \begin{multicols}{2}
        \textsc{Infinity},
        \\\textsc{Extensionality},
        \\\textsc{Pair},
        \\\textsc{Union};
    \end{multicols}
    \item The full scheme of \textsc{foundation}
    \begin{align*}
        \exists x \ \phi(x) \rightarrow \exists x \ (\phi(x) \land \forall y \in x \ \lnot \phi(y));
    \end{align*}
    \item And the schemes of
    \textsc{$\Delta_0$--separation} and \textsc{$\Delta_0 $--collection}, the latter being 
    \begin{align*}
        \forall x \in u \ \exists y\ \phi(x,y) \rightarrow \exists v \ \forall x \in u \
        \exists y \in v \ \phi(x,y),
    \end{align*}
    for all $\Delta_0$ formulae in which $v$ does not occur free.
\end{enumerate}
\end{definition}
Note that the schemes of \textsc{separation} and \textsc{collection} deal with arbitrary sets, and a priori not e.g., only with countable ones.
 
\begin{definition}[Admissible set]
    An \emph{admissible} set is a model of $\mathsf{KP}$ of the form
    \begin{align*}
        (A, \in_A),
    \end{align*}
    where $A$ is transitive and $\in_A\, =\, \in\cap A\times A$ is membership.
\end{definition}

Any admissible set actually satisfies $\Sigma_1$--\textsc{collection} and $\Delta_1$--\textsc{separation}. We also point out the difference between \textsc{replacement} and \textsc{collection}. It can be shown that $\Sigma_1$--\textsc{replacement} is also a theorem of $\mathsf{KP}$. We will be particularly interested in admissible initial segments of G\"odel's hierarchy $L$. An ordinal $\alpha$ is \textit{admissible} if $L_\alpha$ is admissible.

We then have a generalized notion of admissible sets which gives us a hierarchy of set theories. 

\begin{definition}[$n$--admissibility]
    For any $1 \leq n < \omega$ we say that a set $A$ is $n$--admissible if\begin{enumerate}
        \item $A$ is admissible,
        \item $(A, \in_A)$ is a model of \textsc{$\Delta_{n}$--separation} and \textsc{$\Sigma_{n}$--collection}.
    \end{enumerate}
    We say that an ordinal $\alpha$ is $n$--admissible if $L_{\alpha}$ is $n$--admissible.
\end{definition}

We now present some characterizations and consequences of $n$--admissibility and $\Sigma_n{-}$admissibility. These are well known, but we recall them for the reader's convenience (see for example~\cites{Devlin, Jensen}):

\begin{lemma}
    Let $\omega < \alpha$ be an infinite ordinal and $1 \leq n$, a natural number, the following assertions on $L_{\alpha}$ are equivalent:
    \begin{enumerate}
        \item It is a model of $\KP +$ \textsc{$\Sigma_{n-1}$--separation} $+$ \textsc{$\Delta_{n-1}$--collection} and;
        \item It satisfies $\Sigma_n$ bounding: \begin{align*}
        \forall \delta < \alpha \quad 
            (L_{\alpha} \models &\forall \gamma < \delta \ \exists y \ \phi(\gamma,y)) 
            \\ &\rightarrow \exists \lambda < \alpha \ 
            L_{\alpha} \models (\forall \gamma < \delta\ \exists y \in L_{\lambda} \ 
            \phi(\gamma, y)),
        \end{align*}
        where $\phi$ is $\Pi_{n-1}$ with parameters from $L_{\alpha}$;
        \item For any function $f$ with domain some $\delta < \alpha$ which is $\Sigma_n$ 
        (equivalently $\Pi_{n-1}$) over $L_{\alpha}$, 
        $f[\gamma] \in L_{\alpha}$ for every $\gamma < \delta$.
    \end{enumerate}
    Thus, we have the classical bounded quantifier-elimination rule: For any $\Pi_{n-1}$ formula 
    $\phi$, ``$\forall x \in t\ \exists y\ \phi$'' is equivalent to a $\Sigma_n$ formula. Moreover any such $L_{\alpha}$ is consequently a model of \textsc{$\Sigma_{n}$--collection}. 
    
    We then call $L_{\alpha}$ or $\alpha$, $\Sigma_n$--admissible.\label{charKPL}
\end{lemma}

We say that a structure of the language of set theory \M\ \textit{has $\Sigma_n$ Skolem functions} if it has the following property: there is a function $h$ definable over the structure and which associates a witness to all pairs $(\psi,a)$ of $\Pi_{n-1}$ formula $\psi$ and parameter $a \in \M$ such that $\M \models \exists x \psi(x,a)$. In other words, a $\Sigma_n$ Skolem function of \M\ is a partial function $h \colon \omega \times M \to M$, $h \in M$ such that for all set $A$ that is $\Sigma_n$--definable with some parameters $\bar{p}$, there is an $i$ such that $h(i,\bar{p}) \in A$. The following lemmata are stated in \cite{MS}, some of them being easily adapted to our needs:

\begin{lemma}
    If $\alpha$ is $\Sigma_n$--admissible, then $L_{\alpha}$ has a parameter-free 
    $\Sigma_{n+1}$ Skolem function. \label{noParamSkolem}
\end{lemma}

\begin{lemma}
    If $L_{\alpha}$ is $\Sigma_n$--admissible, then it satisfies $\Delta_n$ \textsc{separation}, 
    that is, for any $u \in L_{\alpha}$ and $\Sigma_n$ formulae $\phi(z)$ and $\psi(z)$ 
    such that $L_{\alpha} \models \forall z \ (\phi(z) \leftrightarrow \lnot\psi(z))$, $\{z \in u \mid \phi(z)\} \in L_{\alpha}$. Moreover, it satisfies $\Sigma_{n}$-\textsc{collection} and thus $L_{\alpha_n}$ is $n$--admissible.\label{betterSepL}
\end{lemma}

\begin{lemma}
    If $\alpha$ is $(k+1)$--admissible, then there are unboundedly many $\gamma < \alpha$ such that $L_{\gamma} \preceq_k L_{\alpha}$.\label{charSepL}
\end{lemma}

\begin{lemma}
    For any ordinals $\gamma < \beta$, if $L_{\gamma} \preceq_n L_{\beta}$ and $\beta$ is $({n-1})$--admissible, then 
    $\gamma$ is $n$--admissible. \label{boundingPrinciple}
\end{lemma}

\begin{lemma}
    Let $\kappa < \omega_1^{\mathsf{CK}}, \alpha_n$ be the first $n$--admissible ordinal satisfying  ``$\omega_{\kappa}$ exists'', then every element of $L_{\alpha_n}$ is $\Sigma_{n+1}$--definable over $L_{\alpha_n}$ with $\omega_{\kappa}$ as parameter.
\end{lemma}

The following lemma is a consequence of the so-called \textit{acceptability} of the constructible hierarchy. We define $\KP^{\gamma}_{\infty} = \bigcup_{n \in \mathbb{N}} \KP^{\gamma}_n$.
\begin{lemma}
    Let $\kappa < \omega_1^{\mathsf{CK}}$ and $\beta_{\kappa}$ be the smallest ordinal $\beta$ such that $L_{\beta} \models \KP^{\kappa}_{\infty}$, then for all $\gamma < \beta_{\kappa}$, we have
    \[L_{\gamma+1} \models |L_\gamma| \leq | \mathcal{P}^{\kappa_{\gamma}}(\omega)| \text{``}\leq \omega_{\kappa} \text{''}.\]
    That is, in $L_{\gamma+1}$ there is an injection from $L_{\gamma}$ to $\omega_{\kappa_{\gamma}}$, where $\kappa_{\gamma}$ is the largest ordinal $\alpha \leq \kappa$ such that $L_{\gamma} \models$ ``$\omega_{\alpha}$ exists.''
     \label{cardinalityKPL}
\end{lemma}
\begin{proof}
    We proceed by induction on $\kappa$, starting with $\kappa = -1$, where our claim is that $L_{\gamma}$ is finite in $L_{\gamma +1}$ and $\beta_{-1} = \omega$ (where we remove \textsc{infinity} from $\ZF^-$). This case is obviously true. Note also that $\Ps^0(\omega) \coloneqq \omega$ (then we ask for \textsc{infinity} to hold). 

    Let $\kappa < \omega_1^{\mathsf{CK}}$. We first make two remarks: the first is that $L_{\gamma}$ always satisfies the Axiom of Choice, so the existence of an injection in one direction is equivalent to the existence of a surjection in the opposite direction. The second is that $\Ps^{\kappa}(\omega)$ admits a trivial coding of pairs, i.e., pairs of elements of $P^{\kappa}(\omega)$ can be coded by single elements of $\Ps^{\kappa}(\omega)$.
    
    To continue with the induction, suppose first that there is $\kappa' < \kappa$ such that $\gamma \leq \beta_{\kappa'}$, then the conclusion is immediate from our inductive hypothesis. Now we proceed by transfinite induction. If $\gamma$ is a successor, the conclusion follows from the countable number of formulae and the bound given by induction on the cardinality of the parameter space, $L_{\gamma}$. If $\gamma$ is a limit ordinal less than $\beta_{\kappa}$, then it is not $k$--admissible for some $k < \omega$. It means that, for some $k <\omega$ and some $\delta < \gamma$, there is a $\Sigma_k$--definable map $f$ over $L_{\gamma}$ such that $f[\delta]$ is unbounded in $L_{\gamma}$. We can now define a surjective function from $\Ps^{\kappa_{\gamma}}(\omega)^{L_{\gamma}}$ to $L_{\gamma}$ by combining the maps $g_{\delta} \colon \Ps^{\kappa_{\delta}}(\omega)^{L_{\delta}} \to L_{\delta}$ and $g_{f(\zeta)} \colon \Ps^{\kappa_{f(\zeta)}}(\omega)^{L_{f(\zeta)}} \to L_{f(\zeta)}$ given by the induction hypothesis and the observation that $\Ps^{\kappa_{\gamma}}(\omega)^{L_{\gamma}}$ code all these power set domains. Note that the maps $g_{f(\zeta)}$ can be produced using the Axiom of Choice in $L_{\gamma+1}$.
\end{proof}

\subsection{On some choice schemes}

We recall the schema of $\Sigma_{m}$--Dependent Choice, stated as \begin{align*}
    \forall X \ \exists Y \ \eta(X, Y) \rightarrow \exists Z (= (Z_i)_{i < \omega}) \ \forall (i < \omega) \ \eta(i, (Z_j)_{j < i}, Z_i).
\end{align*}

Strong $\Sigma_{m}$--$\DC$ is a stronger version of dependent choice, namely. It asserts the existence of a sequence $Z = {(Z_i)}_{i \in \omega}$, such that\begin{align*}
    &\forall Y\ (\eta(i,(Z_j)_{j < i}, Y) \rightarrow \eta((Z_j)_{j < i}, Z_i)),
\end{align*}
where $\eta$ is $\Sigma_{m}$. 

\section{Unprovability of Determinacy Principles in Higher-Order Arithmetic} \label{SectionLowerBound} 

\subsection{Unprovability of determinacy in higher-order arithmetic}\label{SubsectionHOA_LB}

In this section, we present our unprovability results for theories of determinacy in higher-order arithmetic. We will see that the principles of $(\Sigma^0_{1 + \gamma + 2})_n\text{--}\Det$ are not provable in suitably chosen systems of higher-order arithmetic, defined as subsystems of $\mathsf{ZFC}$ in the language of set theory. We introduced the notation $\KP^{\gamma}_n$, where the subscript $1 \leq n $ denotes the 
amount of \textsc{separation} and \textsc{collection} and the superscript $1 \leq \gamma$ denotes the order of the available objects --``$\Ps^{\gamma}(\mathbb{N})$ exists''--. 
For a different study and axiomatization of higher and third-order arithmetic in particular, in the fashion of a many-sorted language in continuation of $\mathcal{L}_2$, one can consult Hachtman~\cite{Z3Det}.
The results in this section extend Theorem \ref{TheoremMSIntro}. Below, we use $\wo(\gamma)$ for the assertion that $\gamma$ is a wellordering of $\mathbb{N}$.

\begin{theorem}
Consider $2 \leq n < \omega$, $1 \leq \gamma < \omega_1^{\mathsf{CK}}$ and let \label{notmPigammaDet} 
\begin{align*}
        \KP^{\gamma}_n \coloneqq \KP + \wo(\gamma) + \text{``$\Ps^{\gamma}(\mathbb{N})$ exists'' } + \textsc{$\Sigma_{n}$--collection} + \textsc{$\Delta_{n}$--separation}.
    \end{align*}
    Then, 
    \begin{align*}
    \KP^{\gamma}_n \not\vdash (\Sigma^0_{1 + \gamma + 2})_{n}\text{--}\Det.
    \end{align*}
\end{theorem}

Theorem \ref{notmPigammaDet} is stated for $2 \leq n$. This assumption is necessary for the proof, which does not work for the case $n = 1$. However, the case $n = 1$ of Theorem \ref{notmPigammaDet} has been proved by Hachtman \cite{Hachtman}, using different methods.

\begin{remark}
    Martin (see \cite{Martin}) has previously shown that the determinacy of $\Sigma^0_{1+\gamma+3}$ games is not provable in $\mathsf{ZC}^- + \textsc{$\Sigma_1$ replacement} + $ ``$\mathcal{P}^\gamma(\mathbb{N})$ exists'' (here, $\mathsf{ZC}^-$ is $\mathsf{ZFC}$ deprived of the \textsc{power set} and \textsc{replacement} axioms). Theorem \ref{notmPigammaDet} is a refinement of this result. 
    \end{remark}

The rest of \S \ref{SubsectionHOA_LB} is devoted to the proof of Theorem \ref{notmPigammaDet}.
The starting point for our argument is the proof of Theorem{~}\ref{TheoremMSIntro} in~\cite{MS}. 
As we go along, we indicate the differences between our proof and theirs.

We will build up a game in $2^{< \omega}$ with a $(\Sigma^0_{1+ \gamma + 2})_{n}$ winning 
condition. That is, a play $x$ is a win for $\I$ if and only if it belongs to a set $X$ being $(\Sigma^0_{1+ \gamma + 2})_{n}$ (here we consider the effective version of the Borel hierarchy). 
We will describe this winning set with $n$ different $\Sigma^0_{1 + \gamma + 2}$ conditions $\phi_{i = 0, \dots, n-1}$ which will lead to the victory of Player \I\ 
if and only if the smallest $i$ such that $\phi_{\I}$ does not hold is even. This will correspond to a failure of Player \II\ to respect certain properties. Otherwise, Player \II\ wins. We add the final condition $\phi_{n}(m) \equiv m \neq m$ at the end to settle the game in case none of the preceding conditions fail.
In the game, Player \I\ and Player \II\ must play consistent and complete theories $T_{\I}$ and $T_{\II}$ in the language of set theory which extends $T^{\gamma}_n$. This is done as follows: fixing some enumeration $\{\psi_i:i\in\mathbb{N}\}$ of all formulas in the language of set theory, we think of a player as ``accepting'' a formula $\psi_i$ if his or her play during turn $i$ is $1$; and we think of the player as ``rejecting'' the formula if the play is $0$.

In a first time we will define a $(\Sigma^0_{1+ \gamma + 2})_{2n+2}$ game $G'^{\gamma}_n$ with a payoff set defined by a sequence of conditions 
\begin{align*}
    (C_{\II}0), (C_{\I}0), (C_{\II}1), (C_{\I}1), \ldots, (C_{\II}(1+k)), (C_{\I}(1+k)), \ldots, (C_{\II}n), (C_{\I}n).
\end{align*}
At even stages, we will introduce conditions denoted by $C_{\II}i$, for $i\leq n$: these are conditions that  \II\ must ensure hold in order to avoid losing. 
Similarly, conditions $C_{\I}i$ will appear in odd stages, and these must be satisfied by \I\ in order to avoid losing. By definition, the first player not to satisfy one of his or her conditions will lose the game. The model that will witness the failure of $(\Sigma^0_{1 + \gamma + 2})_{n}$ determinacy is the unique wellfounded one 
of the theory \begin{align*}
    T^{\gamma}_n \coloneqq \KP^{\gamma}_n + V = L + \forall \alpha \in\Ord (L_{\alpha} \text{ is not a model of } \mathsf{KP^{\gamma}_n}).
\end{align*}
The later model is  $L_{\alpha_n^{\gamma}}$, where $\alpha_n^{\gamma}$ is the smallest ordinal $\alpha$ such that 
$L_{\alpha}$ is a model of $\KP^{\gamma}_n$. The objective of our game is to satisfy the following lemma.

\begin{lemma}
    The game $G'^{\gamma}_n$ satisfies: \begin{enumerate}
        \item If Player \I\ plays the theory of $L_{\alpha_n^{\gamma}}$, she wins;
        \item If Player \I\ does not play the theory of $L_{\alpha_n^{\gamma}}$, but Player \II\ does, then Player \II\ wins. 
    \end{enumerate}\label{gamedef}
\end{lemma}

We will show that we can reorganize the conditions of $G'^{\gamma}_n$ so that to devise a $(\Sigma^0_{1+ \gamma + 2})_{n}$ game $G^{\gamma}_n$ which preserves the properties of the lemma. Assuming we showed the existence of such a game,  let us prove theorem~\ref{notmPigammaDet}.
\begin{proof}[Proof of theorem~\ref{notmPigammaDet}]
    Consider the game $G^{\gamma}_n$. This game also satisfies the conclusion of Lemma~\ref{gamedef}. First notice that Player $\II$ cannot have a winning strategy, for Player $\I$ can always win by playing the theory of $L_{\alpha^{\gamma}_n}$. Suppose Player $\I$ has a winning strategy $\sigma$. Then, the theory of $L_{\alpha^{\gamma}_n}$ is computable with oracle $\sigma$ for suppose $\II$ stick to the strategy consisting of copying the moves of Player \I\ (which is a computable strategy); then, ignorant about Player $\II$ actual strategy, Player $\I$ can only ensure a win by providing the exact sequence encoding the theory of $L_{\alpha^{\gamma}_n}$. 
    It follows that $\sigma$ cannot belong to $L_{\alpha^{\gamma}_n}$ by Tarski's theorem on the undefinability of truth. We claim that
\begin{align*}
L_{\alpha^{\gamma}_n} \models \text{``$G^{\gamma}_n$ is not determined.''}
\end{align*}
Otherwise, letting $\sigma \in L_{\alpha^{\gamma}_n}$ be a strategy which is winning in $L_{\alpha^{\gamma}_n}$, we have that $\sigma$ is also a winning strategy in reality, as $L_{\alpha^{\gamma}_n}$ is a $\beta$-model.
Therefore, $L_{\alpha^{\gamma}_n}$ is a model of $\KP^{\gamma}_n + V = L$ in which some $(\Sigma^0_{1 + \gamma + 2})_n$ game is not determined, completing the proof.
\end{proof}

Note than even if the use of the effective hierarchy is crucial to allow the transfer of determinacy from in to out of the constructible universe, it will be clear that we can relativize theorem~\ref{notmPigammaDet} to any real.

Let us now devise the game under discussion and prove the key lemma~\ref{gamedef}.

We ask that the players play a definably Henkin theory in the sense that for every formula of the form $\phi(x)$, there is some formula $\theta(y)$ such that 
\[\exists x\, \phi(x) \in T_{\I} \to \exists!y\, \theta(y) \wedge \exists x\, (\theta(x)\wedge \phi(x)) \in T_{\I}.\]
If this condition is satisfied, then $T_{\I}$ has a unique minimal model $\M_{\I}$ up to isomorphism, which is determined by all parameter-free definable elements of any model of $T_{\I}$.
We demand the same of $T_{\II}$, and denote by $\M_{\II}$ the corresponding model, if it exists. 

We start to define the first conditions $C_{\I}0$ and $C_{\II}0$. Furthermore, we phrase these and the subsequent conditions in terms of the models $\M_{\I}$ and $\M_{\II}$ by abusing notation, rather than the theories $T_{\I}$ and $T_{\II}$. We use the subscript notation $\M_{::}$ when we do not want to specify which model we are talking about.
 \begin{align*}
    &(C_{\I}0) : \qquad \M_{\I} \models T^{\gamma}_n \land \M_{\I} \text{ is an $\omega$--model } 
    \land \Ps^{\gamma+1}_{\M_{\I}}(\mathbb{N}) \not\subseteq \Ps^{\gamma + 1}_{\M_{\II}}(\mathbb{N}).   
    \\&(C_{\II}0) : \qquad \M_{\II} \models T^{\gamma}_n \land \M_{\II} \text{ is an $\omega$--model } 
    \land \Ps^{\gamma+1}_{\M_{\II}}(\mathbb{N}) \not\subseteq \Ps^{\gamma + 1}_{\M_{\I}}(\mathbb{N}).  
\end{align*}
In our conditions we abuse notations when we write equalities or inclusions between sets of elements of different models, by which we mean we can inject one into the other with an appropriate monomorphism (which we will provide).

To express this condition, we need the following definition. For a code $t \in \mathbb{N}$ of an element from $\M_{\I}$, we will often abbreviate ``$\M_{\I} \models \forall x \in t, \ x \in \Ps^{\gamma}(\mathbb{N})$'' by ``$t \in \Ps^{\gamma + 1}_{\M_{\I}}(\mathbb{N})$'' and \textit{mutatis mutandis} for $\M_{\II}$ and the ordinals $\alpha \leq \gamma$. Since we require the models played to have the same set of natural numbers as in the real world, we will write $\mathbb{N}$ instead of $\omega_{\M_{\I}}$ or $\omega_{\M_{\II}}$. 

The key aspect of the construction is devising rules of the appropriate complexity which allow us to compare the models $\M_\I$ and $\M_{\II}$.
In principle, with both models being term-models, we need to rely on an isomorphism to compare elements from one model to another. Here, we can do it using two facts that are entailed by satisfying the first condition above: first that both models have to be $\omega$--models, allowing quantification on first order variable to design the corresponding elements of $\omega_{\M_{::}}$; second that $\M_{::}$ are minimal segment of the constructible hierarchy satisfying $\KP^{\gamma}_n$ and thus, each of their lower levels is to be coded by a subset of $\Ps_{\M_{::}}^{\gamma}(\mathbb{N})$. We shall define now a relation on natural numbers that will code the isomorphism relation between both model, which we denote by $\simeq$, under the right hypotheses. \begin{definition}
    We define $D_{\gamma}$, with domains $\Ps^{\gamma}_{\M_{\I}}(\mathbb{N}) \times \Ps^{\gamma}_{\M_{\II}}(\mathbb{N})$, the ``dummy intersection'' of the $\gamma$-th power sets, by induction on $\alpha$:
\begin{align*}
    D_1(t_1,t_2) &\leftrightarrow (t_1,t_2) \in \mathcal{P}(\mathbb{N})_{\M_{\I}}\times \mathcal{P}(\mathbb{N})_{\M_{\II}} \\
    &\qquad \wedge \forall n \ (\M_{\I} \models n \in t_1 \leftrightarrow \M_{\II} \models n \in t_2) \\
     D_{1 + \alpha + 1}(t_1,t_2) &\leftrightarrow D_{1+\alpha}(t_1,t_2) \lor \bigg[
     (t_1,t_2) \in \mathcal{P}(\mathbb{N})^{1+\alpha+1}_{\M_{\I}}\times \mathcal{P}(\mathbb{N})^{1+\alpha+1}_{\M_{\II}}\\
     &\qquad \wedge \forall (z_1,z_2) \ \Big(D_{1 + \alpha}(z_1,z_2) \rightarrow (\M_{\I} \models z_1 \in t_1 \leftrightarrow \M_{\II} \models z_2 \in t_2)\Big)\bigg] \\
    D_{\delta}(t_1,t_2) &\leftrightarrow \bigvee_{\alpha < \delta} D_{\alpha}(t_1,t_2).
\end{align*} 
If $\gamma = \alpha + 1$ is successor, $D_\gamma(t_1,t_2)$, is a $\Pi^0_{1 + \alpha}$ condition, and otherwise it is a $\Sigma^0_{\gamma}$ condition.\label{dummy}
\end{definition} 

The idea is that $D_\gamma(t_1,t_2)$ expresses that the terms $t_1$ and $t_2$ code the same element of $\mathcal{P}^\gamma(\mathbb{N})$. However, this is not quite the case. For instance, one of the models might contain real numbers not in the other. In this case $D_2$ already might yield ``false positives'' when applied to elements of the model that only differ in reals that do not belong to the common part of $\M_\I$ and $\M_{\II}$. The next lemma asserts that this is the only potential problem with $D_\gamma$.

\begin{lemma}\label{LemmaDummyInt}
    The following properties hold \begin{enumerate}
        \item\label{LemmaDummyInt1} For all $\xi\leq\gamma$, if $x \in \Ps^{\xi + 1}(\omega)_{\M_{\I}}$ and $y \in \Ps^{\xi + 1}(\omega)_{\M_{\II}}$, then we have $x \simeq y \rightarrow D_{\xi +1}(x,y)$;
        \item\label{LemmaDummyInt2} Suppose that $\xi \leq \gamma$, $a \in \Ord^{\M_\I}$, $b\in \Ord^{\M_{\II}}$ are such that  $\Ps^{\xi}(\omega)_{L_a} \simeq z_{\II}$ for some $z_{\II} \in \M_{\II}$ and  $\Ps^{\xi}(\omega)_{L_b}\simeq z_\I$ for some $z_{\I} \in \M_\I$, so that $z_{\I} \simeq z_{\II}$.
        Then for all $x \in \Ps^{\xi + 1}(\omega)_{L_{a+1}}$ and all $y \in \Ps^{\xi + 1}(\omega)_{L_{b+1}}$, we have $x \simeq y \leftrightarrow D_{\gamma +1}(x,y)$.
        \item\label{LemmaDummyInt3} Let $\xi \leq \gamma$ be so that $\Ps^{\xi}(\omega)_{\M_{\I}} \simeq \Ps^{\xi}(\omega)_{\M_{\II}}$ then for all $x \in \Ps^{\xi + 1}(\omega)_{\M_{\I}}$ and all $y \in \Ps^{\xi + 1}(\omega)_{\M_{\II}}$ we have $x \simeq y \leftrightarrow D_{\gamma +1}(x,y)$.
    \end{enumerate}
    
     \label{Dummy_intersect}
\end{lemma}
\begin{proof}
We prove \eqref{LemmaDummyInt1} and \eqref{LemmaDummyInt2} by induction; \eqref{LemmaDummyInt3} follows. The proof of \eqref{LemmaDummyInt1} is straightforward. For \eqref{LemmaDummyInt2} suppose first $\exists x \in \Ps^{\xi+1}(\omega)_{\M_{\I}} \ \exists y \in \Ps^{\xi+1}(\omega)_{\M_{\II}}$ such that $x \simeq y$, then  we have $D_{\xi+1}(x,y)$ by \eqref{LemmaDummyInt1}.
Suppose now that $x \not\simeq y$. Observe that $x$ contains only elements in $\mathcal{P}^\xi(\mathbb{N})_{L_a^{\M_\I}}$ and $y$ contains only elements in $\mathcal{P}^\xi(\mathbb{N})_{L_b^{\M_{\II}}}$, by the construction of $L$. 
Since by hypothesis $\Ps^{\xi}(\omega)_{L_a} \simeq z_{\II}$ for some $z_{\II} \in \M_{\II}$ and  $\Ps^{\xi}(\omega)_{L_b}\simeq z_\I$ for some $z_{\I} \in \M_\I$, the fact that $x \not\simeq y$ implies there is $t_1 \in \Ps^{\xi}(\omega)_{\M_{\I}}$, let us say in $x$ and $t_2 \in \Ps^{\xi}(\omega)_{\M_{\II}}$ with $t_1 \simeq t_2$ so that 
\begin{align*}
\M_{\I} \models t_1 \in x \land \M_{\II} \models t_2 \not\in y.
    \end{align*}
    The fact that $t_1$ must have an isomorphic copy in the other model--and \textit{vice versa} where applicable--follow indeed from the assumption $z_{\I} \simeq z_{\II}$.
    But then $D_{\xi}(t_1,t_2)$, hence $\lnot D_{\xi+1}(x,y)$ by \eqref{LemmaDummyInt1}, which concludes the proof of \eqref{LemmaDummyInt2}.
\end{proof}

\begin{lemma}
Conditions $(C_{\I}0)$ and $(C_{\II}0)$ are $\Sigma^0_{1 + \gamma + 2}$.\label{C0}
\end{lemma}
\begin{proof}
    Without loss of generality, we prove the claim for $(C_{\I}0)$. Saying that the set of sentences played by I is a complete and consistent extension of $T^{\gamma}_n$ is $\Pi^0_1$. Constructibility ensures the definably Henkin character of our theories.

    The second conjunct expresses that for every term $t$ such that $\M_{\I} \models t \in \omega$ there exists a natural number $n \in \mathbb{N}$ such that $\M_{\I} \models t = \sum_n 1$; thus it is $\Sigma^0_2$.
    
    The last conjunct asserts the existence of a term $t_1$ satisfying $\M_{\I} \models t_1 \subseteq \Ps^{\gamma}(\mathbb{N})$ but such that $\forall t_2 (\M_{\II} \models t_2 \subseteq \Ps^{\gamma}(\mathbb{N}) \rightarrow \lnot D_{\gamma + 1}(t_1, t_2))$, whose conclusion implies that $t_1 \not\simeq t_2$ by Lemma~\ref{Dummy_intersect}. The latter being $\Sigma_{1 + \gamma + 2}$, this concludes the proof.
\end{proof}

From now on, we will suppose that exactly one of the term-models is wellfounded: since the only wellfounded $\omega$--model of $T^{\gamma}_n$ is $L_{\alpha_n^{\gamma}}$, not both can be, provided they satisfy $C_{::}0$; moreover, in order to show that the game has the properties we shall need, it suffices to consider only plays in which one of the two models is wellfounded, which we call $\M$ and the other, illfounded one, $\N$. 

We now want to accurately identify the maximal common segment $\A^{\gamma + 1}$ of $\M$ and $\N$, that is, by the observation here above, the wellfounded part of $\N$. 

\begin{definition}
We define \begin{align}
    \A_1^{\gamma+1}(w_1, w_2) \leftrightarrow \ 
    &\exists (t_1, t_2) \in \Ps^{\gamma +1}_{\M_{\I}}(\mathbb{N}) \times \Ps^{\gamma + 1}_{\M_{\II}}(\mathbb{N}),\nonumber \\ 
    &\exists \beta_1, \beta_2, \ t_1 = L_{\beta_1}^{\M_{\I}}, \ t_2 = L_{\beta_2}^{\M_{\II}}, \text{and} \nonumber\\
    &(\M_{\I} \models t_1 \text{ codes } w_1 \land \M_{\II} \models t_2 \text{ codes } w_2 \land (t_1,t_2) \in D_{\gamma + 1}), \nonumber 
\end{align} a $\Sigma_{1 + \gamma + 1}$ condition. \label{DefA1}
\end{definition}

However, the dummy intersection $D_{\gamma}$ is not always suitable for our purpose. Indeed, as we already pointed out, it can misinterpret two high-order objects as isomorphic if they disagree only about objects (e.g., reals) not in the common segment of the models. 

The following conditions will endeavour to address this issue as well as to force each player to support evidence that they didn't produce illfounded sequences within their models. As we state further conditions, we will require the witnesses taken by our inquiries to be of higher and higher definable complexity. Along the way, we want to show that these conditions lead to $L_{\alpha}$, the wellfounded part of $\N$, being a model of $\KP^{\gamma}_i$ for $i \leq n$, which sentence is to reflect inside $\M_{::}$. This way not all conditions can be soundly satisfied by both player, since it can't simultaneously happen that \begin{align*}
    \M_{::} \models (L_{\alpha} \models \KP^{\gamma}_n) \quad \text{and} \quad \M_{::} \models T^{\gamma}_n, 
\end{align*}
by the minimality requirement of $T^{\gamma}_n$.
For this to work, we need to ensure in particular that $L_{\alpha}$, the common wellfounded part, is a model of ``$\Ps^{\gamma}(\mathbb{N})$ exists.''
We will make use of the following folklore lemma.

\begin{lemma}
    Let $\N$ be an $\omega$--model of $V = L$ and suppose $\N$ is illfounded with $\mathrm{wfo}(\N) = \alpha$ and that $\kappa \in L_{\alpha}$ is the largest cardinal of $L_{\alpha}$. Say $X \in \N$ is a non-standard code if $X \subseteq \kappa$ codes a linear order of $\kappa$ so that $\N$ has an isomorphism from $X$ onto some non-standard ordinal of $\N$. Then, \begin{align*}
        \mathcal{\I}^{\kappa} \coloneqq \{ X \in \N \setminus L_{\alpha} \mid X \text{ is a non-standard code } \}
    \end{align*}
    is non-empty, and has no $<^{\N}_{L}$--least element. \label{overspill}
\end{lemma}

Thus, we will require that for \begin{align*}
    \kappa = \omega, \omega_1, \cdots, \omega_{\alpha}, \cdots < \omega_{\gamma} =_{\KP^{\gamma} + V=L} |\Ps^{\gamma}(\mathbb{N})|,     
\end{align*}
the set of non-standard codes of cardinality $\kappa$, $\mathcal{\I}^{\kappa}$, is either empty or has a $<^{\N}_{L}$--least element. If so, $\omega_{\gamma}$ is a fortiori the greatest cardinal of $L_{\alpha}$. This is equivalent to asking the same condition of the non-standard codes that are elements of $\Ps^{\gamma}(\mathbb{N})$. Below, we let $\C^\gamma_\I$ be the set of elements of $\mathcal{P}^\gamma(\mathbb{N})_{\M_\I}$ which code ordinals of $\M_\I$, and we define $\C^\gamma_{\II}$ similarly. For elements of $\C^\gamma_{\I}$, we write $y <_{\C^\gamma_{\I}} x$ if $y$ codes a smaller ordinal in $\M_\I$, and similar for $\C^\gamma_{\II}$.
\begin{lemma}\label{LemmaEquivPC}
Suppose that $(C_\I0)$ and $(C_{\II}0)$ hold. Then, the following are equivalent:
\begin{enumerate}
\item $\Ps^{\gamma}(\mathbb{N})_{\M_{\I}} \subseteq \Ps^{\gamma}(\mathbb{N})_{\M_{\II}}$,
\item all elements of $\C^\gamma_{\I}$ code ordinals coded in $\C^\gamma_{\II}$.
\end{enumerate}
A similar result holds after swapping the roles of $\I$ and $\II$ above.
\end{lemma}
\proof
Since $\M_\I$ and $\M_{\II}$ are models of $V = L$ with no initial segments satisfying $\KP^{\gamma+1}$, it follows that every element of $\Ps^{\gamma}(\mathbb{N})_{\M_{\I}}$ is definable from an element of  $\C^\gamma_{\I}$ and all its $<_{\C^\gamma_{\I}}$-equivalent copies.
The converse is immediate.
\endproof

\begin{definition}
    \begin{align*}
       (C_{\I}1) : \Ps^{\gamma}(\mathbb{N})_{\M_{\I}} \not\subseteq \Ps^{\gamma}(\mathbb{N})_{\M_{\II}} \rightarrow \C^\gamma_{\I} \setminus \C^\gamma_{\II}& \text{ has a $<_{\C^\gamma_{\I}}$--minimal element}.  
        \\(C_{\II}1) :  \Ps^{\gamma}(\mathbb{N})_{\M_{\II}} \not\subseteq \Ps^{\gamma}(\mathbb{N})_{\M_{\I}}
         \rightarrow \C^\gamma_{\II} \setminus \C^\gamma_{\I}& \text{ has a $<_{\C^\gamma_{\II}}$--minimal element}.
    \end{align*}
\end{definition}

Conditions $(C_{\I}1)$ and $(C_{\II}1)$ tell us that descending sequences through $\N$ should be constituted out of objects from the highest cardinality available. It is not immediately clear that $(C_{\I}1)$ and $(C_{\II}1)$ are indeed expressible in a $\Sigma^1_{1+\gamma+2}$ way; indeed, our way of formalizing these conditions will differ slightly from their intended meaning, but only in cases in which this makes no difference. This is done in Lemma \ref{LemmaC1Formalization}.

Before doing so, we mention that 
in light of Lemma~\ref{overspill}, the conditions imply that the wellfounded part of the illfounded model satisfies $\KP^{\gamma}$.

\begin{lemma}\label{LemmaKPGammaFollowsFromC1}
Suppose that $(C_{\I}0)$, $(C_{\II}0)$, $(C_{\I}1)$ and $(C_{\II}1)$ hold. Then, $L_\alpha$ satisfies ``$\Ps^{\gamma}(\mathbb{N})$ exists'' and indeed $\KP^{\gamma}$.
\end{lemma}
\proof 
Recall that in $L_\alpha$ the cardinality of $\Ps^{\beta}(\mathbb{N})$ is $\omega_{\beta}$ for all $\beta$ and by $(C_{\I}1)$, $(C_{\II}1)$. By Lemma~\ref{overspill} applied in $\N$, no $\omega_{\beta}$ with $\beta < \gamma$ can be the largest cardinal of $L_{\alpha}$, so $L_{\alpha} \models$ ``$\omega_{\gamma}$ exists'' and thus $L_{\alpha} \models$ ``$\Ps^{\gamma}(\mathbb{N})$ exists.''
\endproof

We now turn to the formalization of $(C_{\I}1)$ and $(C_{\II}1)$. For this, recall that the rules of the game are such that the first player to violate any of the conditions loses.

\begin{lemma}\label{LemmaC1Formalization}
There exist formulas $\varphi_{\I}, \varphi_{\II} \in \Sigma^0_{1+\gamma+2}$ which express conditions $(C_{\I}1)$, $(C_{\II}1)$ in all plays which satisfy $(C_{\I}0)$ and $(C_{\II}0)$ and in which one of $\M_{\I}$ or $\M_{\II}$ is wellfounded.
\end{lemma}
\begin{proof}
We define $\varphi_{\I}\in \Sigma^0_{1+\gamma+2}$; $\varphi_{\II}\in \Sigma^0_{1+\gamma+2}$ is defined similarly. The formula $\varphi_{\I}$ is defined using $D_{\gamma}$ and written as an implication $A \rightarrow B$, where the premise $A$ asserts 
\[\Ps^{\gamma}(\mathbb{N})_{\M_{\I}} \not\subseteq \Ps^{\gamma}(\mathbb{N})_{\M_{\II}}.\] 
By the proof of Lemma~\ref{C0}, this formula is $\Sigma^0_{1 + \gamma + 1}$. The conclusion $B$ is stated as
\begin{align}
        &\exists y \in  \C^\gamma_{\I} \ \Big[ \forall x \in \C^\gamma_{\II} \ \lnot D_{\alpha + 1}(y,x) \land \forall t_1 \in  \C^\gamma_{\I} \ \big(t_1 <_{\C^\gamma_{\I}} y \rightarrow \exists t_2\in \C^\gamma_{\II}\, D_{\alpha + 1}(t_1,t_2)\big)\Big], \text{or}\nonumber\\
        &\exists y \in  \C^\gamma_{\I} \ \Big[ \forall x \in \C^\gamma_{\II} \ \lnot D_{\gamma}(y,x) \land \forall t_1 \in  \C^\gamma_{\I} \ \big(t_1 <_{\C^\gamma_{\I}} y \rightarrow \exists t_2\in \C^\gamma_{\II}\, D_{\gamma}(t_1,t_2)\big)\Big], \label{eqDefC1}
    \end{align} 
    depending on whether $\gamma$ is a successor ($\gamma = \alpha + 1$) or a limit ordinal. 
In both cases, by our analysis of the complexity of $D_{\gamma}$, $B$ turns out to be $\Sigma^0_{1 + \gamma + 2}$ (indeed $\Sigma^0_{1 + \alpha + 3}$ in the successor case).

Let us claim that this formula is indeed equivalent to $(C_{\I}1)$. Suppose that $(C_{\I}1)$ holds and $\Ps^{\gamma}(\mathbb{N})_{\M_{\I}} \not\subseteq \Ps^{\gamma}(\mathbb{N})_{\M_{\II}},$ so there is a $<_{\C^\gamma_{\I}}$-minimal element in $\C^\gamma_{\I} \setminus \C^\gamma_{\II}$ by Lemma \ref{LemmaEquivPC}, say $y \in \M_\I$. We assume without loss of generality that $y$ is chosen $<_L^{\M_\I}$-minimal (within its $\C^\gamma_{\I}$-equivalence class).
We claim that $y$ witnesses \eqref{eqDefC1}. 
For this, first notice that 
\[\forall t_1\in \C^\gamma_{\I}\, \big( t_1<_{\C^\gamma_{\I}} y \to \exists t_2\in \C^\gamma_{\II} D_\gamma(t_1,t_2)\big).\]
Indeed, by definition of $y$, for any such $t_1$ there exists a $t_2 \in C^{\gamma}_{\II}$ with $t_1 \simeq t_2$, hence the conclusion follows from lemma \ref{LemmaDummyInt}\eqref{LemmaDummyInt1}. Suppose towards a contradiction that $D_\gamma(y,x)$ holds for some $x \in \C^\gamma_{\II}$, that is, $D_{\gamma}$ would not be able to identify $y$. 

Below, for this proof only, we let $|t|$ denote the ordinal denoted by some $t \in \C^\gamma_{\I}$ or $t\in \C^\gamma_{\II}$.
Let $\beta <\gamma$ be least such that $y \in \mathcal{P}^{\beta+1}(\mathbb{N})$ and let $\kappa_{\I}$ be such that ${L_{|y|}^{\M_\I}} \models \kappa_{\I} = \omega_\beta$. By the minimality of $y$ and $\beta$, we have $\kappa_{\I} <^{\M_\I} |y|$ and thus there is $\kappa_{\II} \in \M_{\II}$ with $\kappa_{\I} \simeq \kappa_{\II}$. Thus, we have 
\[y \subset \mathcal{P}^\beta(\mathbb{N})_{L_{|y|}^{\M_\I}} \subseteq L^{\M_{\I}}_{\kappa_{\I}} \simeq L^{\M_{\II}}_{\kappa_{\II}} \in \M_{\II},\]
in particular $y \subset \mathcal{P}^\beta(\mathbb{N})_{L_{|y|}^{\M_\I}}$ follows from acceptability of $L$, that is Gödel's proof of GCH.
This means that $x$ must coincide with $y$ on that domain, with respect to the real isomorphism relation.
Indeed, let $x' := x \cap \mathcal{P}^\beta(\mathbb{N})_{L_{\kappa_{\II} +1}^{\M_\II}}$, we surely still have $D_\gamma(y,x')$ and thus as in Lemma \ref{LemmaDummyInt}\eqref{LemmaDummyInt2} we obtain $y\simeq x'$, which is a contradiction.

Suppose now that $\Ps^{\gamma}(\mathbb{N})_{\M_{\I}} \not\subseteq \Ps^{\gamma}(\mathbb{N})_{\M_{\II}}$ and that $y$ witnesses \eqref{eqDefC1}. We claim that $(C_{\I}1)$ holds. 
Recall that by hypothesis one of $\M_\I$ or $\M_{\II}$ is wellfounded. Clearly $(C_{\I}1)$ holds if $\M_\I$ is wellfounded. Similarly, if $\M_{\II}$ is wellfounded, then $y$ must be equal to $\mathsf{Ord}\cap \M_{\II}$, contradicting $(C_{\II}0)$. This proves the lemma.
\end{proof}

We can now return to the task of defining $\A^{\gamma + 1}$, which --in contrast to the situation in \cite{MS}-- must be split into four cases for technical reasons.

\begin{definition}\label{DefinitionCalA}
    The isomorphism between the greatest common part of $\M$ and $\N$, (i.e., $L_{\alpha}$) is coded by $\A^{\gamma + 1}$ which we define according to the following cases.
    \begin{enumerate}
    \item If $H_1 \coloneqq$``$\Ps^{\gamma}(\mathbb{N})_{\M_{\I}} = \Ps^{\gamma}(\mathbb{N})_{\M_{\II}}$'' holds--a $\Pi^0_{1 + \gamma + 1}$ condition--, we set $\A^{\gamma + 1} = \A_1^{\gamma + 1}$;
    \item If $H_2 \coloneqq$``$\Ps^{\gamma}(\mathbb{N})_{\M_{\I}} \subsetneq \Ps^{\gamma}(\mathbb{N})_{\M_{\II}}$'' holds--a $\Delta^0_{1 + \gamma + 2}$ condition--, then we know by condition $(C_{\II}1)$ that there exists a $<_{\C^{\gamma}_{\II}}$--least element of $\C^{\gamma}_{\II} \setminus \C^{\gamma}_{\I}$ coding an ordinal $\delta_{\II}$. Then we define $\A^{\gamma + 1}$ as:
    \begin{align*}
        \A^{\gamma+1}_2(w_1, w_2) \leftrightarrow \ 
        &\exists (t_1, t_2) \in [\Ps^{\gamma + 1}_{\M_{\I}}(\mathbb{N}) \times (\Ps^{\gamma + 1}_{\M_{\II}}(\mathbb{N}) \cap L^{\M_{\II}}_{\delta_{\II} + 1})] \\ 
        &\exists \beta_1, \beta_2, \ t_1 = L_{\beta_1}^{\M_{\I}}, \ t_2 = L_{\beta_2}^{\M_{\II}}, \text{and} \nonumber\\
        &(\M_{\I} \models t_1 \text{ codes } w_1 \land \M_{\II} \models t_2 \text{ codes } w_2 \land (t_1,t_2) \in D_{\gamma + 1});
    \end{align*}
    \item If $H_3 \coloneqq$``$\Ps^{\gamma}(\mathbb{N})_{\M_{\II}} \subsetneq \Ps^{\gamma}(\mathbb{N})_{\M_{\I}}$'' holds, then we define $\A^{\gamma + 1} = \A_3^{\gamma + 1}$ \emph{mutatis mutandis};
    \item If $H_4 \coloneqq$``$\Ps^{\gamma}(\mathbb{N})_{\M_{\I}} \not\subseteq \Ps^{\gamma}(\mathbb{N})_{\M_{\II}} \land \Ps^{\gamma}(\mathbb{N})_{\M_{\II}} \not\subseteq \Ps^{\gamma}(\mathbb{N})_{\M_{\I}}$'' holds--a $\Sigma^0_{1 + \gamma + 1}$ condition--, then we have minimal $\delta_{\I}$ and $\delta_{\II}$ as in the respective two preceding cases. Then we define $\A^{\gamma + 1}$ as \begin{align*}
        \A^{\gamma+1}_4(w_1, w_2) \leftrightarrow \ 
        &\exists (t_1, t_2) \in [(\Ps^{\gamma + 1}_{\M_{\I}}(\mathbb{N}) \cap L^{\M_{\II}}_{\delta_{\I} + 1}) \times (\Ps^{\gamma + 1}_{\M_{\II}}(\mathbb{N}) \cap L^{\M_{\II}}_{\delta_{\II} +1})] \\ 
        &\exists \beta_1, \beta_2, \ t_1 = L_{\beta_1}^{\M_{\I}}, \ t_2 = L_{\beta_2}^{\M_{\II}}, \text{and} \nonumber\\
        &(\M_{\I} \models t_1 \text{ codes } w_1 \land \M_{\II} \models t_2 \text{ codes } w_2 \land (t_1,t_2) \in D_{\gamma + 1}).
    \end{align*}
\end{enumerate}
    \label{PageFourCases}
\end{definition}

\begin{figure}
    \centering
    \begin{tikzpicture}
        \draw[thick] (0,-1) -- (0,3) node[above] {$\mathcal M = L_{\alpha_n^{\gamma}}$};
        \draw (-.1,0) -- (.1,0) node (node) {};
        \node[left] at ($(node) - (.2,0)$) {$\A^{\gamma + 1} \cong L_\alpha$};
        \draw (-.1,-0.5) -- (.1,-0.5) {};
        \node[right] at ($(node) - (0,0.5)$) {$\omega_{\gamma}^{\M} \cong \omega_{\gamma}^{\N}$};
        \draw[thick] (0,0) arc (180:90:2.5) node[right] {$\mathcal N$};
    \end{tikzpicture}
    \caption{A typical situation in the game of $\KP^{\gamma}_n$, for the case $H_1$.}
    \label{typicalGame}
\end{figure}
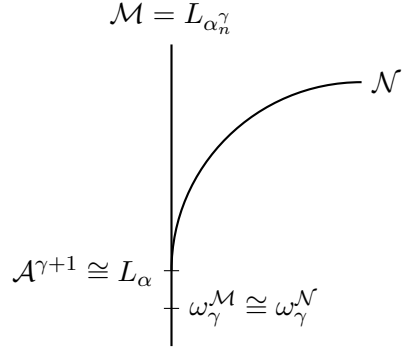

Figure~\ref{typicalGame} depicts the situation we have so far, with $\A^{\gamma + 1}$ coding the wellfounded part of $\N$, $L_{\alpha}$ for some $\mathrm{wfo}(\N) =: \alpha$.

The proofs of lemma~\ref{Dummy_intersect} and lemma~\ref{LemmaC1Formalization} show that this definition really describe the isomorphism relation on the common wellfounded part. That is, they showed the following lemma.
\begin{lemma}
    In a play satisfying $(C_{\I}0), (C_{\II}0)$ and $(C_{\I}1), (C_{\II}1)$ and in which one of $\M_{\I}$ or $\M_{\II}$ is wellfounded. Let $L_{\alpha}$ be the wellfounded part of $\N$ and $i \in \{1,2,3,4\}$ be such that $H_i$ holds, then $\forall w_1,w_2, \ w_1 \simeq w_2 \in L_{\alpha} \leftrightarrow \A_i(w_1,w_2)$.
\end{lemma}

The remaining conditions $C_{::}k$ (for $2 \leq k \leq n-1$) will always involve $\A^{\gamma + 1}$. Below, we define conditions $C_{::}k$ involving $\A^{\gamma + 1}$. Strictly speaking, $C_{::}k$ will really be a conjunction of four conditions of the form ``$H_i \rightarrow C_{::}k(\A^{\gamma + 1}_i)$'', with $i \in \{1,2,3,4\}$, distinguishing among the four cases involved in the definition of $\A^{\gamma + 1}$, which we denote by $H_i$. 
It will be useful to think of $\A^{\gamma + 1}$ as a $\Sigma^0_{1 + \gamma + 1}$ set. Although it is not clear that we can do so, we shall see below (see Lemma \ref{ComplexityC2}) that this is possible.

The remaining conditions in the game follow the dynamic of the proof of \cite{MS}. Here, they must be stated by cases in terms of our new definition of $\A^{\gamma + 1}$, which requires extra work and \textit{only makes sense if conditions $(C_{\I}0)$ and $(C_{\II}0)$, $(C_{\I}1)$, and $(C_{\II}1)$ all hold}, and we must verify that these can be stated with the right complexity, making use of our definition of $\mathcal{A}^{\gamma+1}$.

We intend to define each further conditions according to the following principles. \textit{If all the previous conditions hold}, they assert that the models have no infinite descending sequence of a certain kind. We will look for $\Sigma_n$ definable non-standard codes to constitute the descending sequence under discussion. 
If instead they are no such illfounded collection of non-standard codes, we will be able to infer properties on the models. This allows us to pursue our inquiry for the illfounded model in order to satisfy lemma~\ref{gamedef}.

We start our chase of the illfounded structure by analysing the collection of $\Sigma_1$ formulas that hold in one model and not in the other.
To this aim, we define the following classes.
For $i = 1, 2, 3, 4$, we put
\begin{align*}
    (W_{\M_{\I},1})_i =  \Big\{ \beta \in Ord^{\M_{\I}} \mid \exists (x_1,x_2) \in \A^{\gamma + 1}_i,& \phi \in \Delta_0, 
    \big[(\exists z \in L^{\M_{\I}}_{\beta} \ \M_{\I} \models \phi(z,x_1)) \land \\&(\M_{\II} \models \lnot\exists y \phi(y,x_2))\big]  \Big\}.
\end{align*}
The class $W_{\M_{\II},1}$ is defined \emph{mutatis mutandis}. While we write $W_{\M_{::},1}$, to abuse notation, we actually mean $(W_{\M_{::},1})_i$ depending on the case we fall into.
The next conditions are then 
\begin{align*}
    &(C_{\I}2) : \qquad W_{\M_{\I},1} \text{ has a least element or is empty}.  
    \\&(C_{\II}2) : \qquad W_{\M_{\II},1} \text{ has a least element or is empty}.
\end{align*}

\begin{lemma}\label{ComplexityC2}
The condition $(C_{::}2)$ is $\Sigma^0_{1 + \gamma + 2}$.
\end{lemma} 
\proof
Let us consider condition $(C_{\I}2)$; we can treat the other one similarly. The idea is to write down $(C_{\I}2)$ as a disjunction of four formulas, each treating one of the possibilities laid out in Definition \ref{DefinitionCalA}.
We say that $(C_{\I}2)$ holds if one of the following holds:
\begin{enumerate}
\item $\Ps^{\gamma}(\mathbb{N})_{\M_{\I}} = \Ps^{\gamma}(\mathbb{N})_{\M_{\II}}$ and $(W_{\mathcal{M}_{\I},1})_1$ is empty or has a least element. 
\item $\Ps^{\gamma}(\mathbb{N})_{\M_{\I}} \subsetneq \Ps^{\gamma}(\mathbb{N})_{\M_{\II}}$ and there exists some $\delta = \delta_{\II}$ such that the following hold:
\begin{enumerate}
\item $\delta_{\II}$ is $<_L^{\mathcal{M}_{\II}}$-least such that there is an element of $\Ps^{\gamma}(\mathbb{N})_{\M_{\II}}\setminus \Ps^{\gamma}(\mathbb{N})_{\M_{\I}}$ definable over $L_{\delta_{\II}}^{\mathcal{M}_{\II}}$, and 
\item $(W_{\mathcal{M}_{\I},1})_2$ is empty or has a least element. 
\end{enumerate}
\item $\Ps^{\gamma}(\mathbb{N})_{\M_{\II}} \subsetneq \Ps^{\gamma}(\mathbb{N})_{\M_{\I}}$ and there exists some $\delta = \delta_{\I}$  such that the following hold:
\begin{enumerate}
\item $\delta_{\I}$ is $<_L^{\mathcal{M}_{\I}}$-least such that there is an element of $\Ps^{\gamma}(\mathbb{N})_{\M_{\I}}\setminus \Ps^{\gamma}(\mathbb{N})_{\M_{\II}}$ definable over $L_{\delta_{\I}}^{\mathcal{M}_{\I}}$,
\item $(W_{\mathcal{M}_{\I},1})_3$ is empty or has a least element. 
\end{enumerate}
\item Both $\Ps^{\gamma}(\mathbb{N})_{\M_{\I}} \not\subseteq \Ps^{\gamma}(\mathbb{N})_{\M_{\II}}$ and $\Ps^{\gamma}(\mathbb{N})_{\M_{\II}} \not\subseteq \Ps^{\gamma}(\mathbb{N})_{\M_{\I}}$, and there exist ordinals $\delta_{\I}$ and $\delta_{\II}$ for which the following hold:
\begin{enumerate}
\item $\delta_{\I}$ is $<_L^{\mathcal{M}_{\I}}$-least such that there is an element of $\Ps^{\gamma}(\mathbb{N})_{\M_{\I}}\setminus \Ps^{\gamma}(\mathbb{N})_{\M_{\II}}$ definable over $L_{\delta_{\I}}^{\mathcal{M}_{\I}}$,
\item $\delta_{\II}$ is $<_L^{\mathcal{M}_{\II}}$-least such that there is an element of $\Ps^{\gamma}(\mathbb{N})_{\M_{\II}}\setminus \Ps^{\gamma}(\mathbb{N})_{\M_{\I}}$ definable over $L_{\delta_{\II}}^{\mathcal{M}_{\II}}$,
\item $(W_{\mathcal{M}_{\I},1})_4$ is empty or has a least element. 
\end{enumerate}
\end{enumerate}
Here, recall that each $(W_{\mathcal{M}_{\I},1})_i$ is defined using $\A^{\gamma+1}_i$, and the ordinals $\delta_{\I}$ and $\delta_{\II}$ may appear in the definition. 

Now, observe that precisely one of the hypotheses in these four alternatives holds, and each of these is $\Delta^0_{1+\gamma+2}$, as pointed out in Definition \ref{DefinitionCalA}. Moreover, by directly inspecting the definitions, we see that each of the alternatives above has a $\Sigma^0_{1+\gamma+2}$ definition.  For instance, let us consider the second one, which begins with an existential quantification over $\delta = \delta_{\II}$. The first requirement of $\delta$ is that it is $<_L^{\mathcal{M}_{\II}}$-least such that there is an element of $\Ps^{\gamma}(\mathbb{N})_{\M_{\II}}\setminus \Ps^{\gamma}(\mathbb{N})_{\M_{\I}}$ definable over $L_{\delta_{\II}}^{\mathcal{M}_{\II}}$. This has the form
\begin{align} 
\exists x& \in L_{\delta_{\II}+1}^{\mathcal{M}_{\II}} \cap \Ps^{\gamma}(\mathbb{N})_{\M_{\II}}\, 
\forall y\in \Ps^{\gamma}(\mathbb{N})_{\M_{\I}}\, (x,y) \not \in D_\gamma \nonumber\\
& \wedge 
\forall \eta<\delta_{\II}\forall x \in L_{\eta+1}^{\mathcal{M}_{\II}} \cap \Ps^{\gamma}(\mathbb{N})_{\M_{\II}}\, \exists y\in \Ps^{\gamma}(\mathbb{N})_{\M_{\I}}\, (x,y)  \in D_\gamma. \label{CompCI2Case2a}
\end{align} 
Recalling that $D_\gamma$ is of complexity $\Sigma^0_\gamma$ if $\gamma$ is a limit, or $\Pi^0_{1 + \gamma -1}$ if $\gamma$ is a successor, in both cases we see that  \eqref{CompCI2Case2a} has complexity $\Sigma^0_{1+\gamma+2}$. The second requirement on $\delta_{\II}$ is that $(W_{\mathcal{M}_{\I},1})_2$ is empty or has a least element. Inspecting the definition of $\A^{\gamma+1}_2$in Definition \ref{DefinitionCalA}, we see that it is $\Sigma^0_{1+\gamma+1}$ (with $\delta$ as a free variable), and thus inspecting the definition of $(W_{\mathcal{M}_{\I},1})_2$ we see that it is also $\Sigma^0_{1+\gamma+1}$. Hence, the assertion that  $(W_{\mathcal{M}_{\I},1})_2$ is empty is $\Pi^0_{1+\gamma+1}$ and the assertion that $(W_{\mathcal{M}_{\I},1})_2$ has a least element is $\Sigma^0_{1+\gamma+2}$. We conclude that the second clause of $(C_{\I}2)$ is $\Sigma^0_{1+\gamma+2}$. The others are treated similarly.
Since $(C_{\I}2)$ is the disjunction of the four clauses, it is itself also $\Sigma^0_{1+\gamma+2}$.
\endproof 

From now on, we abuse notation by simply writing $\A^{\gamma+1}$, $W_{\M_{\II},k}$, and so on. The rest of the proof is closer to that in \cite{MS}.
Now comes a key fact that will start our inductive search for the illfounded model. 

\begin{lemma}
    Suppose conditions $C_{::}i$ are satisfied for $i=0,1,2$. \label{star1}
    Then there is a $\beta \in Ord^{\N} \setminus \A^{\gamma + 1}$ such that $\A^{\gamma + 1} \preceq_1 L^{\N}_{\beta}$. 
\end{lemma}
\begin{proof}
This is proved by the same argument of \cite{MS}*{Claim 5.6}. We have included the proofs in the Appendix for the reader's convenience.
\end{proof}

The goal now is to define the remaining conditions such that if all the conditions $C_{::}i$ for 
$i = 0,1, \dots k$ hold, then $\A^{\gamma + 1}$ codes an initial segment satisfying $\KP^{\gamma}_k$. 
However, by definition of $T^{\gamma}_n$, such a (strict) initial segment cannot be a model of $\KP^{\gamma}_n$ and thus one of the conditions we are about to define is doomed to fail. We will prove this by 
induction and to that aim need the following induction hypothesis. We want $ \wedge_{i=0}^k C_{::}i$ 
to imply the existence of $\beta_1$ and $\beta_2$ such that \begin{align*}
    (\star_k) (\beta_1,\beta_2) : \qquad \beta_1 \in Ord^{\M_{\I}} \setminus \A_{\I}^{\gamma + 1} \land 
    \M_{\I} &\models L_{\beta_1} \text{ satisfies } \KP^{\gamma}_{k-1} \land
    \\\beta_2 \in Ord^{\M_{\II}} \setminus \A_{\II}^{\gamma + 1} \land 
    \M_{\II} &\models L_{\beta_2} \text{ satisfies } \KP^{\gamma}_{k-1} \land
    \\ L^{\M_{\I}}_{\beta_1} &\equiv_{k, \A} L^{\M_{\II}}_{\beta_2}, 
\end{align*}
where $\equiv_{k,\A^{\gamma + 1}}$ is written for $\Sigma_k$ elementary equivalence, with parameters from $\A^{\gamma + 1}$ and $z \in \A_{\I}^{\gamma + 1} \leftrightarrow \exists w \ (z,w) \in \A^{\gamma + 1}$, a $\Sigma^0_{1 + \gamma + 1}$ property.
By lemma~\ref{star1}, we know that so far, such a pair of ordinals of the respective models 
satisfying $(\star_1)$ exists. Indeed, we can take $\beta_1 = \alpha \in \M$ (from $L_{\alpha}$) and 
$\beta_2 = \beta \in \N$ (given by the lemma). For the sake of definiteness, we will always 
assume that Player \I\ is playing the wellfounded model; the situation in which Player \II\ is playing the wellfounded model is parallel. 

\begin{definition}[$S_k$ formulae]
    We say that a formula of set theory is $S_k$ if it is a Boolean combination of formulae of the form 
    $(\forall x \in z) \ \psi(x, \bar{y})$ where $\bar{y}$ are free variables and $\psi$ is $\Sigma_k$.
\end{definition}
By $\equiv_{k, \A^{\gamma + 1}}$, we mean the isomorphism relation with respect to $\Sigma_k$ formulae with parameters in $\A^{\gamma + 1}$ (we remind that we no longer specify about the $\A^{\gamma + 1}_i$, but simply write $\A^{\gamma + 1}$).
\begin{lemma}
    If $L^{\M_{\I}}_{\beta_1} \equiv_{k, \A^{\gamma + 1}} L^{\M_{\II}}_{\beta_2}$, then $L^{\M_{\I}}_{\beta_1}$ 
    and $L^{\M_{\II}}_{\beta_2}$ satisfies the same $S_k$-sentences with parameters from $\A^{\gamma + 1}$ substituted 
    for the free variables $z$ and $\bar{y}$. 
\end{lemma}
\begin{proof}
    This is because $\A^{\gamma + 1}$ is transitive since then, given a formula of the form $(\forall x \in z) \ \psi(x, \bar{y})$ and $z, \bar{y} \in \A^{\gamma + 1}$ for any $x \in z$, the sentence $ \psi(x, \bar{y})$ is $\Sigma_k$ with parameters from $\A^{\gamma + 1}$. Then by definition of ``$\models$'' the claim follows easily.
\end{proof}

We are ready to move on and dive into the definitions of the generalized versions of 
the sets $W_{\M_{::},1}$. This time we search for non-standard $\Sigma_k$-definable subsets 
of each model ($1 < k$). Once again, either there are a lot of elements witnessing an infinite 
descending sequence and thus revealing the identity of the illfounded model, or their rarity implies 
the existence of ordinals satisfying $(\star_k)$.
\begin{align*}
    W^{ \beta_1, \beta_2}_{\M_{\I},k} = \qquad \Big\{ \beta \in \beta_1 \mid \exists (x_1,x_2) \in \A^{\gamma + 1},& \phi \in S_{k-1}, 
    \big[(\exists z \in L^{\M_{\I}}_{\beta} \ L^{\M_{\I}}_{\beta_1} \models \phi(z,x_1)) \land \\&(L^{\M_{\II}}_{\beta_1} \models \lnot\exists y \phi(y,x_2))\big]  \Big\},
\end{align*}
and $W^{ \beta_1, \beta_2}_{\M_{\II},k}$ is defined \emph{mutatis mutandis}.
Now let us define the remaining conditions involved in determining the 
winner of the game (as before we treat indexes $k > 1$). \begin{align*}
    (C_{\I}(1+k)) : \qquad \text{There exist }& \beta_1, \beta_2 \text{ such that } 
    (\star_{k-1})(\beta_1,\beta_2) 
    \\ &\land W^{ \beta_1, \beta_2}_{\M_{\I},k} \text{ has a least element or is empty}.  
    \\(C_{\II}(1+k)) : \qquad \text{There exist } \beta_1, \beta_2& \text{ such that } 
    (\star_{k-1})(\beta_1,\beta_2) 
    \\ &\land W^{ \beta_1, \beta_2}_{\M_{\II},k} \text{ has a least element or is empty}.
\end{align*}

\begin{lemma}
Conditions $(C_{\I}(1+k))$ and $(C_{\II}(1+k))$ are $\Sigma^0_{1 + \gamma + 2}$, for $1 < k$.
\end{lemma}
\begin{proof}
This can be seen by a direct inspection of the definition. We state the precise definitions similarly to lemma~\ref{ComplexityC2}, distinguishing four cases according to whether $\A^{\gamma+1}$ is defined as $\A^{\gamma+1}_1$, $\A^{\gamma+1}_2$, $\A^{\gamma+1}_3$, or $\A^{\gamma+1}_4$. Each of these cases produces a $\Sigma^0_{1 + \gamma + 2}$ definition.
\end{proof}

\begin{lemma}
    Suppose that $\beta_1, \beta_2$ satisfy $(\star_k)$. Then 
    \begin{enumerate}
        \item $L_{\alpha} \preceq_{k} L_{\beta_1}$ and $\A^{\gamma + 1} \preceq_k L_{\beta_2}$;
        \item $L_{\alpha} \models \KP^{\gamma}_{k+1}$;
        \item There exists a descending sequence of $\N$-ordinals $(\gamma_i)_{i<\omega}$ converging 
        down to $\mathrm{Ord}^{\A^{\gamma + 1}}$ such that $L_{\gamma} \preceq_{k} L_{\beta_2}$.
    \end{enumerate}\label{star_kCons}
\end{lemma}
\begin{proof}
This is proved by the argument of \cite{MS}*{Lemma 5.10}. We have included the proof in the Appendix for the reader's convenience. 
\end{proof}
\begin{lemma}
    For each a play of our game satisfying conditions $(C_{\I}i)$ and $(C_{\II}i)$ for all $0 \leq i < 1+k$, if $(C_{\I}(1+k))$ and $(C_{\II}(1+k))$ are satisfied as well, then there are $\beta_1$ and $\beta_2$ in $\M_{\I}$, $\M_{\II}$ satisfying ($\star_k$), for $k < n-1$. The proof also works for $k = n-1$, which leads to a contradiction showing that no play can satisfy all the conditions. \label{domino}
\end{lemma}
\begin{proof}
   This is proved by the same argument as \cite{MS}*{Lemma 5.10}. We have included the proof in the Appendix, for the reader's convenience.
\end{proof}

\begin{proof}[Proof of lemma~\ref{gamedef}]
    Taking the model produced by the theory of $L_{\alpha_n^{\gamma}}$, since $\mathrm{Ord}^{\A} = \alpha \in \M \models Th(L_{\alpha_n^{\gamma}})$, $\alpha$ cannot be $n$-admissible (we can suppose $C_{\II}0,1$ has not failed). So by lemmata~\ref{domino} and~\ref{star_kCons}, there is a $k < n$ such that either $C_{\I}(1+k)$ or $C_{\II}(1+k)$ fails. Suppose $C_{\II}(1+k)$ is the first condition to fail and that $\I$ wins the game. Since $\forall i < 1+k$, all the conditions $C_{::}(i)$ are satisfied, this failure means that $\M_{\II}$ is illfounded. An analogous argument works when $C_{\I}(1+k)$ is the first condition to fail and $\II$ thus wins the game.
\end{proof}

Now we will modify the game, in order to still satisfy lemma~\ref{gamedef}, but with a $(\Sigma^0_{1 + \gamma + 2})_{n}$ 
game that we will call $G^{\gamma}_n$. First, we need to check both $(C_{\I}0), (C_{\I}1)$ and $(C_{\II}0), (C_{\II}1)$ before to begin with the rest of the conditions if we want our game to work as intended. We first check them for $\II$ so that she loses when the models are equal. In particular, this crucially allows us to define $\A^{\gamma+1}$.
Then we must test conditions $(C_{::}(k+2))$ 
after both conditions $(C_{::}(k+1))$ has been eventually verified to be satisfied, as unveiled in Lemmata~\ref{domino} and~\ref{star_kCons}. 
On the other hand, since only the player playing $\N$ can lose by such a condition, the order of $(C_{\I}(k+2))$ 
and $C_{\II}(k+2)$ does not matter. 

From these observations, the $(\Sigma^0_{1 + \gamma + 2})_{n}$ game $G^{\gamma}_{n}$ is defined by the condition depicted on figure~\ref{GameG1n}.

\begin{figure}[h]
    \begin{framed}
    \centering
    \begin{tabular}{llllllllll}
        Even:\ & $(C_{\II}0, 1)$ &       & $ (C_{\II}2, 3)$ &       &          &$ (C_{\II}n)$   &           & & \\
            &       &       &       &       & $\cdots$ &           \\
        Odd:\ &       & $(C_{\I}0, 1, 2)$ &       & $ (C_{\I}3, 4)$ &          &           \\
    \end{tabular}
    
\end{framed} 
    \caption{The game $G^{\gamma}_{n}$ for $n$ even.}\label{GameG1n}
\end{figure}

In contrast to this result, notice that Hachtman~\cite{Hachtman} result completes our proof for the case for $n=1$. He indeed analyzed the reverse-mathematical strength of $\Sigma^0_{1 + \gamma+ 3}$--$\Det$ which he showed to be equivalent to the existence of a wellfounded model of $\Pi_1$--$\mathsf{RAP}_{\gamma}$ ($\Pi_1$-Reflection to Admissible Principles). We have just shown that the strength of the principles $(\Sigma^0_{1 + \gamma+ 3})_n$--$\Det$ is growing fast between $\Pi_1$--$\mathsf{RAP}_{\gamma}$ and $\Pi_1$--$\mathsf{RAP}_{\gamma + 1}$. 
In \S\ref{SectionUpperBound} we shall show that $\KP^{\gamma} + \Sigma_n$--\textsc{separation} does prove these determinacy principles, establishing upper bounds for them.

In the next subsection, we show that as in the paper of Hachtman~\cite{Hachtman} and Montalbán and Shore~\cite{MSCons}, these determinacy principles, in addition to a weak base theory, prove the existence of wellfounded models of $\KP^{\gamma}_n$.

\subsection{Existence of $\beta$-models from determinacy}\label{SubsectionBetaModels}

Given a theory  $T$ in the language of second-order arithmetic or set theory,  we write $\beta(T)$ for the statement ``For every $X \subseteq \mathbb{N}$, there exists a countably coded $\beta$-model $M \ni X$ such that $M \models T$.'' 

In this section, we shall simultaneously generalize Theorem \ref{TheoremMSBeta} and strengthen Theorem \ref{notmPigammaDet} by showing that $(\Sigma^0_{1+\gamma+2})_{n+1}$, in addition to being unprovable in $\KP^{\gamma}_{n+1}$, proves the consistency of the latter, and indeed the existence of a $\beta$-model of that theory.  

To prove the theorem, we first need an appropriate version of theorem~\ref{notmPigammaDet}, namely the object of the second paper of Montalbán and Shore on the subject (\cite{MSCons}*{1.8, 3.1}), but with our adjustments towards its generalization. 

For the case $n=1$, we refer once again to the work of Hachtman~\cite{Hachtman}.

\begin{theorem}\label{beta_inequality}
Let $2 \leq n < \omega$, and let $\alpha^{\gamma,*}_{n}$ be the smallest limit of infinitely many ordinals $\alpha$ such that $L_{\alpha} \models \KP^{\gamma}_n$. Then,
\begin{enumerate}
\item\label{beta_inequality1} $(\Sigma^0_{1 + \gamma + 2})_n\text{--}\Det \vdash (\alpha^{\gamma}_{n}$ exists$)$; but
\item\label{beta_inequality2} $L_{\alpha^{\gamma,*}_n}\not\models (\Sigma^0_{1 + \gamma + 2})_n\text{--}\Det$.
\end{enumerate} 
\end{theorem}

Essentially both points of theorem~\ref{beta_inequality} derive from the construction of the proof of theorem~\ref{notmPigammaDet}.
Therefore, even though most of the useful rules have already been discussed, we need to manipulate them with care like in~\cite{MSCons}. However, we shall need to point out some additional facts about the construction which were not needed in  \cite{MSCons} but which will be needed for \S\ref{SectionReflection}.

For clarity, we first sketch the strategy and then sketch the proofs of this offspring of our previous section.

The initial technique was to build a game so that a winning strategy would compute the theory of $L_{\alpha^{\gamma}_n}$ ($n \geq 2$). See figure~\ref{tab original game}. Then by Tarski's undefinability of truth, such a game can't be determined in the model $L_{\alpha^{\gamma}_n}$. 

\begin{figure}
    \centering
    \begin{tabular}{ c c|c } 
        & winning condition & winning condition \\
        & for \II & for \I \\
        \hline

        & $\lnot C_{\I}0 \lor$ & \\ 
        $A_0$ & $[C_{\II}0 \land C_{\II}1]$ & \\

        &  & $C_{\I}0$ \\
        $A_1$ &  & $\land C_{\I}1$ \\ 
        &  & $\land C_{\I}2$ \\ 

        & $\vdots$ & \\
        & & $\vdots$ \\

        & $(C_{\II}1 + (2j - 1))$ & \\
        $A_{2j}$ & $\land (C_{\II}(1 + 2j))$ & \\

        & & $(C_{\I}(1 + 2j))$ \\
        $A_{2j + 1}$ & & $\land (C_{\I}1 + (2j + 1))$ \\

        & $\vdots$ & \\
        & & $\vdots$ \\

        $A_{n-1}$& $(C_{\II}n)$ & \\

        \hline
    \end{tabular}
    \caption{The game of theorem~\ref{notmPigammaDet} for even $n$}\label{tab original game}
\end{figure}

Now, we build a $(\Sigma^0_{1 + \gamma + 2})_n$ game so that its determinacy implies $\alpha^{\gamma}_n$ to exist.

We proceed by contradiction. 
We show that if $\alpha^{\gamma}_n$ does not exist, we can construct a $(\Sigma^0_{1 + \gamma + 2})_n$ game  that cannot be determined. 
By ``$\alpha^{\gamma}_n$ does not exist'', we mean there are no ordinal $\beta$ so that $L_{\beta} \models \KP^{\gamma}_n$. 
In particular, from a real-world perspective, every wellorder is embeddable in an initial segment of $\alpha^{\gamma}_n$. 

Here’s what we want to unfold: if a winning strategy belongs to some $L_{\alpha}$ ($\alpha \in Ord$ ``$< \alpha_n^{\gamma}$''), then we want to prove that $Th(L_{\alpha})$ is an appropriate counter strategy to this allegedly ``winning'' strategy. 
Hence, there can be no winning strategy. 
This will follow in particular because $Th(L_{\alpha})$ cannot be computable from such a winning strategy, echoing the insights of theorem~\ref{notmPigammaDet}.

Here, by supposing $\alpha^{\gamma}_n$ does not exist, we uncover two pivotal outcomes: first, that the set of $\alpha$ so that every element of $L_{\alpha}$ is definable from $\Ps^{\gamma}({\mathbb{N}})^{L_{\alpha}}$ is unbounded; second, that it is sufficient to use a $(\Sigma^0_{1 + \gamma + 2})_n$ winning condition to produce such an undetermined game. 
The conditions that this game should satisfy are as follows. 
\begin{enumerate}
    \item Each player will have to play complete and consistent theories extending some theory $\bar{T}^{\gamma}$, leading to $\omega$--term-models;
    \item If $\M_{\I}= \M_{\II}$, then \I\ wins;
    \item If $\M_{\I}$ is wellfounded, then \I\ wins if it is an extension of the wellfounded part of $\M_{\II}$;
    \item Vice-versa for $\II$.
\end{enumerate}
Thus, we will have to use the same techniques to compare such models as in the previous section.

\begin{proof}[Sketch of the proof of theorem~\ref{beta_inequality}.\ref{beta_inequality1}]
Let $\bar{T}^{\gamma}$ be the theory \begin{align*}
    \KP^{\gamma}+ V = L + \forall \beta \in\Ord (L_{\beta} \not \models \KP^{\gamma}_{\infty}). 
\end{align*}
The last condition is equivalent to asking for $L_{\delta}$ to be injectable in its highest cardinal $\omega_{\alpha}$ for $\alpha \leq \gamma$, provably in $L_{\delta + 1}$, for every ordinal $\delta$ (see lemma~\ref{cardinalityKPL}).

First let us prove that such a game cannot be determined if $\alpha^{\gamma}_n$ does not exist. 
Set \begin{align*}
    Y = \{\alpha \mid L_{\alpha} \models \bar{T}^{\gamma} \text{ and every member of $L_{\alpha}$ is definable in $L_{\alpha}$} \}.
\end{align*}
Observe that if $\alpha^{\gamma}_n$ does not exist, this set is unbounded. For otherwise, put $\delta = \sup Y$, $\alpha$ the least admissible ordinal greater than $\delta$ and $\M$ the elementary sub-model of $L_{\alpha}$ containing its definable elements. Since $\alpha^{\gamma}_n$ does not exist, by using lemma~\ref{cardinalityKPL}, we show that $\delta+ 1 \subseteq \M$ and so does its Mostowski collapse: Moreover it is admissible, so it is $L_{\alpha}$, showing that $\alpha \in Y$ for the desired contradiction.

Suppose $\I$ has a winning strategy $\sigma$. Consider $\alpha \in Y$ so that $\sigma \in L_{\alpha}$. We can check that if $\II$ plays $Th(L_{\alpha})$ against $\sigma$, she wins. The situation is symmetric when we suppose $\II$ has a winning strategy (cf.~also the proof in \cite{MSCons}).

We define the following $\Sigma^0_{1 + \gamma + 2}$ conditions.
\begin{align*}
    &(C_{\I}0): \quad \M_{\I} \models \bar{T}^{\gamma} \land \M_{\I} \text{ is an $\omega$--model } \\
    \\
    &(C_{\I}1): \quad (\Ps^{\gamma}(\mathbb{N})_{\M_{\I}} \setminus \Ps^{\gamma}(\mathbb{N})_{\M_{\II}} \neq \emptyset) 
    \\ & \qquad  \qquad \rightarrow \Ps^{\gamma}(\mathbb{N})_{\M_{\I}} \setminus \Ps^{\gamma}(\mathbb{N})_{\M_{\II}} \text{ has a $<^{\M_{\I}}_{L}$--least element};\\
    \\
    &(C_{\I}1\text{new}): (\Ps^{\gamma+1}(\omega)_{\M_{\II}} \subseteq \Ps^{\gamma+1}(\omega)_{\M_{\I}}) \rightarrow \exists \beta \in Ord^{\M_{\I}} \setminus \A^{\gamma + 1} \ (\M_{\II} \preceq_n L_{\beta}^{\M_{\I}});\\
    \\
    &(C_{\I}2) : \quad W_{\M_{\I},1} \text{ has a least element or is empty}.\\
    \\
    &(C_{\I}(1+k)\infty) : \text{There exist } \beta_1, \beta_2 \text{ such that } 
    (\star_{k-1})(\beta_1,\beta_2) 
    \\ &\qquad \qquad \qquad \land W^{ \beta_1, \beta_2}_{\M_{\I},k} \text{ has a least element or is empty},
\end{align*}
for $k > 1$. Here, the classes $\A^{\gamma+1}$ and $W_{\M_{\I},1}$ are defined as in the proof of theorem~\ref{notmPigammaDet} in a $\Sigma^0_{1 + \gamma + 1}$ way. The conditions $(C_{\I}(1+k)\infty)$ are the same as the ones of theorem~\ref{notmPigammaDet}, but we allow $\beta_1, \beta_2$ to take the value $\infty$, which we interpret as the class of the ordinals of the models and where $L_{\infty}$ is the entire model. This is necessary since at early stages of checking the winning conditions of the game, we can't assume yet that the models are incomparable (see figure~\ref{tab game beta}). The conditions $C_{\II}$ are defined dually. 

\begin{lemma}\label{LemmaBetaModelsComplexityCInew}
Each of conditions $(C_{\I}i)$, $(C_{\II}i)$, $(C_{\I}1\text{new})$, and $(C_{\II}1\text{new})$ is $\Sigma^0_{1+\gamma+2}$.
\end{lemma}
\proof
This mostly follows from the proof of Theorem~\ref{notmPigammaDet}, as most of the conditions are the same. Let us verify condition $(C_{\I}1\text{new})$. 

As in Lemma \ref{C0}, the formula $\Ps^{\gamma+1}(\omega)_{\M_{\II}} \not\subseteq \Ps^{\gamma+1}(\omega)_{\M_{\I}}$ is $\Sigma^0_{1+\gamma+2}$. Since $\A^{\gamma+1}$ is $\Sigma^0_{1+\gamma+1}$, the consequent of the implication in $(C_{\I}1\text{new})$ is also $\Sigma^0_{1+\gamma+2}$, as desired.
\endproof

We organize the rules of the game as depicted in figure~\ref{tab game beta}. The proof that this game satisfies the four conditions above is a straightforward adaptation of the argument in~\cite{MSCons} using the methods in the proof of Theorem~\ref{notmPigammaDet}. That is, we break down the situation into four scenarios; $\M_{\I} = \M_{\II}$, $\M_{\I} \subsetneq \M_{\II}$, $\M_{\II} \subsetneq \M_{\I}$ and $\M_{\I}$ and $\M_{\II}$ are incomparable. We remind here how the game works. We emphasize that to satisfy the specifications of the game, we can restrict our focus in the cases were at least one of the models is wellfounded. Then:

\begin{enumerate}
    \item Condition $1$ is guaranteed by $C0$, and we add to it the requirement $\M_{\I} \neq \M_{\II}$ as a disqualifying rule for $\II$.
    \item Verifying $C1$ allows us to formalize $\A^{\gamma+1}$  appropriately and with the right descriptive complexity.
    \item Under the assumption that $\alpha^{\gamma}_n$ does not exist, we can show the following. In the scenario $\M_{\II} \subsetneq \M_{\I}$ and $\M_{\I}$ is wellfounded then $C_{\I}1$new fails if and only if $\M_{\II}$ is the wellfounded part of $\M_{\I}$, thus disqualifying Player \I\, and vice versa for the scenario $\M_{\I} \subsetneq \M_{\II}$ (and $\M_{\II}$ is wellfounded).
    \item As soon as we have ruled out each of these scenarios, we ensure that if $\M_{\II} \subsetneq \M_{\I}$  but $\M_{\II}$ is not the wellfounded part of $\M_{\I}$, then Player $\II$ loses, forcing alternatively $\I$ and $\II$ to produce incomparable models.
    \item Under the assumption that $\alpha^{\gamma}_n$ does not exist, the remaining conditions ensure to eliminate the ill-founded models, since this is the eliminating criterion in the case where they play incomparable models.
\end{enumerate}
This method makes sure we meet the criteria of our game.
\end{proof}

\begin{remark}\label{RemarkWS}
Let us make an observation concerning the game just constructed in the proof of Theorem \ref{beta_inequality}\eqref{beta_inequality1}.
    Suppose one of the players has a winning strategy $\sigma$. By Shoenfield's absoluteness theorem, there is such a $\sigma$ in $L$, so  there is an ordinal $\alpha$ such that $\sigma \in L_{\alpha}$. Consider $\alpha^*$, the smallest ordinal greater than $\alpha$ such that $L_{\alpha^*}$ is a model of $\mathsf{KP}^\gamma$, and consider a run of the game in which the winning player plays according to $\sigma$ and the other player $\Ps$ plays $Th(L_{\alpha^*})$ which exists by $\Pi^1_1\CA$. 
    First $\Ps$ cannot lose because of $C0$, $C1$ or any $C(1+k)$ since $\Ps$ plays a wellfounded $\omega$-model of $\bar{T}^{\gamma}$. Also, since it is the least ordinal greater than $\alpha$ which satisfies $\bar{T}^\gamma$ and in particular it is a successor admissible, so it cannot have $\Sigma_n$ elementary extensions, and so $\Ps$ does not lose because of $C1new$. 
    In case $\Ps$ is Player II, Player I also cannot win due to the two models being equal, as then Player II would simply be copying Player I's move, and the play would be recursive in $\sigma$, which is impossible, as the play computes $Th(L_{\alpha^*})$.
    The winning strategy cannot achieve a victory by producing a strictly bigger model $\M$ of $\KP$, since then its theory would be definable from $L_{\alpha^* + 2} \subsetneq \M$, so the two models must be incompatible. 
    
The only possibility is thus that $\Ps$ lose via the $n$th empty condition at the end of our normal form construction for the winning $(\Sigma^{0}_{1+\gamma + 2})_n$ payoff set. 
But then $\star_{n-1}$ holds (see the proof of theorem~\ref{notmPigammaDet}) and by lemma~\ref{star_kCons}, $\A^{\gamma+1}$ refers to a level $L_{\delta}$ so that $L_{\delta} \models \KP^{\gamma}_n$. 

We have just proved the following. Suppose the game constructed in the previous proof is determined and let $\alpha$ be such that a winning strategy belongs to $L_\alpha$. Then, $\alpha^\gamma_n$ exists and is $\leq\alpha$.

This observation also answers the following natural question: in the proof of the theorem, we reached a contradiction from the assumption that $\alpha^\gamma_n$ exists. In the real world, where $\alpha^\gamma_n$ exists and determinacy holds, one of the two players has a winning strategy. The proof of the theorem did not reveal which one it is. We see thus that the real winner of the game depends on the parity of $n$.

Finally, we remark that the argument given above can be carried out from $\mathsf{RCA}_0$ alone using the assumption of $(\Sigma^{0}_{1+\gamma + 2})_n$--$\Det$, the key being that the players can indeed play models of $\bar T^\gamma$. Indeed, by Hachtman \cite{Hachtman}, $\Sigma^{0}_{1+\gamma + 2}$--$\Det$ implies that every real -- in particular the winning strategy $\sigma$ -- belongs to a $\beta$-model of $\mathsf{KP}^\gamma$. Thus, $\sigma$ is not a winning strategy vacuously.
\end{remark}

\begin{figure}
    \centering
    \begin{tabular}{ c c|c } 
        & winning condition & winning condition \\
        & for \II & for \I \\
        \hline

        & $\lnot C_{\I}0$ & \\ 
        $A_0$ & $\lor [C_{\II}0 \land C_{\II}1 ]$ & \\
        & $\land \M_{\I} \neq \M_{\II}$ & \\

        &  & $C_{\I}0$ \\
        &  & $\land C_{\I}1$ \\ 
        $A_1$ & & $\land C_{\I}1$new \\
        & & $\land C_{\I}2$ \\
        
        & $\M_{\II} \not\subseteq \M_{\I}$ &\\
        & $\land C_{\II}1$new &\\
        $A_2$ & $\land C_{\I}2$ & \\
        & $\land C_{\I}3\infty$ & \\

        &  & $\M_{\I} \not\subseteq \M_{\II}$ \\
        $A_3$ & & $\land C_{\I}3\infty$  \\
        & & $\land C_{\I}4\infty$ \\

        & $\vdots$ & \\
        & & $\vdots$ \\

        & $(C_{\II}1 + (2j - 1)\infty)$ & \\
        $A_{2j}$ & $\land (C_{\II}(1 + 2j)\infty)$ & \\

        & & $(C_{\I}(1 + 2j)\infty)$ \\
        $A_{2j + 1}$ & & $\land (C_{\I}1 + (2j + 1)\infty)$ \\

        & $\vdots$ & \\
        & & $\vdots$ \\

        $A_{n-1}$& $(C_{\II}n)$ & \\

        \hline
    \end{tabular}
    \caption{An undetermined game when $\alpha^{\gamma}_{n}$ does not exist, for even $n$. The first player to fail one of the requirements loses the game.}\label{tab game beta}
\end{figure}

We now turn to the proof of Theorem \ref{beta_inequality}\eqref{beta_inequality2}. We begin with the same strategy of Theorem~\ref{notmPigammaDet}. That is, we want a game so that \begin{enumerate}
    \item If \I\ plays $Th(L_{\alpha_n^{\gamma,*}})$, she wins;
    \item If \I\ doesn't but $\II$ does, then he wins.
\end{enumerate}
We just have to show we can modify the rules so that it works for $\alpha_n^{\gamma,*}$. 

\begin{proof}[Sketch of the proof of theorem~\ref{beta_inequality}.\ref{beta_inequality2}]
    Let $T^{\gamma, *}_n$ be the theory \begin{align*}
        V = L + \text{ there are unboundedly many ordinals $\alpha$ such that $L_{\alpha} \models \KP^{\gamma}_{n}$,} 
        \\\text{but only finitely many such $\alpha$ under any ordinal.} 
    \end{align*}
    In other words, these $\alpha$ are cofinal in the class of ordinals of any model of this theory,  
    We define the following $\Sigma^0_{1 + \gamma + 2}$ conditions.
    \begin{align*}
        &(C_{\I}0): \quad \M_{\I} \models T^{\gamma, *}_n \land \M_{\I} \text{ is an $\omega$--model };\\
        \\
        &(C_{\I}1): \quad (\Ps^{\gamma}(\mathbb{N})_{\M_{\I}} \setminus \Ps^{\gamma}(\mathbb{N})_{\M_{\II}} \neq \emptyset) \\
        & \qquad \qquad \qquad \rightarrow \Ps^{\gamma}(\mathbb{N})_{\M_{\I}} \setminus \Ps^{\gamma}(\mathbb{N})_{\M_{\II}} \text{ has a $<^{\M_{\I}}_{L}$--least element};\\
        \\
        &(C_{\I}1\text{new}*): \exists \beta \in Ord^{\M_{\I}} \setminus \A^{\gamma + 1} \ \exists \langle \gamma_1, \dots, \gamma_m \rangle \text{ an increasing sequence}\\
        & \qquad \qquad \qquad \text{such that $(\forall (i \leq m) \ \M_{\I} \models (\gamma_i \leq \beta \land L_{\gamma_i} \models \KP^{\gamma}_n))$ and} \\
        & \qquad \qquad \qquad \{[\forall \gamma \in \beta \ (\M_{\I} \models (L_{\gamma} \models \KP^{\gamma}_n) ) \\
        & \qquad \qquad \qquad \qquad \rightarrow (\exists (i \leq m) \ \gamma = \gamma_i)] \land [\forall (i \leq m) \ L_{\gamma_i} \neq \A^{\gamma + 1}]\}\\
        \\
        &(C_{\I}2) : \quad W_{\M_{\I},1} \text{ has a least element or is empty}.\\
        \\
        &(C_{\I}(1+k)\infty) : \text{There exist } \beta_1, \beta_2 \text{ such that } 
        (\star_{k-1})(\beta_1,\beta_2) 
        \\ &\qquad \qquad \qquad \land W^{ \beta_1, \beta_2}_{\M_{\I},k} \text{ has a least element or is empty},
    \end{align*}
    for $k > 1$. We mostly reuse the same conditions as before. The conditions $C_{\II}$ are defined dually. They are organized as depicted in figure~\ref{tab game limit} to constitute the rules of the game. As in Lemma \ref{LemmaBetaModelsComplexityCInew}, one sees that the conditions belong to the right complexity class.  The proof that this game satisfies the two conditions needed is the same as in~\cite{MSCons}, once more adapted as in Theorem~\ref{notmPigammaDet}.
    \end{proof}
    
    \begin{figure}
        \centering
        \begin{tabular}{ c c|c } 
            & winning condition & winning condition \\
            & for \II & for \I \\
            \hline
    
            & $\lnot C_{\I}0$ & \\ 
            $A_0$ & $\lor [C_{\II}0 \land C_{\II}1 ]$ & \\
            &$\land \M_{\I} \neq \M_{\II}$ &\\
    
            &  & $C_{\I}0 \land C_{\I}1$ \\ 
            $A_1$ & & $\land C_{\I}1$new* \\
            & & $\land C_{\I}2$ \\
            
            & $\land C_{\II}1$new* &\\
            $A_2$ & $\land C_{\I}2$ & \\
            & $\land C_{\I}3\infty$ & \\
    
            & & $\land C_{\I}3\infty$  \\
            $A_3$ & & $\land C_{\I}4\infty$ \\
    
            & $\vdots$ & \\
            & & $\vdots$ \\
    
            & $(C_{\II}1 + (2j - 1)\infty)$ & \\
            $A_{2j}$ & $\land (C_{\II}(1 + 2j)\infty)$ & \\
    
            & & $(C_{\I}(1 + 2j)\infty)$ \\
            $A_{2j + 1}$ & & $\land (C_{\I}1 + (2j + 1)\infty)$ \\
    
            & $\vdots$ & \\
            & & $\vdots$ \\
    
            $A_{n-1}$& $(C_{\II}n)$ & \\
    
            \hline
        \end{tabular}
        \caption{An undetermined game in $L_{\alpha^{\gamma, *}_{n}}$, for even $n$. Here, the first player to fail one of the conditions loses.}\label{tab game limit}
    \end{figure}  

We also need the following result, which is standard and easily established in a mild strengthening of $\Pi^1_1{-}\mathsf{CA}_0$.
\begin{lemma}
    If $T$ is a true $\Pi^1_3$ sentence, then there is an ordinal $\delta$ such that \begin{align*}
        L_{\delta} \models T \land \forall \gamma \ \exists \beta > \gamma \ \beta \text{ is admissible}
    \end{align*} 
    but $L_{\delta}$ is not $\Sigma_1$-admissible (and so $\mathbb{R} \cap L_{\delta} \not \models \Delta^1_2\CA$ and in particular $L_{\delta} \not \models \KP^{\gamma}$).~\label{Pi13limit}
\end{lemma}

We can now turn to the proof of Theorem \ref{generalised_inequalities} from the introduction.
\begin{proof}[Proof of theorem~\ref{generalised_inequalities}] Take $n \geq 2$ and $1 \leq \gamma < \omega_1^{\mathsf{CK}}$.
    First $$\beta(\KP^{\gamma} + \Sigma_{n}\text{--\textsc{separation}}) \to \beta((\Sigma^0_{1+\gamma+2})_{n}{-}\Det)$$ is immediate from theorem~\ref{mPi4Det}, applied inside $\beta$-models. Applying~\ref{Pi13limit} (valid in $\KP^{\gamma} + \Sigma_{n}\text{--\textsc{separation}}$), we can even prove $$\KP^{\gamma} + \Sigma_{n}\text{--\textsc{separation}} \to \beta((\Sigma^0_{1+\gamma+2})_{n}{-}\Det),$$ indeed if $\delta$ is a limit of admissible, $L_{\delta}$ is a $\beta$-model. Thus, the first application is not reversible since $\beta((\Sigma^0_{1+\gamma+2})_{n}{-}\Det) \not\to \beta(\beta((\Sigma^0_{1+\gamma+2})_{n}{-}\Det))$ by Gödel's incompleteness.
Clearly, $\KP^{\gamma}_{n+1} \vdash \beta((\Sigma^0_{1 + \gamma + 2})_n\text{-}\Det)$ is not reversible, as $\beta((\Sigma^0_{1 + \gamma + 2})_n\text{-}\Det)$ is $\Pi^1_3$. 

    Second, 
    $$\beta((\Sigma^0_{1+\gamma+2})_{n}\text{--}\Det) \to
        (\Sigma^0_{1 + \gamma+ 2})_{n}\text{--}\Det,$$
    is an application of $\Sigma^1_1$ correctness of $\beta$-models, and that this cannot be reversed is once more a consequence of Gödel's incompleteness theorem. 
    Finally, $$(\Sigma^0_{1 + \gamma+ 2})_{n}\text{--}\Det \to
    \beta(\KP^{\gamma}_{n})$$
    is proved using~Theorem \ref{beta_inequality}.\ref{beta_inequality1}. That it cannot be reversed follows from~Theorem \ref{beta_inequality}.\ref{beta_inequality2} since $$L_{\alpha_{n}^{\gamma,*}} \models \beta(\KP^{\gamma}_{n}) \land \lnot (\Sigma^0_{1 + \gamma+ 2})_{n}{-}\Det.$$
    Henceforth, we concluded the proof.
\end{proof}

\section{Provability of Determinacy Principles in Higher-Order Arithmetic} \label{SectionUpperBound}

In the preceding section, we showed that we need a stronger winning condition to devise the Friedman-style game compared to the case of second-order arithmetic. The following theorem shows that the preceding proof is optimal in terms of the separation scheme and that we can provide better upper bounds than the ones in $\Z_2$. This is our second main contribution.

\begin{lemma}
Suppose $1 \leq m < \omega$ and $1 \leq \gamma \leq \omega_{1}^{\mathsf{CK}}$.
    Then, the theory $\KP^{\gamma} +\ \Sigma_{m}$\text{--}\textsc{separation} proves that for every $X \subseteq \mathbb{N}$, \label{ShoenfieldAbsloutnessGamma}
    \begin{align*}
        L(X) \models \Sigma_{m}\text{--}\textsc{separation} +  \text{``}\Ps^{\gamma}(\omega)\text{ exists''}.
    \end{align*}
\end{lemma}
\begin{proof}
    The fact that $L(X) \models \Sigma_{m}\text{--}\textsc{separation}$ is standard and can be shown directly using the fact that the function $\alpha\mapsto L_\alpha(X)$ is a total, uniformly $\Sigma_1(X)$ function, provably in $\KP$. Now by $\Sigma_1$--\textsc{separation}, we have Axiom--$\beta$ (see e.g., Barwise~\cites{Barwise} or Simpson \cite{Simpson}), i.e., every wellorder $W$ is isomorphic to an ordinal $\mathrm{otp}(W)$. Let 
\[K = \sup\{\mathrm{otp}(W) \subseteq \Ps^{\xi}(\omega) \times \Ps^{\xi}(\omega) \mid W \text{ is a well order and } \xi<\gamma\},\] 
It follows from $\Sigma_1$--\textsc{separation} that $K$ exists and from $\Sigma_1$--\textsc{collection} that $\kappa :=\sup K$ exists. By definition, $\kappa \geq \omega_{\gamma}$, thus $L(X) \models |\kappa| \geq \omega_{\gamma}$, so $L(X) \models$``$\Ps^{\gamma}(\omega)$ exists'' by GCH.
\end{proof}
We remark that we used $\Sigma_1$--\textsc{separation} crucially in the preceding proof. This is necessary, as H. Friedman has shown that $\KP^1$ alone does not prove $L \models $ ``$\mathbb{R}$ exists.'' For strengthening of Friedman's result, we refer the reader to Mathias \cite{Ma01}. The main result of this section is:

\begin{theorem}
    For all $1 \leq m < \omega$ and $1 \leq \gamma < \omega_1^{\mathsf{CK}}$ \label{mPi4Det} 
    \begin{align*}
        \KP^{\gamma} + \textsc{$\Sigma_{m}$--separation} \vdash {(\Pi^0_{1 + \gamma + 2})}_m\text{--}\Det.
    \end{align*}
\end{theorem}

From now on we fix $1 \leq m < \omega$ and $1 \leq \gamma < \omega_1^{\mathsf{CK}}$ and reason in $\KP^{\gamma} + \Sigma_{m}$\text{--}\textsc{separation}. 
We shall also assume $V = L$. 
According to Lemma~\ref{ShoenfieldAbsloutnessGamma},
we still have access to $\KP^{\gamma} + \Sigma_{m}$\text{--}\textsc{separation}. Moreover, it is enough to prove determinacy under this assumption by Shoenfield absoluteness, since (lightface) determinacy for a Borel class such as $(\Pi^0_{1+\gamma+2})_m$ is $\Sigma^1_2$ (the proofs will relativize to arbitrary real parameters).
Hence, we shall henceforth assume $V = L$, and we shall thus have access to $\Sigma_m$--Collection, the Axiom of Choice, and the Generalized Continuum Hypothesis.

Firstly, we aim to simulate $(\Pi^0_{1+ \gamma + 2})_m$ games with moves in the natural numbers as a $(\Pi^0_3)_m$ game with moves of higher order. We use Martin's method of \textit{unravelling} from his proof of Borel determinacy (see~\cite{Kechris}, for example). Next, we will carefully implement the mechanism of Martin for the difference hierarchy to deal with the specific form of the payoff set. We will use the information that the unravelled game gives us about the original one to show that $\Sigma_m$--\textsc{separation} is enough to prove the existence of a winning strategy for it.

Let us first present Martin's notion of unravelling.
\begin{definition}[Covering of a tree]
    Let $T$ non-empty pruned tree on a set $A$. A \emph{covering} of $T$ is a triple $(\tilde{T}, \pi, \phi)$, where \begin{enumerate}
        \item $\tilde{T}$ is a non-empty pruned tree (on some $\tilde{A}$);
        \item $\pi: \tilde{T} \to T$ is monotone with $|\pi(s)| = |s|$, giving rise to a continuous function $\pi : [\tilde{T}] \to [T]$;
        \item $\phi$ maps strategies for Player \I\ (resp. \II) in $\tilde{T}$ to strategies for Player \I\ (resp. \II) in $T$, in such a way that $\phi(\tilde{\sigma})$ restricted to positions of length $\leq n$ depends only on $\tilde{\sigma}$ restricted to positions of length $\leq n$, for all $n$;
        \item If $\tilde{\sigma}$ is a strategy for \I\ (resp. \II) in $\tilde{T}$ and $x \in [\phi(\tilde{\sigma})] \subseteq [T]$, then there is an $\tilde{x} \in [\tilde{\sigma}]$ such that $\pi(\tilde{x}) = x$.
    \end{enumerate}
Moreover, for $k < \omega$,  we say that $(\tilde{T}, \pi, \phi)$ is a \emph{$k$-covering} if $T_{\mid 2k} = \tilde{T}_{\mid 2k}$ and $\pi_{\mid \tilde{T}_{\mid 2k}} = \mathrm{id}$. Finally, we will call $\pi^{-1}(X) \subseteq [\tilde{T}]$ the \emph{lift} of $X \subset [T]$.~\label{covering}
\end{definition}

\begin{remark}
    In particular, if $\tilde{U} \subseteq \tilde{T}$ is a subtree of $\tilde{T}$, then $\pi(\tilde{U}) \subseteq T$ is also a subtree of $T$ because of condition $2$.\label{locality of unravelling}
\end{remark}

Under the current hypotheses, we show the following.

\begin{theorem}[Martin \cite{BorelDetMartin}]\label{TheoremMartin}
    If $T$ is a non-empty and pruned tree on $\omega$ and $X \subseteq [T]$ is $\Pi^0_{1+ \gamma + 2}$, then for each $k < \omega$ there is a $k$-covering of $T$ with a $\Pi^0_3$ lift of $X$ and a tree $\tilde{T}$ on $\Ps^{\gamma}(T)$.
\end{theorem}

We need to review Martin's unravelling technique to prove it holds with our hypothesis and in our way of stating it. In his paper, he unravelled to clopen sets only, which is why we will repeat the proof.

A particular case of the theorem is when $T$ is countable, in which case $\tilde{T}$ can be encoded as a tree on $\Ps^{\gamma}(\omega)$.

Let $T$ be any non-empty pruned tree and let $X \subseteq [T]$ be closed and $k < \omega$. The game $G(X, T)$ has the form depicted in figure~\ref{OriginalGame}.

\begin{figure}
    \centering
    \begin{tabular}{lllllllllll}
        \I & $a_0$ &       & $a_2$ &       &          &$a_{2k-2}$   &           & $a_{2k}$ & & \\
            &       &       &       &       & $\cdots$ &           &           &  & & $\cdots$ \\
        \II &       & $a_1$ &       & $a_3$ &          &           &$a_{2k-1}$ & &$a_{2k+1}$ &\\
    
    \end{tabular}
    \caption{A closed countable game.}\label{OriginalGame}
\end{figure}

We define our auxiliary tree $\tilde{T}$ by showing how to play the game of the covering. We denote by $\SI(T)$ and $\QSI(T)$ the set of strategies and quasistrategies for $\I$ in the game tree $T$ and vice-versa for $\II$. First, $T_{\mid 2k} = \tilde{T}_{\mid 2k}$. After the move $a_{2k-1}$, Player \I\ has to play a pair $(a_{2k}, \Sigma_{\I})$, with $(a_n)_{n \leq 2k} \in T$ and $\Sigma_{\I} \in \QSI(T_{(a_n)_{n \leq 2k}})$ (Player $\II$ starts playing first in $\Sigma_{\I}$). Next Player $\II$ has two options. 

\paragraph*{Option $1$:}

Player II plays $(x_{2k+1},u)$, where $u \in T_{(a_n)_{n \leq 2k+1}}$ and $u \in (\Sigma_{\I})_{x_{2k+1}} \setminus (T_X)_{(a_n)_{n \leq 2k+1}}$. That is, a position consistent with the quasistrategy played by $I$, where she was the last one to play and so that all the possible infinite sequences to be formed from it while still playing consistently according to the quasistrategy will remain out of $X$.

If so, all the following moves $a_{2k+2}, a_{2k+3}, a_{2k+4}, \dots$ have to be consistent with $u$. We then depict the auxiliary game as in figure~\ref{AuxGame1}.

\paragraph*{Option $2$:}

Player II plays $(x_{2k+1}, \Sigma_{\II})$, where $\Sigma_{\II} \in \QSII((\Sigma_{\I})_{(x_{2k+1})})$ and $\Sigma_{\II} \subseteq (T_X)_{(a_n)_{n \leq 2k+1}}$. In other words, $\Sigma_{\II}$ only envisages moves in $\Sigma_{\I}$ that always lead to sequences in $X$.

If so, all the following moves $a_{2k+2}, a_{2k+3}, a_{2k+4}, \dots$ have to be consistent with $\Sigma_{\II}$. We then depict the auxiliary game as in figure~\ref{AuxGame2}.

In both cases $(a_n)_{n < \omega} \subset T$

\begin{figure}
    \centering
    \begin{tabular}{lllllllllll}
        \I & $a_0$ &       & $a_2$ &       &          &$a_{2k-2}$   &           & $(a_{2k}, \Sigma_{\I})$ & & \\
            &       &       &       &       & $\cdots$ &           &           &  & & $\cdots$ \\
        \II &       & $a_1$ &       & $a_3$ &          &           &$a_{2k-1}$ & &$(a_{2k+1},u)$ &\\
    
    \end{tabular}
    \caption{A clopen uncountable game (Option $1$).}\label{AuxGame1}
\end{figure}

\begin{figure}
    \centering
    \begin{tabular}{lllllllllll}
        \I & $a_0$ &       & $a_2$ &       &          &$a_{2k-2}$   &           & $(a_{2k}, \Sigma_{\I})$ & & \\
            &       &       &       &       & $\cdots$ &           &           &  & & $\cdots$ \\
        \II &       & $a_1$ &       & $a_3$ &          &           &$a_{2k-1}$ & &$(a_{2k+1}, \Sigma_{\II})$ &\\
    
    \end{tabular}
    \caption{A clopen uncountable game (Option $2$).}\label{AuxGame2}
\end{figure}

\begin{lemma}[unravelling lemma]\label{unravelling lemma}
    Let $T$ be a non-empty and pruned tree and let $X \subseteq [T]$ be closed. Suppose that $\Ps(T)$ exists. Then, for each $k < \omega$ there exists a $k$-covering of $T$ with a clopen lift of $X$.
\end{lemma}

\begin{proof}
    From Martin, see Appendix \ref{appendix unravelling}.
\end{proof}

Note that we crucially use open determinacy for sets of uncountable cardinality in the proof of Lemma \ref{unravelling lemma}, and this requires $\Sigma_1$--\textsc{separation}.

From the proof of Lemma~\ref{unravelling lemma} notice that there exists a winning strategy for one Player in $\tilde{T}$ if and only if there exists a winning strategy for that Player in $T$. This would be an element of $\Ps(T)$, while the former was a subset of it.

\begin{lemma}[Existence of inverse limits]\label{existence of inverse limits}
    Let $k < \omega$. Let $(T_{i+1}, \pi_{i+1}, \phi_{i+1})$ be a $(k+i)$-covering of $T_i$, $i = 0,1,2, \dots$ Then there is a pruned tree $T_{\infty}$ and $\pi_{\infty, i}, \phi_{\infty, i}$ such that this triple is a $(k+i)$-covering of $T_i$ and $\pi_{i+1} \circ \pi_{\infty, i+1} = \pi_{\infty, i}$, $\phi_{i+1} \circ \phi_{\infty, i+1} = \phi_{\infty, i}$.
\end{lemma}

\begin{proof}
From Martin, see Appendix \ref{appendix unravelling}.
\end{proof}

Figure~\ref{InvLimit} is sketching the situation in the last lemma.

\begin{proof}[Proof of Theorem \ref{TheoremMartin}]
    We prove the theorem by induction on $\gamma$. 
    First, suppose $\gamma= \alpha + 1$. 
    Observe that, for any natural number $k$, if a $k$-covering has some $\Pi^0_3$ lift of a set $X$, then it has a $\Sigma^0_3$ lift for its complementary set. 
    Let $(B_i)_{i < \omega}$ be an enumeration of all the $\Sigma^0_{1 + \alpha + 1}$ sets such that the $\Pi^0_{1 + \gamma + 2}$ set $X$ is formed by the intersection $\bigcap_{i < \omega} B_i$. 
    Let $\pi_0$ be the $k$ covering given by the induction hypothesis for the game $G(T, B_0)$ with a $\Pi^0_3$ lift in a game tree on $\Ps^{\alpha}(T)$, $T_0$. Note that we can apply the induction hypothesis because $\gamma = \alpha+1$ is a successor.
    We can write $T_0 \supseteq \pi_0^{-1}(B_0) = \bigcap_j \bigcup_k C^0_{\langle j,k \rangle}$, with each  $C^0_{\langle j,k \rangle}$ closed.
    
    Using Lemma~\ref{unravelling lemma}, we then inductively define $(k+n)$-coverings $\pi_{0;n+1}$ of $T_0^n$; $T_0^0 = T_0$ and $\pi_{0;n+1}: T_0^{n+1} \to T_0^{n}$ is defined as in lemma~\ref{unravelling lemma} to unravel $\pi^{-1}_{1;n}\circ \cdots \circ \pi^{-1}_{1;1}(C^0_n)$ into a clopen set, in $T_0^{n+1}$, a game tree on $\Ps^{\alpha + 1}(T)$. 
    Invoking Lemma~\ref{existence of inverse limits}, we then obtain a $k$-covering $(T^{\infty}_0, \pi_{0;\infty}, \phi_{0; \infty})$ of $T_0$ with a lift of $\pi_1^{-1}(B_0)$ that is $\Pi^0_{2}$. We rename $T_1 = T^{\infty}_0$.

    Now, let one inspect two properties of the covering from Lemma~\ref{unravelling lemma}. Firstly, the fact that the strategies played in the auxiliary game concern the future moves in the tree and secondly, the way they are stacked by Lemma~\ref{existence of inverse limits}. We thus may observe that $T_1$ is a tree on $\Ps^{\alpha+1}(T)$.

    To summarize we have a $k$-covering $\pi_1 = \pi_{0;\infty} \circ \pi_0: T_1 \to T$ that has a $\Sigma^0_{2}$ lift of $B_0$.

    By induction, we define through the same process $\pi_{i+1}: T_{i+1} \to T_i$ to be the $(k+i)$-covering with a $\Sigma^0_{2}$ lift of $\pi^{-1}_i \circ \cdots \circ \pi^{-1}_1(B_i)$ in a game tree on $\Ps^{\alpha+1}(T)$. 

    Invoking Lemma~\ref{existence of inverse limits}, we then obtain  a $k$-covering $(T_{\infty}, \pi_{\infty}, \phi_{\infty})$ of $T$ with a lift of $X$ that is $\Pi^0_{3}$. 
    Moreover, as before, we observe that $T_{\infty}$ is a tree on $\Ps^{\alpha+ 1}(T)$.
    
    Suppose that for $\gamma$ a limit ordinal, we proved the theorem for any $\alpha < \gamma$. 
    Since by hypothesis we enjoy the existence of $\bigcup_{\alpha< \gamma} \Ps^\alpha(\omega)$ we can thus unravel any $\Pi^0_{1 + \alpha + 2}$ game into a $\Pi^0_3$ set and actually into a clopen game since $\Ps^{1 + \alpha + 5}(\omega)$ exists. 
    Thus, by the method of inverse limits, we get covering of $\Pi^0_{\gamma}, \Pi^0_{\gamma+ 1}$ and $\Pi^0_{\gamma + 2}$ games with respectively $\Pi^0_{1}, \Pi^0_{2}$ and $\Pi^0_{3}$ lifts in game trees on $\Ps^\gamma(\omega)$. 
    In particular, the conclusion follows.
\end{proof}

\begin{figure}
    \centering
    \begin{tikzcd}[sep = large]
    T_0 & \arrow[l, "\pi_1", swap] T_1 & \arrow[l, "\pi_2", swap] T_2 & \cdots \arrow[l, "\pi_3", swap] & \arrow[l, "\pi_i", swap] T_i & T_{i+1} \arrow[l, "\pi_{i+1}", swap] & \arrow[l, "\pi_{i+2}", swap] \cdots
    \\
    \\ & & & T_{\infty} \arrow[llluu, "\pi_{\infty,0}", pos=0.7] \arrow[lluu, "\pi_{\infty,1}", pos=0.7, swap] \arrow[luu, "\pi_{\infty,2}", pos=0.7, swap] \arrow[ruu, "\pi_{\infty,i}", pos=0.7] \arrow[rruu, "\pi_{\infty, i+1}", pos=0.7, swap]
    \end{tikzcd}
    \caption{Composing Martin's unravelling.}\label{InvLimit}
\end{figure}
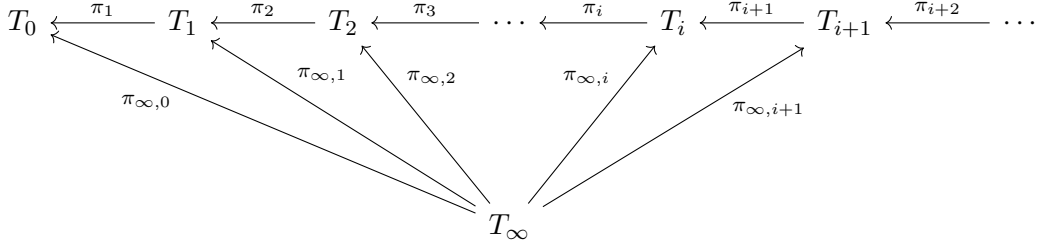

The following lemma follows immediately from the definition.
\begin{lemma}
    Take $\pi: T_0 \to T_1$, a $k$-covering of $T_1$ and $A_0, A_1, \dots A_{m-1} \subseteq [T_1]$ and $A_m = \emptyset$. Suppose $j = \mu \{i: x_1 \not\in A_i\}$ is odd. Then if $x_0 \in [T_0]$ is such that $\pi(x_0) = x_1$ we have that $k = \mu \{i: x_0 \not\in \pi^{-1}(A_i)\}$ is odd, indeed $k=j$. The result is the same if we change odd by even and $\not\in$ by $\in$ (and then $A_m = [T_1]$).
\end{lemma}
In other words, the lifting of the difference is the difference of the lifting.

We now begin working towards the proof of Theorem \ref{mPi4Det}.
In the following, we  consider sequences $s \in \omega^{\leq m}$ and a tree $T \subseteq [\Ps^{\gamma}(\omega)]^{< \omega}$, the covering of some tree $T^* \subseteq 2^{< \omega}$ with a $(\Pi^0_{1+ \gamma + 2})_m$ set $A^* \subseteq [T^*]$ and a $(\Pi^0_3)_m$ set $\pi^{-1}(A^*) \eqqcolon A \subseteq [T]$. 
We set up the construction of the $A_i$ for $0 \leq i < m$ as 
\begin{align*}
    A_i = \bigcap_{k < \omega} A_{i,k} \qquad \text{ and } \qquad A_{i, k} = \bigcup_{j < \omega} A_{i,k,j},
\end{align*} 
with $\Sigma^0_2$ sets $A_{i}$ and $\Pi^0_1$ sets $A_{i,k}$, with $A_{i,k} \subset [T] \subseteq [\Ps^{\gamma}(\omega)]^{\omega}$.

In the following definition, we will state the equivalent of Martin's property inside our current setting of $(2 + \gamma)$-th arithmetic. 
The goal of this property $P^{s}(T)$ is to give an approximate answer to the question ``Which Player wins the game $G(T, A)$'', by taking into account the specific structure of the set $A$. 
Here comes the difference in the case of second-order arithmetic. 
As before, asking whether Player \I\ has a winning strategy in $G(T, A)$ or not is a $\Sigma_1$ statement. However, by unravelling, this is equivalent to the existence of a winning strategy in $G(T^*, A^*)$. Since such a strategy is a countable object, its existence can be expressed in a $\Delta_0(\mathbb{R})$ way, and this will allow eliminating one quantifier.

Given $s$ we put $l \coloneqq m - |s|$.

\begin{definition}
    We define relations $P^s(T)$ by induction on $|s| \leq m$:
    \begin{enumerate}
        \item When $|s|=0$, $P^{\langle \ \rangle}(T)$ if and only if $\I$ (respectively, $\II$) has a winning strategy in $G(T^*,A^*)$ if  $l$ is even (respectively, odd).
        \item For $|s|=n+1$ and $l$ even, $P^{s}(T)$ if and only if there is a quasistrategy $U$ for Player \I\ in $T$ such 
        that \begin{align}
            [U] \subseteq A \cup A_{l, s(n)} \qquad \text{and} \qquad 
            P^{s\upharpoonright n}(U) \text{ fails}.\label{even}
        \end{align}
        \item For $|s|=n+1$ and $l$ odd, $P^{s}(T)$ if and only if there is a quasistrategy $U$ for Player \II\ in $T$ such 
        that \begin{align}
            [U] \subseteq \Bar{A} \cup A_{l, s(n)} \qquad \text{and} \qquad
            P^{s\upharpoonright n}(U) \text{ fails}.\label{odd}
        \end{align}  
    \end{enumerate}
    A quasistrategy $U$ \emph{witnesses} $P^{s}(T)$ if $U$ is as required in the appropriate clause.
    \label{P}
\end{definition}

\begin{figure}
\begin{center}
 \begin{tikzcd}
U \arrow[r, "\subseteq"] \arrow[d, "\pi", dashed] & T \arrow[d, "\pi"] \\
U^* \arrow[r, "\subseteq"]                        & T^*               
\end{tikzcd}
\end{center}
\caption{The tree $U^*$ in the definition of $P^{\langle\rangle}(T)$.}
\label{FigureTreeTU}
\end{figure}

We now show that the relation $P^s(T)$ just defined is of complexity $\Sigma_{|s|}$. This is one of the key steps which differs from the case of subsystems of $\mathsf{Z}_2$.
\begin{lemma}
Suppose $|s| \leq m$. Then, there is a $\Sigma_{|s|}$ formula $\phi(x_0, x_1, z_{i=0, \dots, |s|} ,y)$ such that for all $T$, $P^s(T)$ is equivalent to $\phi(T, A,(A_{s(i)})_{0 \leq i < |s|}, \mathbb{R})$.
\end{lemma}
\proof
Suppose without loss of generality that $m$ is even. The proof is induction on $|s|$. 
When $|s| = 0$, the formula $P^{\langle \ \rangle}(T)$ asserts that Player I has a winning strategy in $G(T,A)$. However, by Theorem \ref{TheoremMartin}, Player I has a winning strategy in $G(T,A)$ if and only if she has one in $G(T^*, A^*)$. Since $A^*\subseteq\mathbb{R}$ and $T^*$ is a tree on $\mathbb{N}$, such a strategy can be coded by a real number, if it exists. 
Now, we claim that the mapping 
\[(T,A) \mapsto (T^*,A^*)\]
is $\Delta_0$. To see this, observe that $T^*$ is obtained from $T$ by removing the auxiliary moves from the plays, and $A^*$ is obtained  as the set of all branches through $T^*$ induced by a branch in $A$ via the natural projection. Thus,  $P^{\langle \ \rangle}(T)$ can be stated as a $\Delta_0$ formula involving $T$, $A$ and $\mathbb{R}$ as parameters. This proves the claim for $|s| = 0$. 

When $|s|=n+1 \leq m$, $P^{s}(T)$ holds if and only if there exists some $U \subseteq T$ such that $[U]\subset A\cup A_{l,s(n)}$ (or $[U]\subset \bar A\cup A_{l,s(n)}$, according to the parity of $n$) and such that $P^{s\upharpoonright n}(U)$ fails. By induction hypothesis $P^{s\upharpoonright n}(U)$ is defined by a $\Sigma_{n}(U,A, (A_{s(i)})_{0 \leq i < |s|}, \mathbb{R})$ formula.  Noting that, by remark~\ref{locality of unravelling}, if $U \subseteq T$ is a subtree of $T$, then $\pi(U) = U^*$ is a subtree of $T^*$ (see figure~\ref{FigureTreeTU}), we see that $P^{s}(T)$ is $\Sigma_{n+1}$, as desired.

\endproof

\begin{definition}[Failure everywhere]
    We say that $P^s(T)$ \emph{fails everywhere} if $P^s(T_p)$ fails for every $p \in S$. This is a $\Pi_{|s|}$ 
    sentence.
\end{definition}

We now state 
Lemma \ref{failure} and Lemma \ref{binary} below, which are due to Martin. The fact that $\Sigma_m$-\textsc{separation} suffices for the proof is due to Montalb\'an and Shore. For the reader's convenience, we have included the proofs in Appendix~\ref{SectAppendixUpperBound}.

\begin{lemma}
If $P^s(T)$ fails, then there is a quasistrategy $W$ in $S$ such that $P^s(W)$ fails everywhere.\label{failure}
\end{lemma}

\begin{definition}[Strong witness]
    For $|s| = n+1$, $W$ \emph{strongly witnesses} $P^s(T)$ if, for all $p \in W$, $W_p$ witnesses $P^s(T_p)$, that 
    is, $W$ witnesses $P^s(T)$ and $P^{s\upharpoonright n}(W)$ fails everywhere. This is a $\Pi_{|s|-1}$ sentence.
\end{definition}

Lemma \ref{enpowerment} below is an immediate consequence of Lemma \ref{failure}.

\begin{lemma}
    If $P^s(T)$, then there is a $W$ that strongly witnesses it.\label{enpowerment}
\end{lemma}

\begin{lemma}
    If $|s| = n+1$, then at least one of $P^s(T)$ and $P^{s\upharpoonright n}(T)$ holds.\label{binary}
\end{lemma}

We are now ready to proceed to the proof of the main result of this section. 
\begin{proof}[Proof of Theorem\text{~}\ref{mPi4Det}]
Without loss of generality, we suppose that $m$ is odd and Player \II\ has no winning strategy in $G(T, A)$, that is $P^{\langle \ \rangle}(T)$ 
    fails. By Lemma\text{~}\ref{failure}, there is a quasistrategy $W^{\langle \ \rangle}$ that Player \I\ can follow such that 
    $P^{\langle \ \rangle}(W^{\langle \ \rangle})$ fails everywhere. We define a quasistrategy $U$ for Player \I\ in $W^{\langle \ \rangle}$ 
    by induction on $|p|$ for $p \in U$.
    
    To $\langle \ \rangle \in U$, we associate a quasistrategy $W^{\langle \ \rangle}$ such that
    $P^{\langle \ \rangle}(W^{\langle \ \rangle})$ fails everywhere. Suppose then $p \in U$, $|p| = j+1$ and $W^{p}$ have been 
    defined with $P^{\langle \ \rangle}(W^p)$ failing everywhere. The children $q$ of $p$ in $U$ are the same as those of $p$ in $W^p$. 
    Since $P^{\langle \ \rangle}(W^p)$ fails everywhere, so does $P^{\langle \ \rangle}(W^p_q)$, which by Lemma\text{~}\ref{binary} implies 
    $P^{\langle j \rangle}(W^p_q)$. Now we use of Lemma\text{~}\ref{enpowerment} to get the existence of a $W^q$ 
    that strongly witnesses it. To continue our induction, we have to choose such a $W^q$, which depends on the previously 
    chosen ones, we can do it the same way as exposed in the proof of the preceding lemma (with in principle lower complexity needed).

    Now we show that $U$ is a winning quasi-strategy, giving rise to a winning strategy for Player \I\ using the Axiom of Choice. Consider any play $x \in [U]$. By construction, for every $j$ 
    \begin{align*}
        x \in [W^{x[j+1]}] \qquad \text{and} \qquad W^{x[j+1]} \text{ witnesses } P^{\langle j \rangle}
        (W^{x[j] }_{x[j+1]}).
    \end{align*}
    By the first property of the witness, for every $j$ 
    \begin{align*}
        x \in A \cup A_{m-1, j} \quad (\text{resp. } \Bar{A} \cup A_{m-1, j}).
    \end{align*}
    As $\bigcap_{j < \omega} A_{m-1, j} = A_{m-1} \subseteq A$ 
    by definition, it follows that $U$ is winning for Player \I\ in $G(A,T)$, as desired. 
\end{proof}

\section{Reflection Principles in Higher-Order Arithmetic} \label{SectionReflection}

In this last section, our goal is to generalize the results of {Pacheco} and Yokoyama~\cite{pacheco} and answer the question posed at the end of \cite{pacheco}. 
Pacheco and Yokoyama \cite{pacheco} prove the equivalence between the assertion that \textit{all} (finite) Boolean combinations of $\Pi^0_3$ games are determined and a reflection principle for second-order arithmetic. Below, given a class a formulae $\Gamma$ and a theory $T$, the reflection principle $\Gamma\Rfc(T)$ is the sentence \begin{align*}
    \forall \phi \in \Gamma \ (\mathrm{Pr}_T(\lceil \phi \rceil) \rightarrow \mathrm{Tr}_{\Gamma}(\lceil \phi \rceil)). 
\end{align*} 
Using a partial truth-predicate, we let the models of the ambient theory decide the interpretation of non-standard formulae $\phi \in \Pi^1_n$ for possibly non-standard $n$, if such a truth-predicate exists. 
Over $\ACA$, this is equivalent to replacing $\mathrm{Tr}_{\Gamma}(\lceil \phi \rceil)$ by $\phi$ and thus only considering standard sentences when $\ACA \subseteq T$ is a finitely axiomatisable theory. However, we will show that these technicalities can be avoided.

\begin{theorem}[Pacheco and Yokoyama \cite{pacheco}]\label{TheoremPY}
Over $\RCA$,  ``$\forall n \ {(\mathbf{\Pi}^0_3)_n}\text{-}\Det$'' is \emph{equivalent} to $\Pi^1_3\Rfc(\Z_2)$.
\end{theorem}
Notice that here we use the boldface class $\mathbf{\Pi}^0_3$,  allowing any real parameters.
Their proof makes use of Theorem \ref{TheoremMSBeta}; specifically, of the fact that over $\RCA$ we have
\begin{align*}
\Pi^1_{n+2}\CA  \to  (\mathbf{\Pi}^0_3)_n\text{-}\Det \to \beta(\Delta^1_{n+2}\CA),
\end{align*}
for $n \geq 1$. Pacheco and Yokoyama \cite{pacheco} ask whether there are analogues of Theorem \ref{TheoremPY} for third-order arithmetic and beyond. In this section, we answer these questions.
\begin{theorem}
Let $\Z_{2 + \gamma}$ denote the theory \begin{center}{$Z^- + \Sigma_1-$\textsc{replacement} $+$ $\wo(\gamma)$ $+$ ``$\Ps^{\gamma}(\omega)$  exists'', } \end{center} for $1 \leq \gamma < \omega_1^{\mathsf{CK}}$ a recursive ordinal.
Then, we have:
\[\RCA \vdash \forall n \ (\mathbf{{\Pi}}^0_{1 + \gamma + 2})_{(n + 1)}\text{-}\Det \leftrightarrow \Pi^1_3\Rfc(\KP^{\gamma}_{\infty}) \leftrightarrow \Pi^1_3\Rfc(\Z_{2 + \gamma}),\]
\label{leonardoQ}
and indeed
\begin{align*}
    \RCA &\vdash \forall \delta < \gamma \ \forall n \ (\mathbf{{\Pi}}^0_{1 + \delta+ 2})_{(n + 1)}\text{-}\Det 
\\ &\leftrightarrow \Pi^1_3\Rfc(\{\KP^{\delta}_{\infty} : \delta < \gamma \}) 
\\ &\leftrightarrow \Pi^1_3\Rfc(\Z_{2 + \delta} : \delta < \gamma).
\end{align*}
\end{theorem}

\begin{lemma}
The following hold:  \label{preliminary to thm}
 \begin{enumerate}
        \item $\RCA \vdash \forall n > 1 \ \mathrm{Pr}_{\KP^{\gamma}_{n+1}}(\lceil (\Pi^0_{1 + \gamma + 2})_n\text{-}\Det \rceil)$; and\label{preliminary to thm1}
      \item $\RCA \vdash \forall n > 1 \ (\mathbf{{\Pi}}^0_{1 + \gamma + 2})_n\text{-}\Det \to  \beta(\KP^{\gamma}_n) $.
       \label{preliminary to thm2}
    \end{enumerate}
\end{lemma}

\begin{proof}[Proof of Lemma~\ref{preliminary to thm}]
To prove item \ref{preliminary to thm1}, we formalize the proofs given in \S\ref{SectionUpperBound} in $\RCA$. We can construct these uniformly in the natural numbers occurring as parameters. We then see that the function $f_1$ mapping each $n$ to the G\"odel number of the formalization of the proof of the implication
\[\KP^{\gamma}_{n+1} \to (\Pi^0_{1 + \gamma + 2})_n{-}\Det\]
given in \S\ref{SectionUpperBound} is recursive, provably in $\RCA$. 
 Still reasoning in $\RCA$, we show using $\Sigma_1$-induction that $f_1(n)$ is indeed a proof for all $n\in\mathbb{N}$, for suppose this is false. Then there is a least $n$ such that $f_1(n+1)$ is not a proof. This means that the formalization of one of the lemmas from \S\ref{SectionUpperBound} fails to be a proof for $n + 1$. However, for each of these lemmas, by analysing the proofs given in \S\ref{SectionUpperBound} and in Appendix \ref{SectAppendixUpperBound} one can extract a primitive recursive function which bounds the length of the proof for the case $n +1$ in terms of the length of the proof for the case $n$. Indeed, in most of these proofs, which were by induction on $|s| < n+1$, each step of the induction is essentially what is written down in the proof in \S\ref{SectionUpperBound} and in Appendix \ref{SectAppendixUpperBound}. Thus, we see by $\Sigma_1$-induction that the formalization of the proofs given here are indeed proofs for all $n$, provably in $\RCA$, which is a contradiction.

To prove item \ref{preliminary to thm2}, we work in $(\mathbf{{\Pi}}^0_{1 + \gamma + 2})_n\text{-}\Det$. Note that by this assumption we may freely assume $\Pi^1_1{-}\mathsf{CA}_0$ rather than merely $\RCA$. Indeed by appealing to Hachtman \cite{Hachtman} we may assume that every real belongs to a $\beta$-model of $\KP^\gamma$.
Letting $X$ be a real parameter, we claim that there is some $\alpha$ such that $L_\alpha(X) \models \KP^\gamma_n$. We assume for simplicity that $X = 0$ and relativize afterwards. The idea now is to carry out the proof in \S\ref{SectionLowerBound}. 

Note first that there is a primitive recursive function mapping each $m\in\mathbb{N}$ to the definition of the $(\mathbf{{\Pi}}^0_{1 + \gamma + 2})_n$ game constructed in the proof of Theorem \ref{beta_inequality}.\ref{beta_inequality1}, so within $\RCA$ we have access to this definition for $n$. 

It remains to verify that the proof of Theorem \ref{beta_inequality}.\ref{beta_inequality1} goes through under the current hypothesis (for possibly non-standard $n$). We now observe that the argument in Remark \ref{RemarkWS} can be carried out under the current assumptions: letting $\xi$ be such that a winning strategy for, say Player I, belongs to $L_\xi$, we consider the run of the game in which Player I plays according to the winning strategy and Player II plays a model of $\KP^\gamma$ containing $\sigma$, which exists by Hachtman \cite{Hachtman}. By $\Pi^1_1{-}\mathsf{CA}_0$, we can construct the wellfounded part of Player I's model and see that it is of the form $L_\alpha$. Now by carrying out the proof in \S\ref{SectionLowerBound} for $n$, we see by induction on $m\leq n$ that $L_\alpha$ is a model of $\KP^\gamma_n$. This induction is possible because both $L_\alpha$ and Player I's model have truth predicates and satisfy the full schema of \textsc{Foundation}.
\end{proof}

We may now proceed to the proof of the theorem, answering Pacheco and Yokoyama's question.

\begin{proof}[Proof of Theorem~\ref{leonardoQ}]
    First we assume $\Pi^1_3\Rfc(\Z_{2+ \gamma})$. 
    Let $n \geq 1$ be a natural number. 
    By Lemma~\ref{preliminary to thm}.\ref{preliminary to thm1}, we have $\mathrm{Pr}_{\Z_{2+ \gamma}}(\lceil (\mathbf{\Pi}^0_{1 + \gamma + 2})_n\text{-}\Det \rceil)$. 
    Since the latter is a $\Pi^1_3$ sentence, by reflection we get $(\mathbf{\Pi}^0_{1 + \gamma + 2})_n\text{-}\Det$. 
    Hence, we have proved $\forall n \ (\mathbf{\Pi}^0_{1 + \gamma + 2})_{n+1}\text{-}\Det$. 

    We now assume $\forall n \ (\mathbf{\Pi}^0_{1 + \gamma + 2})_{(n + 1)}\text{-}\Det$. Suppose towards a contradiction that $\phi$ is a $\Pi^1_3$ sentence such that $\mathrm{Pr}_{\KP^{\gamma}_{\infty}}(\lceil \phi \rceil)$ holds but $\phi$ does not. 
    Let $n > 1$ be sufficiently large so that $\mathrm{Pr}_{\KP^{\gamma}_n}(\lceil \phi \rceil)$. 
    Write the true $\Sigma^1_3$ sentence $\lnot \phi$ as ``$\exists X \forall Y \exists Z \ \theta$'' with set quantifiers ranging over $\Ps(\mathbb{N})$ and $\theta$ arithmetical. 
    Since $\phi$ is supposedly false, let us consider a counterexample $X_0 \subseteq \mathbb{N}$. By Lemma~\ref{preliminary to thm}.\ref{preliminary to thm2}, we can find a $\beta$-model $\M$ of $\KP^{\gamma}_n$ containing $X_0$. Since $\M$ is a $\beta$-model, it satisfies all true $\Pi^1_2$ sentences with parameters in $\M$ and since $X_0 \in \M$, it thus satisfies $\lnot\phi$, contradicting  
$\mathrm{Pr}_{\KP^{\gamma}_{\infty}}(\lceil \phi \rceil)$.
    We conclude that $\Pi^1_3\Rfc(\KP^{\gamma}_{\infty})$ holds.
\end{proof}

Concerning the case of transfinite length differences of such sets from the Borel hierarchy, it is shown in Martin~\cite{Martin} that $\mathrm{Rec}(\Z_2)$ is enough to prove $\Delta^0_4\text{-}\Det$. That is, determinacy holds for any stage $\gamma < \omega_1$ of the difference hierarchy for $\Pi^0_3$ sets. Adding a satisfaction predicate to the language of $L_{\mathrm{Set}}$ similarly allows us to generalize the proof of theorem~\ref{mPi4Det} to transfinite differences of $\Pi^0_4$ sets. Nevertheless, it is unknown how we can generalize our results to characterize each level $\gamma$ of this extended hierarchy. In particular, we have not investigated whether the results in Chapter 6 of \cite{AgWe} can be extended to the difference hierarchy over $\Pi^0_\gamma$ sets.

\appendix
\section{Appendix: Proofs of some results of Martin, Montalb\'an and Shore}
\subsection{Proofs of some lemmata from \S\ref{SectionLowerBound}}\label{AppendixProofsLB}
In this appendix, we include the proofs of lemmata \ref{star1}, \ref{star_kCons}, and \ref{domino} from \S \ref{SectionLowerBound}. The arguments are due to Montalb\'an and Shore and largely following \cite{MS}.
\begin{proof}[Proof of Lemma \ref{star1} on p.\pageref{star1}]
   Suppose for a contradiction that for every ordinal $\gamma \in Ord^{\N} \setminus \A^{\gamma + 1}$, there is a $\Sigma_1$ formula with parameters in $\A^{\gamma + 1}$ true in $L_{\gamma}$ but not in $\A^{\gamma + 1}$. By hypothesis, $W_{\N, 1}$ has a least element $\delta$. By definition of $W_{\N, 1}$, we have $\delta \not\in \A^{\gamma + 1}$. Since from this point we suppose $\N$ to be ill-founded, let \begin{align*}
        \delta > \gamma_0 > \gamma_1 > \gamma_2 > \cdots > Ord^{\A^{\gamma + 1}}
    \end{align*}
    be a descending sequence converging down to the cut $(Ord^{\A^{\gamma + 1}}, Ord^{\N} \setminus \A^{\gamma + 1})$. By our absurd assumption, for each $i$, there is a $\Delta_0$ formula $\phi_i$ with parameters in $\A^{\gamma + 1}$ and a $<_L$-least witness $z_i \in L_{\gamma_i}$ such that 
    \begin{align*}
        \N \models \phi_i(z_i) \quad \text{but} \quad \A^{\gamma + 1} \models \lnot \exists y \phi_i(y).
    \end{align*}
    By thinning out our sequence if necessary, we may assume that $z_i \not\in L_{\gamma_i+1}$ so that ${z_i : i < \mathbb{N}}$ is a $<^{\N}_{L}$-descending sequence. Since we assumed by our absurd hypothesis that $\delta$ was the least element of $W_{\N, 1}$, for all $i$, \begin{align*}
        \M \models \exists y \phi_i(y).
    \end{align*}
    Let $y_i$ be the $<^{\M}_{L}$-least such witnesses. Since $\M$ is well-founded, the sequence $\{y_i : i < \mathbb{N}\}$ cannot be a $<^{\M}_{L}$-descending sequence. So there exist two indices $i < j$ such that \begin{align*}
        z_j <^{\N}_{L} z_i \quad \text{but} \quad y_i <^{\M}_{L} y_j.
    \end{align*}
    Therefore, $L_{\gamma_{j + 1}}$ is a witness in $L_{\gamma_j}$ for the $\Delta_0$ formula 
    \begin{align*}
        \psi(x) \equiv \exists (z \in x) \ \phi_j(z) \land \forall (z \in x) \  \lnot\phi_i(z),
    \end{align*} 
    that is true in $\N$ but not in $\M$ where we have $y_i <^{\M}_{L} y_j$. This however shows that $\gamma_{j+1}$ is an element of $W_{\N, 1}$, a contradiction.
    \end{proof}

\begin{proof}[Proof of Lemma 
 \ref{star_kCons} on p.\pageref{star_kCons}]   First, we claim that ${\alpha}$ is not $\Sigma_k$ definable in $L_{\beta_1}$, with parameters from $L_{\alpha}$. Since $\alpha \in \M \models T^{\gamma}_n$, it follows that $\alpha$ is not $n$-admissible and by lemma~\ref{cardinalityKPL}, every $\beta \in L_{\alpha}$ is of cardinality at most $\mathcal{P}^{\gamma}(\mathbb{N})$ (here we mean $\Ps^{\gamma}_{L_{\alpha}}(\mathbb{N})$). Thus, there is a $\Pi_{n-1}$ definable map on $L_{\alpha}$ from $\Ps^{\gamma}(\mathbb{N})$ onto $\alpha$, eventually with parameters. This defines in $L_{\alpha}$ a $\Sigma_n$ well ordering of $\Ps^{\gamma}(\mathbb{N})$ of order type $\alpha$. This ordering cannot belong to $\N$ as it would define its well-ordered part. By our absurd hypothesis, in $L_{\beta_1}$, we have a $\Sigma_k$ definition of this ordering using the $\Sigma_k$ definition of $\alpha$ and bounded quantification over $L_{\alpha}$. However then, since $L^{\M_{\I}}_{\beta_1} \equiv_{k, \A^{\gamma + 1}} L^{\M_{\II}}_{\beta_2}$, this ordering is now definable in $L_{\beta_2}$ and hence belongs to $\N$, a contradiction.

 For point $1$ it is of course sufficient to show $L_{\alpha} \preceq_{k} L_{\beta_1}$ since the other is $\Sigma_k$ elementary equivalent to it, over $\A^{\gamma + 1}$. Since $\beta_1$ is $(k-1)$-admissible, from lemma~\ref{noParamSkolem} we know that $L_{\beta_1}$ has a parameterless $\Sigma_k$ Skolem function. Let thus be $H$, the $\Sigma_k$-Skolem hull of $L_{\alpha}$ in $L_{\beta_1}$. We show that $H = L_{\alpha}$, which will prove our claim. Suppose otherwise towards a contradiction and consider $L_{\gamma}$, the Mostowski collapse of $H$, with $\alpha < \gamma \leq \beta_1$. Let $\alpha'$ be the ordinal of $H$ being sent to $\alpha \in L_{\gamma}$ by the collapse. By construction of $H$, we would have a $\Sigma_k$ definition of ${\alpha}$ in $L_{\gamma}$, with still parameters from $L_{\alpha}$, since the collapse is the identity over $L_{\alpha}$. However, since \begin{align*}
        L^{\M}_{\gamma} \equiv_{k, \A^{\gamma + 1}} H \equiv_{k, \A^{\gamma + 1}} L^{\M}_{\beta_1},
    \end{align*}
    ${\alpha}$ would be $\Sigma_k$ definable in $L^{\M}_{\beta_1}$, a contradiction.

    Next concerning point $2$, let us, as usual, suppose that our claim is false and thus $\alpha$ is not $k+1$-admissible and so there is a $\Pi_{k}$ definable map on $L_{\alpha}$ from $\Ps^{\gamma}(\mathbb{N})$ onto $\alpha$. Since $\A \preceq_{k} L_{\beta_2}$, as for our first observation, we would have that $\N$ would be able to define its well-founded part and a contradiction would occur.

    Finally, about point $3$, we use theorem~\ref{charSepL} to get from the freshly proved $(k+1)$-admissibility of $\alpha$, the existence of an unbounded infinity of $\gamma < \alpha$ such that $L_{\gamma} \preceq_{k} L_{\alpha} \preceq_{k} L_{\beta_2}$. The set of the $\gamma < \beta_2$ such that $L_{\gamma} \preceq_{k} L_{\beta_2}$ is on the other hand definable in $\N$. So now if for some $\delta \in Ord^{\N} \setminus \A^{\gamma + 1}$ this set had supremum $\alpha$, then $\alpha$ would be definable in $\N$, and we know it is not. So for every $\delta \in Ord^{\N} \setminus \A^{\gamma + 1}$, there exists $\delta > \gamma \in Ord^{\N} \setminus \A^{\gamma + 1}$ such that $L_{\delta} \preceq_{k} L_{\beta_2}$.
\end{proof}

\begin{proof}[Proof of Lemma \ref{domino} on p.\pageref{domino}]
 Let us prove our claim by induction, lemma~\ref{star1} giving us the base step. Assume there exist some fixed $\beta_1$ and $\beta_2$ satisfying $\star_{k-1}$.
    
    Firstly we claim that no ordinal $\delta \in W^{ \beta_1, \beta_2}_{\N,k}$ is in $\A^{\gamma + 1}$. Let $\delta \in \A^{\gamma + 1}$, any $S_{k-1}$ formula $\forall (x \in z) \phi(z, \bar{y})$ and $z_2, \bar{y_2} \in L^{\N}_{\delta} \subseteq A$ such that $L_{\beta_2}^{\N} \models \forall (x \in z_2) \phi(z_2, \bar{y_2})$. By induction hypothesis, it follows that $L_{\beta_2}^{\N} \models \forall (x \in z_1) \phi(z_1, \bar{y_1})$ too, with $z_1$ and $\bar{y_1}$ the images of $z_2$ and $\bar{y_2}$ in $\M$ (via the isomorphism $\A^{\gamma + 1}$). Thus, $\delta \not\in W^{ \beta_1, \beta_2}_{\N,k}$.

    Now, by hypothesis, $W^{ \beta_1, \beta_2}_{\N,k}$ has a least element $\delta$, necessarily not in $\A^{\gamma + 1}$. Also, by clause $3$ of lemma~\ref{star_kCons}, there is a descending sequence \begin{align*}
        \delta > \gamma_0 > \gamma_1 > \gamma_2 > \dots
    \end{align*}
    in $Ord^{\N}$ converging down to $\alpha = Ord^{\N}$, such that, for each $i< \mathbb{N}$, $L^{\N}_{\gamma_i} \preceq_{k-1} L^{\N}_{\beta_2}$.
    Now we argue exactly like lemma~\ref{star1} (where we had $k=1$ and $\preceq_0$ is absoluteness for $\Delta_0$ formula, which follows from the transitivity of the structures) to get that for some $i < \mathbb{N}$, \begin{align*}
        L_{\alpha} \preceq_k L^{\N}_{\gamma_i} \preceq_{k-1} L^{\N}_{\beta_2}.
    \end{align*} 
    Finally, we conclude by lemma~\ref{boundingPrinciple} that $L^{\N}_{\gamma_i}$ is $(k-1)$-admissible as $L^{\N}_{\beta_2}$ is $(k-2)$-admissible by induction hypothesis while $\alpha$ is (even) $\Sigma_k$ admissible by lemma~\ref{star_kCons}, so that $\star_k(\alpha, \gamma_i)$, as required. 
    \end{proof}
    
\subsection{Proofs of some lemmata in \S\ref{SectionUpperBound}}\label{SectAppendixUpperBound}
In this appendix, we include the proofs of some lemmata needed for \S\ref{SectionUpperBound}. The arguments are due to Martin and the complexity computations  are due to Montalb\'an and Shore.

\begin{definition}[local witness]
    A quasistrategy $U$ \emph{locally witnesses} $P^s(T)$ if $|s| = n+1$ and $U$ is a quasistrategy for Player \I\ (resp. \II) if 
        $l$ is even (resp. odd) and there is $D \subseteq T$ such that, for every $d \in D$, there is a quasistrategy 
        $R^d$ for Player \II\ (resp. \I) if $l$ is even (resp. odd) in $T_d$ such that the following conditions are satisfied:
        \begin{enumerate}
            \item $\forall d \in D\cap U$, \ $U_d \cap R^d$ witnesses $P^s(R^d)$.
            \item $[U] \setminus \bigcup_{d \in D} [R^d] \subseteq A \ (\text{resp. }\Bar{A})$.
            \item $\forall p \in T \ \exists^{\leq 1} d \in D, \ d \subseteq p \land p \in R^d$.
        \end{enumerate}
        We observe that ``$U$ locally witnesses $P^s(T)$'' is a $\Sigma_{|s|}$ sentence.
    \end{definition}

    The following lemma will be useful in a recursion in lemma{~}\ref{nolocal} and will make us more familiar with 
    the clauses of the preceding definition. It tells us that if a local witness is not a witness for the 
    second reason, then we can construct a local witness for a preceding relation.

    \begin{lemma}
        Let $|s|=n+1 > 1$, if $U$ locally witnesses $P^s(T)$ and $P^{s[n]}(T)$ is witnessed by some $\hat{T}$, 
        then there is a local witness $\hat{U}$ of $P^{s[n-1]}(\hat{T})$ if $n>1$. 
        When $n=1$, $P^{s[n]}(U)$ fails.\label{induc}
    \end{lemma}
    \begin{proof}
        Without loss of generality, let $m-n (= m - |s[n]|)$ being odd. So $\hat{T}$ is a  
        Player \II's quasistrategy (we can suppose $\hat{T} \subseteq U$). Suppose $n>1$. We reason in $\KP^2_{m+1} + \AC$, for $m \geq 2$. 
        
        The main goal for $\hat{U}$ is to escape from the range of each $R^d$. Let $d \in \Hat{D}$ iff 
        $d \in \hat{T} \cap D$ and Player \II\ has a winning strategy in $G(\hat{T}_d, \Bar{[R^d]})$, an open game. 
        For $d \in \Hat{D}$ we let $\Hat{R}^d$ be Player \II's non-losing quasistrategy in this game and $R^d = \emptyset$ 
        for $d \in D \setminus \hat{D}$. 
        The quasistrategy of Player \II\ is a $\Pi_2$ set and so is the collection $\{(r,d) : d \in \hat{D} \land r \in \hat{R}^d \}$. The idea is now that either Player \I\ can get out of a given $R^d$, or he has to get 
        out of $\hat{R}^d$, in such a way that, by definition, Player \I\ gets a strategy to go out of $R^d$. 
        
        By hypothesis $[\Hat{R}^d] \subseteq [\hat{T}] \subseteq \Bar{A} \cup A_{m-n, s(n-1)}$, so $\Hat{R}^d$ 
        satisfies the first condition to witness $P^{s[n]}(U_d \cap R^d)$.
        However, by property $(1)$ of the local witness, $U_d \cap R^d$ witnesses $P^s(R^d)$ and so, in particular,
        $P^{s[n]}(U_d \cap R^d)$ fails and then $\Hat{R}^d$ is not a witness for it.
        As a consequence, the second condition must fail with $\Hat{R}^d$, that is, there is a witness $\Hat{U}^d$ 
        for $P^{s[n-1]}(\Hat{R}^d)$. 
        We then define a sequence ${(\Hat{U}^n)}_{n < \mathbb{N}}$ such that
            $\Hat{U}^n$ \text{ witnesses } $P^{s[n-1]}(\Hat{R}^d)$,
        using ($\Sigma_{n-2}$) axiom of choice. Finally, we similarly choose strategies for Player \I, 
        ${\sigma_{p,d}}$, winning in $G(\hat{T}_p, \Bar{[R^d]})$ for $d \in \hat{D}$ when some $p \not\in \hat{R}^d$ 
        is reached. 
        
        We now (arithmetically in the above parameters) define by the following a quasistrategy $\Hat{U}$ 
        for Player \I\ in $\hat{T}$.
        \begin{enumerate}[(i)]
            \item If $p \in \Hat{U}$ and there is no $d \in D$ such that $d \subseteq p$ and $p \in R^d$, 
            then the child of $p$ in $\Bar{U}$ are the same as those in $\hat{T}$, otherwise;
            \item If $p \in \Hat{U}$ is a minimal extension of some $d \in D$ such that  
            $p \in R^d \setminus \Hat{R}^d$, then we escaped Player \II's non-losing strategy, which means that Player \I\ can play 
            $\sigma_{p,d}$ until she reaches a $p \not\in R^d$, otherwise;
            \item If $p \in \Hat{U} \cap \hat{D}$, let $\Hat{U}_p = \Hat{U}^p$ as long as we stay in $\Hat{R}_d$.
        \end{enumerate}
    
        We now prove that the three conditions of a local witness hold.
        \begin{enumerate}
            \item Take $p \in \Hat{U} \cap \hat{D}$, by $(iii)$ $\Hat{U}_p \cap \Hat{R}_d = \Hat{U}^p$, which 
            is witnessing $P^{s[n-1]}(\Hat{R}^d)$.
            \item Any play $x \in [\hat{U}] \setminus [\hat{R}^d]$ would have escaped $R^d$ in some finite 
            position by $(ii)$. Thus, \begin{align*}
                [\Hat{U}] \setminus \bigcup_{d \in D} [\Hat{R}^d] \subseteq [U] \setminus \bigcup_{d \in D} 
                [R^d] \subseteq A,
            \end{align*}
            by hypothesis.
            \item This condition is immediate from the corresponding hypothesis since we have just restrained 
            $D$ and the $R^d$'s.
        \end{enumerate}
        
        Finally, when $n=1$, we suppose for a contradiction that there exists a $\hat{T}$ such as for the case 
        $n>1$. We can then keep the same construction with the following differences. 
        When we choose a witness $\Hat{U}^d$ for $P^{s[n-1](\Hat{R}^d)}$, we must take a 
        winning strategy for Player \I\ in $G(\Hat{R}^d, A)$, and we need to show that $[\Hat{U}] \subseteq A$ 
        to see that it witnesses $P^{\langle \ \rangle}(\hat{T})$ for the desired contradiction.
        The point here is that if we stay in some $R^d$, then we follow $U^d$, which is a winning strategy 
        for Player \I\ in $G(\Hat{R}^d, A)$. If we leave $\Hat{R}^d$, then we leave $R^d$ by $(ii)$ in the definition 
        of $\Hat{U}$. If we leave every $R^d$, then we follow $\hat{T}$ and then stay in $U$ and also 
        end up in $A$ by clause $(2)$ of the definition of $U$ being a local witness for $P^s(T)$.
    \end{proof}
    
    Now that the above construction has been done we can prove that there is no ``local-only'' witness.
    
    \begin{lemma}
        If $U$ locally witnesses $P^s(T)$, then $U$ witnesses $P^s(T)$.\label{nolocal}
    \end{lemma}
    \begin{proof}
        Without loss of generality, we suppose that $l$ is even. Let us show the first property of the witness. 
        Consider $x \in [U]$. If $x \in A$, there is nothing to prove. If not, by property $(2)$ of the local witness, 
        $ x \in [R^d]$ for some $d \in D$. Then, by $(1)$, $U_d \cap R^d$ witnesses $P^s(R^d)$ and so by 
        the first property of the latter witness, $x \in A_{l, s(n)}$ as required.
    
        We now show the second part of the definition by induction on $|s| = n+1 \leq m$. Without loss of generality, we suppose $m$ odd.
        We begin with $n = 0$. Suppose for a contradiction $P^{\langle \ \rangle}(U)$, that is there is a winning 
        strategy $\tau$ for Player \II\ in $G(U, A)$.
    
        We claim that there is a $d \in D$ belonging to $\tau$ such that every $x \supseteq d$ in $[\tau]$ is 
        also in $[R^d]$. Suppose the contrary: $\forall d \in D \ \exists d \subset x \in [\tau] \setminus [R^d]$. 
        Now note that every position \begin{align*}
            e \in \tau \setminus \bigcup_{d \in D, d \subset e} R^d,
        \end{align*}
        has a minimal extension $\Hat{d} \in D \cap \tau$. Otherwise, for any $e \subset x \in [\tau]$ 
        we would have $ x \not\in \bigcup_{d \in D} [R^d]$. By property $(2)$ of the local witness it would then follow 
        that $x \in A$, a contradiction with our choice of $\tau$. Next note that, by our assumption, any such $\Hat{d}$ 
        has a minimal extension $\Hat{e} \in \tau \setminus R^{\Hat{d}}$. By property $(3)$ of the local witness, no 
        $\Hat{d} \subset e' \subset \Hat{e}$ is in $D$ and so $\Hat{e}$ has the same property as $e$. 
        We can iterate this process to create a sequence $e_j \subseteq \tau$ such that 
        $\bigcup e_j = x \not\in \bigcup_{d \subset x, d \in D} R^d$, which leads, as above, to a contradiction.
    
        So we have such $d$. Thus, $\tau_d$ is a winning strategy for Player \II\ in $G(U_d \cap R^d,A)$, that is, 
        $P^{\langle \ \rangle}(U_d \cap R^d)$ contradicting property $(1)$ of the local witness and so 
        establishing the desired property.
    
        Now suppose $s = n+1 > 1$. If $n = 1$, lemma{~}\ref{induc} gives the conclusion. If $n > 1$, suppose for a 
        contradiction that $P^{s[n]}(T)$ is witnessed by some $\hat{T}$. By applying lemma{~}\ref{induc} we get a 
        local witness $\hat{U}$ of $P^{s[n-1]}(\hat{T})$. Then the induction hypothesis implies that $\hat{U}$ is 
        a witness for the same property, contradicting the existence of $\hat{T}$ and concluding our induction.
    \end{proof}

    \begin{proof}[Proof of lemma~\ref{failure}]
        Without loss of generality, we suppose $l$ to be odd. First, if $|s| = 0$, then Player \II\ does not have a winning strategy. 
        Then as we are used to, we define $W$ to be Player \I's non-losing quasistrategy and verify that $P^s(W)$ 
        fails everywhere.
    
        Now suppose $|s| = n+1$ and, without loss of generality that $l$ is even. Invoking $\Sigma_{|s|}$ \textsc{separation}, we define the set $D$ by\begin{align*}
            d \in D \leftrightarrow d \in S \land P^s(T_d) \land \lnot P^s(T_{d[|d|-1]}), 
        \end{align*}
        an intersection of a $\Pi_{|s|}$ and a $\Sigma_{|s|}$ set. 
        We suppose $D$ to be non-empty and, as we often do now, by $\Sigma_{|s|}$-$\AC$ we 
        pick a sequence of witnesses $U^d$ of $P^s(T_d)$ for each $d \in D$.
    
        Consider now the game $G(T,B)$ where $B = \{x\in [T] \mid \exists d \in D \ d \subseteq x\}$.
        We claim that Player \I\ has no winning strategy in this game. If there was such one, $\sigma$, then we could define 
        a quasistrategy $U$ for Player \I\ in $T$ by following $\sigma$ until a position $d \in D$ is reached, at 
        which point we move into $U^d$. With $D$ and $R^d = T_d$ we can easily verify that three clauses 
        of $U$ locally witnessing $P^s(T)$ are satisfied:
        \begin{enumerate}
            \item Take $d \in D \cap U$, $U_d \cap R^d = U^d$ witness $P^s(T_d)$;
            \item Since $[U] \subseteq \bigcup_{d \in D} [U^d]$, $[U] \setminus \bigcup_{d \in D} [T_d] = \emptyset$;
            \item Taking any $p \in T$ if there exists $d \subseteq p$ such that $p \in R^d$ the unicity follows 
            from the minimality of $d$;
        \end{enumerate}
        which by lemma{~}\ref{nolocal} contradicts the fact that $P^s(T)$ fails.
    
        Thus, we let $W$ be Player \II's non-losing quasistrategy in $G(T, B)$ and $\sigma_p$ be a chosen winning strategy 
        for Player \I\ if $p \in S \setminus W$ is reached. Suppose for a contradiction that $W$ is not as required. 
        Then for some $q \in W$ we can find a witness $\Hat{U}$ of $P^s(W_q)$. 
        Consequently, we define a quasistrategy $U$ for Player \I\ in $T_q$: 
        \begin{enumerate}[(i)]
            \item We begin to set up $U \cap W_q = \Hat{U}$;
            \item If $p \in U \setminus W$, Player \I\ plays $\sigma_p$ until she reaches a position $d \in D$ from 
            where she plays $U^d$, witnessing $P^s(T_d)$.
        \end{enumerate}
        If we now consider $U$, $\Hat{D} \coloneqq D \cup \{q\}$, $R^{d} \coloneqq T_d$ and $R^q = W_q$,
        we verify that $U$ locally witnesses $P^s(T_q)$: 
        \begin{enumerate}
            \item Take $d \in \hat{D} \cap U$, $U_d \cap T_d$ witnesses $P^s(T_d)$ and $U \cap W_q = \hat{U}$, $P^s(W_q)$; 
            \item As before, $[U] \setminus \bigcup_{d \in D} [R^d] \subset \emptyset$;
            \item Again it follows from minimality and the fact that $q \in W$, which is non-losing. 
        \end{enumerate}
        Using lemma{~}\ref{nolocal} we know that $P^s(T_q)$ holds, but then by 
        definition of $D$ there is a minimal $d \subseteq q$ in $D$, which contradicts the choice of $W$. 
        Thus, $P^s(W)$ fails everywhere.  
    \end{proof}

    \begin{proof}[Proof of lemma~\ref{binary}]
        We prove the lemma by reverse induction on $n<m$. Suppose without loss of generality that $m-n$ is odd and $P^s(T)$ fails. 
        We use $\Sigma_{m}$-Dependent-$\AC$ to define by induction on the length of positions a quasistrategy 
        $U$ for Player \II\ in $S$ along with $D \subseteq S$ and $R^d$ for $d \in D$ showing that \begin{align*}
            U \text{ locally witnesses } P^{s[n]}(T) \text{ if } n > 0 \quad \text{and} \quad 
            U \text{ witnesses } P^{s[n]}(T) \text{ if } n = 0.
        \end{align*}
        It suffices then to use lemma{~}\ref{nolocal} to have the desired property in every case. 
    
        Initiation: $\langle \ \rangle \in U$, we say that it marks $0$. 
        \begin{enumerate}[(i)]
            \item If $n = m-1$, by lemma{~}\ref{failure} we set $W^{\langle \ \rangle}$ be a quasistrategy for Player \II\ in $S$ such that 
            $P^s(W^{\langle \ \rangle})$ fails everywhere.
            \item If $n<m-1$, then we know by reverse induction that $P^{s^{\smallfrown}0}(T)$ holds. Applying lemma{~}\ref{enpowerment} 
            there exist a $W^{\langle \ \rangle}$ strongly witnessing this fact and so $P^s(W^{\langle \ \rangle})$ fails everywhere.
        \end{enumerate}
        
        Recursion step: Take $q \in U$ marking $j < \mathbb{N}$, with $P^s(W^q)$ failing everywhere. Consider the closed game 
        \begin{align*}
            G(W^q, A_{m-n-1,s(n),j}).
        \end{align*}
        If it is not a win for II, we put $q \in D$ and define $\Hat{R}^q$ to be Player \I's non-losing quasistrategy in this game. 
        We also define $R^q$ to be $\Hat{R}^q$ on $W^q$ and to simply $T_q$ elsewhere. 
        Thus, $[\Hat{R}^q] \subseteq A_{m-n-1,s(n),j} \subseteq A^{m-n-1, s(n)}$ by definition and since $\hat{R}^q$ is 
        a non-losing quasistrategy for a closed set. 
        Thus, if $P^{s[n]}(\Hat{R}^q)$, the two properties of $\Hat{R}^q$ witnessing $P^s(W^q)$ would be satisfied, 
        contrary to our assumption that $P^s(W^q)$ fails everywhere. 
        So we may take $U^q$ to be a witness for $P^{s[n]}(\Hat{R}^q)$ (a $\Pi_{|s|-2}$ relation for $n \geq 1$). 
        We now continue to define $U$:
        \begin{enumerate}
            \item On $\Hat{R}^q$, $U = U^q$;
            \item If $p \not \in \Hat{R}^q$ ($p = q$ if the game is not a win for I), Player \II\ can follow a winning strategy 
            $\tau_p$ until he reaches a $q'$ with $[W^q_{q'}] \cap A_{m-n-1,s(n),j} = \emptyset$,
        \end{enumerate}
        which one exists since Player \II\ is playing an open game. As a consequence, we say that $q'$ marks $j+1$. 
        Now $P^s(W^q_{q'})$ fails everywhere since $P^s(W^q)$ does. 
        \begin{enumerate}[(i)]
            \item If $n = m-1$, we define $W^{q'} = W^q_{q'}$.
            \item  If $n < m-1$, then by our reverse induction on $n$, $P^{s^{\smallfrown}j + 1}(W^q_{q'})$
            and there exists $W^{q'}$ strongly witnessing this fact, as well as $P^{s^{\smallfrown}j + 1}(T_{q'})$.
        \end{enumerate}
        
        In the cases $(ii)$, with $n < m-1$, we use $\Sigma_{m}$-Strong-Dependent-$\AC$. 
    
        If $n>0$ we show the properties for $U$, together with $D$ and $R^d$ locally witnessing $P^{s[n]}$. 
        \begin{enumerate}
            \item Take $d \in D \cap U$, by construction $U_d \cap R^d = U_d \cap \Hat{R}^d = U^d$ and 
            $U^d$ witnesses $P^{s[n]}(\Hat{R}^d)$;
            \item We prove it here under;
            \item It follows from the fact we put a new $d \in D$ only once we have left $\Hat{R}^d$. 
        \end{enumerate}
        Let $x \in [U]$ and \begin{align*}
            \emptyset = q_0 \subset q_1 \subset \dots \subset q_i \subset \dots 
        \end{align*}
        be the strictly increasing sequence of the initial segments $q$ of $x$ such that $q_j$ marks $j$.
        By construction, each $q_j \in D$. If the sequence terminates at some $q = q_k$, then, by definition, 
        $x$ never leaves $\Hat{R}^d$ and so $x \in R^d$. So if $x$ is out of the $R^d$'s, the sequence is infinite and \begin{align*}
            x \not\in A_{m-n-1} \subset A_{m-n-1, s(n)} = \bigcup_{j < \mathbb{N}} A_{m-n-1, s(n), j}.
        \end{align*} 
        If $n+1 = m$, $x \not\in A_0$ implies $x \not\in A$, and we are done.
        If $n+1 < m$, as $W^{q_j}$ witnesses $P^{s^{\smallfrown}j}(T_{q_{j+1}})$ and $m-|s^{\smallfrown}j|$ is odd, 
        \begin{align*} 
            x \in \Bar{A} \cup \bigcup_{j<\mathbb{N}} A_{m-n-2, j} = \Bar{A} \cup A_{m-n-2}. 
        \end{align*}  
        As $m-n-1$ is even by our case assumptions, it follows that $ x \in A_{m-n-2} \setminus A_{m-n-1} \subseteq \Bar{A}$. 
        By lemma{~}\ref{nolocal}, $U$ witnesses $P^{s[n]}$.
    
        Finally, if $n=0$, then we argue that $U$ is a winning quasistrategy for Player \II\ in $G(T, A)$. Consider any 
        $x \in [U]$. If there is a $d \in D$ such that $x \in [\Hat{R}^d]$, then $x \in U^d$ by construction. 
        Now $U^d$ is a witness for $P^{\langle \ \rangle}(\Hat{R}^d)$ (as $n=0$, $s[n] = \langle \ \rangle$), 
        that is, $U^d$ is a winning strategy for Player \II\ in $G(\Hat{R}^d, A)$. Thus, $x \in \Bar{A}$, as required. 
        On the other hand, if $x$ leaves every $\Hat{R}^d$, then, by the argument above, $x \in \Bar{A}$ as well. 
    \end{proof}

\subsection{Unravelling Borel games}\label{appendix unravelling}
In this appendix, we present the proof of the unravelling lemma and the existence of inverse limits needed for \S\ref{SectionUpperBound}. The proofs are due to Martin \cite{BorelDetMartin}. 

\begin{proof}[Proof of the unravelling lemma (Lemma \ref{unravelling lemma})]
    Let $G(T, X)$ with the tree $T$ and the payoff set $X$ as in the lemma. Since $\Ps^{\gamma}(T)$ exists, we can define $\tilde{T}$ as above for a given $k$. To prove the lemma it suffices to show that we can construct $\pi$ and $\phi$ that satisfy definition~\ref{covering}. 

    The definition of $\pi$ is straightforward. Given a sequence \begin{align*}
        \langle a_0, \dots, a_{2k-1}, (a_{2k}, \Sigma_{\I}), (a_{2k+1}, \cdot), \dots, a_n \rangle = s \in \tilde{T},
    \end{align*}
    we define $\pi(s) = \langle a_0, \dots, a_{2k-1}, a_{2k}, a_{2k+1}, \dots, a_n \rangle \in T$.

    For the rest of the proof, we define somewhat informally $\phi$ by prescribing each Player how to play in $T$ according to one of their strategies in $\tilde{T}$. We will perform the construction so that the function $\phi$ defined that way satisfies definition~\ref{covering}.

    \paragraph*{Case I. } Let $\tilde{\sigma}$ be a strategy for \I\ in $\tilde{T}$. 

    During the first $2k$ moves, she just follows $\tilde{\sigma}$ until $\tilde{\sigma}$ asks her to play some $(x_{2k}, \Sigma_{\I})$, where she has to play $x_{2k}$ in $\sigma$. Then Player $\II$ plays some $x_{2k+1}$. Let us discuss the two different subcases.
    
    \paragraph*{Subcase $1$. } Player \I\ has a winning strategy $\sigma$ in the open game \begin{align*}
        G((\Sigma_{\I})_{(x_{2k+1})}, [(\Sigma_{\I})_{(x_{2k+1})}] \setminus X_{\langle x_0, \dots, x_{2k+1} \rangle}).
    \end{align*}
    
    Then $\sigma$ requires $\I$ to play $\sigma$. After finitely many steps, the shortest position $u \not\in (T_{X})_{\langle x_0, \dots, x_{2k+1} \rangle}$. Writing $u = \langle x_{2k+2}, \dots, x_{2l-1} \rangle$, the position 
    $$\langle a_0, \dots, a_{2k-1}, (a_{2k}, \Sigma_{\I}), (a_{2k+1}, u), x_{2k+2}, \dots, a_{2l-1} \rangle,$$ 
    can legally be plugged in $\tilde{\sigma}$ to get the remaining of Player \I's strategy. 
    
    \paragraph*{Subcase $2$. } Player \II\ has a winning strategy in the open game \begin{align*}
        G((\Sigma_{\I})_{(x_{2k+1})}, [(\Sigma_{\I})_{(x_{2k+1})}] \setminus X_{\langle x_0, \dots, x_{2k+1} \rangle}).
    \end{align*}
    
    Let $\Sigma_{\II}$ be his canonical quasistrategy in this game (avoiding his losing position in this closed game for him). In particular, $\Sigma_{\II} \subseteq (\Sigma_{\I})_{x_{2k+1}}$. Player \I\ can then legally plug the position $\langle a_0, \dots, a_{2k-1}, (a_{2k}, \Sigma_{\I}), (a_{2k+1}, \Sigma_{\II}) \rangle$ in $\tilde{\sigma}$ and just copy the answer given by this strategy as long as Player \II\ keep playing consistently with $\Sigma_{\II}$ in $T$. If at some point (for some $l$ with $2l-1 > 2k+2$), \II\ plays so that $\langle x_{2k+2}, \dots, x_{2l-1} \rangle \not\in (\Sigma_{\II})_{(x_{2k+1})}$, which is no longer a legal position following some play of $(a_{2k+1}, \Sigma_{\II})$ in $\tilde{T}$, then by definition of $\Sigma_{\II}$, \I\ would now have a winning strategy in \begin{align*}
        G((\Sigma_{\I})_{ (x_{2k+1}, \dots, x_{2l-1}) }, [(\Sigma_{\I})_{(x_{2k+1}, \dots, x_{2l-1})}] \setminus X_{(x_0, \dots, x_{2k+1}, \dots, x_{2l-1})}).
    \end{align*}
    If this is the case, we proceed similarly to subcase 1.

    \paragraph*{Case II. } Let $\tilde{\sigma}$ be a strategy for \II\ in $\tilde{T}$. 

    During the first $2k$ moves, he just follows $\tilde{\sigma}$ until \I plays some $x_{2k}$ in $T$. We define \begin{align*}
        U &= \{ \langle x_{2k+1} \rangle^{\smallfrown} u \in T_{\langle x_0, \dots, x_{2k} \rangle} \mid \ u \text{ has even length and } 
        \\
        &\qquad \exists \Sigma_{\I} \in \QSI(T_{\langle x_0, \dots, x_{2k} \rangle}) \ \tilde{\sigma}(\langle x_0, \dots, (x_{2k}, \Sigma_{\I}) \rangle) = \langle x_0, \dots, (x_{2k+1}, u) \rangle \}; 
        \\
        \mathcal{U} &= \{ x \in [T_{\langle x_0, \dots, x_{2k} \rangle}] \mid \ \exists \langle x_{2k+1} \rangle^{\smallfrown}u \in U \ \langle x_{2k+1} \rangle^{\smallfrown}u \subset x \}.
    \end{align*}

    We consider the open game $G(T_{\langle x_0, \dots, x_{2k} \rangle}, \mathcal{U})$ with the convention that \II\ plays first.

    \paragraph*{Subcase $1$. } Player \II\ has a winning strategy $\tau$ in this game. 

    Then he follows $\tau$ until he reaches a position $u$ witnessing his victory in $G$. Let $\Sigma_{I}$ witnessing that $\langle x_{2k+1} \rangle^{\smallfrown}u \in U$. Writing $u = \langle x_{2k+2}, \dots, x_{2l-1} \rangle$, the position $\langle x_0, \dots, x_{2k-1}, (x_{2k}, \Sigma_{\I}), (x_{2k+1}, u), x_{2k+2}, \dots, x_{2l-1} \rangle$, can legally be plugged in $\tilde{\sigma}$ to get the remaining of Player \II's strategy. 

    \paragraph*{Subcase $2$. } Player \I\ has a winning strategy in this game.

    Let $\Sigma_{I}$ be her canonical winning quasistrategy. Observe that if we plug the position $\langle x_0, \dots, x_{2k-1}, (x_{2k}, \Sigma_{\I})\rangle$ in $\tilde{\sigma}$, it cannot ask Player \II\ to play a $u$ as in option $2$ of the rules of our auxiliary tree, since no sequence of $\Sigma_{\I}$ can be in $U$.

    So \II\ can follow $\tilde{\sigma}$ with this position given to it, as long as \I\ plays in $T$ consistently with $\Sigma_{\I}$. If \I\ does no longer do, then Player $\II$ can play as in subcase $1$ since he is now in a winning position in the game $G$.  

\end{proof}

\begin{proof}[Proof of existence of inverse limits (Lemma \ref{existence of inverse limits})]
    Put \begin{align*}
        &s \in T_{\infty} \leftrightarrow \forall i \ ( \bar{s} \leq 2(k + i) \rightarrow s \in T_i ); 
        \\
        &\pi_{\infty, i}(s) = \left\{\begin{aligned}
            s \qquad &\text{ if } \bar{s} \leq 2(k + i),
            \\\pi_{i+1} \circ \cdots \circ \pi_{j}(s) &\text{ otherwise, for some } \bar{s} \leq 2(k + j);    
        \end{aligned}\right.
        \\
        &\phi_{\infty,i}(\sigma_{\infty}) = \sigma \leftrightarrow \left\{\begin{aligned}
            (\sigma_{\infty})_{\mid 2(k+i)} = (\sigma)_{\mid 2(k+i)} \quad \land
            \\
            \phi_{i+1} \circ \cdots \circ \phi_{j}(\sigma_{\infty})_{\mid 2(k+j)} = (\sigma)_{\mid 2(k+j)} &\text{ for } j > i.    
        \end{aligned}\right.
    \end{align*}

    It is clear that these are well-defined and satisfy the first three conditions of definition~\ref{covering}. It remains to prove to condition $4$ is satisfied. 

    Let $x_{i+k} \in [\phi_{\infty, i+k}(\sigma_{\infty})]$ be the sequences witnessing condition $4$ for the coverings $(T_{i+1}, \pi_{i+1}, \phi_{i+1})$, together with the commutativity of the diagram of figure~\ref{InvLimit}. By construction, $(x_{i + k})_{k < \mathbb{N}}$ converges to some $x_{\infty}$. By definition of the $\phi_{\infty,i}$'s, we have thus $x_{\infty} \in [\sigma_{\infty}]$ and clearly, $\pi_{\infty, i}(x_{\infty}) = x_i$.
\end{proof}


\begin{bibdiv}
    
    
    \begin{biblist}
        
        \bib{AgMSD}{article}{
            author={Aguilera, J. P.},
            title={The Metamathematics of Separated Determinacy},
            note={preprint},
            date={2023},
        }
        
       \bib{AgWe}{article}{
            author={Aguilera, J. P.},
            author={Welch, P. D.},
            title={Determinacy on the Edge of Second-Order Arithmetic},
            note={preprint},
            date={2023},
        }
        
        \bib{Barwise}{book}{
            author={Barwise, Jon},
            title={Admissible Sets and Structures: An Approach to Definability Theory},
            series={Perspectives in Logic},
            publisher={Springer},
            place={Durham},
            date={2009},
        }

        \bib{Devlin}{book}{
            author={Devlin K.},
            title={Constructibility},
            series={Perspectives in Mathematical Logic},
            publisher={Springer Verlag},
            place={Berlin, Heidelberg},
            date={1984},
        }

        \bib{Sigma5Det}{article}{
            author={Friedman, Harvey M.},
            title={Higher Set Theory and Mathematical Practice},
            journal={Annals of Mathematical Logic},
            volume={2},
            date={January 1971},
            number={3},
            pages={325--357},
        }

        \bib{Hachtman}{article}{
            author={Hachtman, Sherwood J.},
            title={Calibrating Determinacy Strength Levels of the Borel Hierarchy},
            journal={The Journal of Symbolic Logic},
            volume={82},
            date={June 2017},
            number={2},
            pages={510--548},
            doi={https://doi.org/10.1017/jsl.2017.15}
        }

        \bib{Z3Det}{article}{
            author={Hachtman, Sherwood J.},
            title={Determinacy in third order arithmetic},
            journal={Annals of Pure and Applied Logic},
            volume={168},
            date={November 2017},
            number={11},
            pages={2008--2021},
            doi={https://doi.org/10.1016/j.apal.2017.05.004}
        }   
        
      \bib{Ma01}{article}{
            author={Mathias, A. R. D},
            title={The strength of Mac Lane set theory},
            journal={Annals of Pure and Applied Logic},
            volume={110},
            date={2001},
            pages={107-234},
        }           

        \bib{Jensen}{article}{
            author={Jensen R.},
            title={The fine structure of the constructible hierarchy},
            journal={Annals of Mathematical Logic},
            volume={4},
            date={1972},
            pages={229--308},
            doi={https://doi.org/10.1016/0003-4843(72)90001-0},
        }

        \bib{Kechris}{book}{
            author={Kechris, Alexander S.},
            title={Classical Descriptive Set Theory},
            series={Graduate Texts in Mathematics},
            publisher={Springer-Verlag},
            place={Budapest},
            date={1995}
        }
        

        \bib{KouptchinskyMSc}{book}{
            author={Kouptchinsky, Thibaut},
            title={On the Limits of Determinacy in Third-Order Arithmetic and Extensions of
Kripke-Platek Set Theory},
            note={MSc Thesis, UCLouvain},
            date={2023},
        }
        
        \bib{Kuratowski}{book}{
            author={Kuratowski, Kasimir},
            title={Topologie I},
            series={Pa\'nstwowe Wydawnidwo Naukowe},
            place={Warsaw},
            date={1958},
        }

        \bib{History}{article}{
            author={B. Larson, Paul},
            title={A Brief History of Determinacy},
            journal={Large Cardinals, Determinacy and Other Topics, The Cabal Seminar},
            volume={IV},
            date={November 2020},
            pages={3--60},
            doi={https://doi.org/10.1017/9781316863534.002}
        }
        
        \bib{BorelDetMartin}{article}{
            author={Martin, Donald A.},
            title={Borel determinacy},
            journal={Annals of Mathematics},
            volume={102},
            date={September 1975},
            number={2},
            pages={363--371},
            doi={https://doi.org/10.2307/1971035},
        }

        \bib{Martin}{book}{
            author={Martin, Donald A.},
            title={Determinacy of Infinitely Long Games},
            status={Draft (September 2018)},
            note={\url{https://www.math.ucla.edu/~dam/booketc/D.A._Martin,_Determinacy_of_Infinitely_Long_Games.pdf}}
            }

        \bib{MS}{article}{
            author={Montalbán, Antonio},
            author={Shore, Richard A.},
            title={The limits of determinacy in second order arithmetic},
            journal={Proceedings of the London Mathematical Society},
            volume={104},
            date={July 2011},
            number={2},
            pages={223--252},
            doi={https://doi.org/10.1112/plms/pdr022}
            }

        \bib{MSCons}{article}{
            author={Montalbán, Antonio},
            author={Shore, Richard A.},
            title={The limits of determinacy in second order arithmetic: consistency and complexity strength},
            journal={Israel Journal of Mathematics},
            volume={204},
            date={July 2014},
            pages={477--508},
            doi={https://doi.org/10.1007/s11856--014--1117--9}
            }

        \bib{CantorDet}{article}{
            author={Nemoto, {Ta}kako},
            author={{Ta}naka, {Ka}zuyuki},
            author={Ould MedSalem, Med Yahya},
            title={Infinite games in the Cantor space and subsystems of second order arithmetic},
            journal={Mathematical Logic Quarterly},
            volume={53},
            date={June 2007},
            number={3},
            pages={226--236},
            doi={https://doi.org/10.1002/malq.200610041},
        }
        
        \bib{Schw15}{article}{
author={Schweber, Noah},
pages={940--969},
journal={Journal of Symbolic Logic},
volume={80},
title={Transfinite Recursion in Higher Reverse Mathematics},
date={2015},
}

        \bib{pacheco}{article}{
            author={Pacheco, Leonardo},
            author={{Yo}koyama, {Ke}ita},
            title={Determinacy and reflection principles in second-order arithmetic},
            journal={arXiv},
            date={September 2022},
            doi={https://doi.org/10.48550/arXiv.2209.04082}
        }

        \bib{Simpson}{book}{
            author={Simpson, Stephen G.},
            title={Subsystems of Second Order Arithmetic},
            edition={2nd ed.},
            series={Perspectives in Logic},
            publisher={Cambridge University Press},
            place={Delhi},
            date={2009}
        }

        \bib{Pi03CA>>Pi03Det}{article}{
            author={Welch, Philip D.},
            title={Weak Systems of Determinacy and Arithmetical Quasi-Inductive Definitions},
            journal={The Journal of Symbolic Logic},
            volume={76},
            date={June 2011},
            number={2},
            pages={418--436},
            doi={https://doi.org/10.2178/JSL/1305810756},
        }

    \end{biblist}
    
\end{bibdiv}
\end{document}